\documentclass[twoside,16pt]{elsarticle}
\usepackage{amssymb,amsmath,natbib,graphicx,lineno,setspace,amsthm,float}
\usepackage{setspace}
\usepackage{rotating}
\usepackage[ruled]{algorithm2e}
\usepackage{threeparttable}
\biboptions{numbers,sort&compress}

\usepackage[font=footnotesize]{caption}
\usepackage{booktabs}
\usepackage{nccmath}
\newcommand*{\boldone}{\text{\usefont{U}{bbold}{m}{n}1}}

\newtheorem{theorem}{Theorem}[section]
\newtheorem{definition}{Definition}[section]
\newtheorem{lemma}{Lemma}[section]
\newtheorem{remark}{Remark}[section]

\newtheorem{corollary}{Corollary}[section]
\newtheorem{proposition}{Proposition}[section]

\usepackage[CJKbookmarks,colorlinks,driverfallback=dvipdfm,linkcolor=red]{hyperref}
\usepackage{mathrsfs}
\journal{Journal of Differential Equations}

\numberwithin{equation}{section}
\numberwithin{theorem}{section}
\numberwithin{lemma}{section}
\numberwithin{definition}{section}
\numberwithin{example}{section}
\numberwithin{assumption}{section}
\usepackage{amsopn}

\makeatletter
\renewenvironment{proof}[1][\proofname]{\par
  \pushQED{\qed}%
  \normalfont \topsep6\p@\@plus6\p@\relax
  \trivlist
  \item[\hskip\labelsep
        \bfseries
    #1\@addpunct{.}]\ignorespaces
}{%
  \popQED\endtrivlist\@endpefalse
}
\makeatother
\renewcommand{\proofname}{Proof}
\usepackage[utf8]{inputenc}

\begin{document}

\begin{frontmatter}
    \title{Existence and asymptotic stability of a generic Lotka-Volterra system with nonlinear spatially heterogeneous cross-diffusion}
    \author[a,b] {Tianxu Wang}
    \author[a] {Jiwoon Sim}
    \author[a,b] {Hao Wang\corref{cor1}}
	\ead{hao8@ualberta.ca}
	\cortext[cor1]{Corresponding author.}
  \address[a] {Department of Mathematical and Statistical Sciences, University of Alberta, Edmonton, Alberta T6G 2G1, Canada}
        \address[b] {Interdisciplinary Lab for Mathematical Ecology \& Epidemiology, University of Alberta, Edmonton, Alberta T6G 2G1, Canada}
       
    \renewcommand{\abstractname}{Abstract (Part 1)}    
    \begin{abstract}
        This article considers a class of Lotka-Volterra systems with multiple nonlinear cross-diffusion, commonly known as prey-taxis models. The existence and stability of classic solutions for such systems with spatially homogeneous sources and taxis have been studied in one- or two-dimensional space, however, the proof is non-trivial for a more general setting with spatially heterogeneous predation functions and taxis coefficient functions in arbitrary dimensions. This study introduces a new weighted \(L_\epsilon^p\)-norm and extends some classical inequalities within this normed space. Coupled energy estimates are employed to establish initial bounds, followed by applying heat kernel properties and an advanced bootstrap process to enhance solution regularity. For stability analysis, we extend LaSalle's invariance principle to a general \( L^\infty \) setting and utilize it alongside Lyapunov functions to analyze the stability of each possible constant equilibrium. All results are achieved without introducing an extra logistic growth term for predators or imposing smallness conditions on taxis coefficients.


    \end{abstract}

\end{frontmatter}


\section{Introduction}

In predator-prey interactions, alarm pheromones are crucial in communication between different species. Predators are attracted by signals generated by prey, which may manifest as chemical, acoustic, visual, or other detectable changes \cite{chivers1996evolution}. In ecosystems, non-random foraging strategies are often observed, with populations of predators moving towards areas with rapidly increasing prey density, guided by these signals \cite{braucker1994graviresponses}. Such movement is commonly referred to as ``prey-taxis". 

Kareiva and Odell \cite{kareiva1987swarms} first proposed the following prototypical Lotka-Volterra system with prey-taxis: 
\begin{equation}
    \label{eq: introduction model 1}
    \begin{aligned}
        &\partial_t X = r(X) - f(X,Y) + \delta_X \Delta X, && (x, t) \in \Omega \times (0, \infty), \\
        &\partial_t Y = e f(X,Y) - d Y -\nabla \left( H(X, Y) \nabla X \right) + \delta_Y \Delta Y, && (x, t) \in \Omega \times (0, \infty), \\
        &\partial_\nu X(x, t) = \partial_\nu Y(x, t)=0, && (x, t) \in \partial\Omega \times (0, \infty), \\
        &(X(x, 0), Y(x, 0)) = (X_0(x), Y_{0}(x)), && x \in \Omega,
    \end{aligned}
\end{equation}
where $X$ and $Y$ represent the population density of prey and predators, respectively. This model has been widely studied in recent years.
Here, $\Omega \subset \mathbb{R}^n$ is a bounded domain, and $\partial_\nu$ is the outer normal derivative. The function $r(X)$ denotes the growth function for prey, $f(X, Y)$ represents predation on prey, and $H(X, Y)$ characterizes attraction towards food sources or the inclination towards prey for predators. The parameter $e$ is the production efficiency, $\delta_X$, $\delta_Y$ are diffusion coefficients, and $d$ is the death rate.

The global existence of solutions to \eqref{eq: introduction model 1} has been studied under different assumptions regarding the taxis sensitivity function $H$. For example, when $H(X, Y) = h(X)$ and $h$ has bounded support, \cite{ainseba2008reaction} established the existence and uniqueness of weak solutions using the Schauder fixed point theorem and a duality technique (see also \cite{bendahmane2008analysis} for multi-species case). Classical solutions were further considered in \cite{tao2010global} for $n \leq 3$ via the contraction mapping principle, alongside $L^p$ and Schauder estimates. On the other hand, the form $H(X, Y) = \chi h(Y)$ is also widely used, where $\chi$ is a constant and $h(Y) \leq CY$. With additional logistic growth terms for predators, Wang and Wang \cite{wang2019global} established global existence for $n = 2$. When $n \geq 3$, the result holds as long as the carrying capacity for predators and the taxis coefficient $\chi$ are sufficiently small. Wu et al. \cite{wu2016global} proved global existence and uniform persistence in arbitrary dimensions under a smallness condition on the taxis coefficient. Jin and Wang \cite{jin2017global} achieved global existence and stability of the coexistence steady state without requiring additional logistic growth terms or imposing the aforementioned smallness condition. However, their findings were limited to the case where $n = 2$.

As we have mentioned, the global existence of solutions under $H(X, Y) = \chi Y$ in \eqref{eq: introduction model 1} needs extra assumptions, such as introducing logistic growth term on predators, imposing smallness condition on parameters or restricting the dimension $n$. Without these assumptions, blow-up behavior does occur. Existing results along this line mainly focus on the classical Keller-Segel system, which shares the same cross-diffusion term with system \eqref{eq: introduction model 1}:
\begin{equation}
   \label{eq: introdution model 2}
   \begin{alignedat}{2}
        &\partial_t X = \Delta X - X + Y , && (x, t) \in \Omega \times (0, \infty), \\
        &\partial_t Y = \Delta Y - \nabla \left( \chi Y \nabla X \right), \qquad && (x, t) \in \Omega \times (0, \infty).
    \end{alignedat}
\end{equation}
Nagai \cite{nagai1995blowup} proved the existence of blow-up solutions in two dimensions when $||Y_0||_{L^1(\Omega)}$ is above a threshold. This is also supported in \cite{jin2016boundedness} and generalized into higher dimensions in \cite{winkler2010aggregation}. To prevent the occurrence of blow-up solutions, sublinear taxis sensitivity functions are considered, namely, the term $\chi Y$ is replaced by $\chi h(Y)$ in \eqref{eq: introdution model 2}, where $h(z) = O(z^\alpha)$ as $z \to \infty$ and $\alpha \leq 1$. Horstmann and Winkler \cite{horstmann2005boundedness} proved the global existence of the solution when $\alpha < \frac{2}{n}$, and for the case $\alpha > \frac{2}{n}$, blow-up solutions were found with $L^1$ bounded initial values. When the initial values are allowed to be $C^1$ functions, the global existence was established when $\alpha < \frac{4}{n+2}$ \cite{tao2013global}. It remains open whether blow-up solutions exist for $\alpha > \frac{4}{n+2}$.

Typically, a three-compartment prey-taxis model is more challenging compared to a one-prey, one-predator model. Guo et al. \cite{guo2019dynamics} studied a relatively simple case where two predators consume a single prey species, with the taxis term depending solely on the gradient of prey. They used semigroup theory to investigate the global existence and boundedness of solutions under the condition of small taxis coefficients. There are also numerous studies focusing on the chemical signal model with one prey taxis term \cite{ahn2020global,li2023positive,geng2024double}.

A more commonly used three-compartment model in ecology is the trophic cascade model which can be seen as an extension of the two-species Lotka-Volterra system described in \eqref{eq: introduction model 1}. Compared to the prey-taxis systems explored in previous studies, the trophic cascade model exhibits more complex coupling structures. 
The basic predator-prey dynamics across three trophic levels with a general setting where the predation strategy is allowed to connect with spatial heterogeneity, can be described by the following equations on $\Omega\times (0, \infty)$:

\begin{equation}
    \label{eq: general model}
    \begin{aligned}
        &\partial_t X = r(X) - f(X,Y,x)+\delta_X \Delta X,\\
        &\partial_t Y = e_1 f(X,Y,x) - g(Y,Z,x) - d_1 Y - \nabla \left(\chi_1(x) h(Y)\nabla X\right) +\delta_{Y}\Delta Y,\\
        &\partial_t Z = e_2 g(Y,Z,x) - d_2 Z - \nabla (\chi_2(x) h(Z) \nabla Y) + \delta_Z \Delta Z,
    \end{aligned}
\end{equation}
with boundary and initial conditions
\begin{equation}
    \label{eq: bc ic}
    \begin{aligned}
        &\partial_\nu X(x,t) = \partial_\nu Y(x,t) = \partial_\nu Z(x,t) = 0, \qquad (x, t) \in \partial\Omega \times (0, \infty), \\
        &(X(x, 0), Y(x, 0), Z(x, 0)) = (X_0(x), Y_{0}(x), Z_0(x)), \qquad x \in \Omega.
    \end{aligned}
\end{equation}
Here, $X$, $Y$, and $Z$ represent population densities for producers, consumers, and predators. A typical example in marine ecosystems is a trophic cascade involving phytoplankton, zooplankton, and fish species. 

The region $\Omega \subset \mathbb{R}^n$ is a bounded set with a smooth boundary, representing the living environment for all three species.
$d_1$ and $d_2$ represent death rates for consumers and predators, respectively. $\delta_i$ represents diffusion rates for the three species, and $e_i$ denotes maximal production efficiencies. Logistic growth is commonly used to model the growth of producers, namely,
\begin{equation*}
    r(X) = rX \left( 1 - \frac{X}{K} \right),
\end{equation*}
where $r$ is the growth rate and $K$ is the maximal carrying capacity for producers. The taxis sensitivity function $h(z)$ is considered to take sublinear growth when $z \to \infty$ in this paper.

In model \eqref{eq: general model}, the predation functions $f$, $g$, and the taxis coefficients $\chi_i$ depend on the space variable $x$. Previous studies have all assumed that the taxis coefficients are constants. However, the resource distribution exhibits high heterogeneity in space. Predation behaviors vary according to spatial influences. For example, within sheltered zones where consumers can evade predators, predators have fewer opportunities and less motivation to search for prey. Therefore, it is necessary to model the taxis coefficient as a function $\chi_i(x)$ rather than a constant. 

The functions $f(X, Y,x)$ and $g(Y, Z,x)$ represent predation on producers by consumers and predation on consumers by predators, respectively. These predation functions are inversely proportional to handling time. Since handling time varies significantly across the region, these functions should also take spatial heterogeneity into consideration. Here, $\tau_1(x)$ indicates the handling time for consumers on producers, and $\tau_2(x)$ represents the handling time for predators on consumers. In this study, $f$ and $g$ are considered as Holling type I and II functions with the following forms
\begin{equation*}
    f(X,Y,x)=\frac{\mu_1 X Y}{1 + \tau_1(x) X}, \quad \text{and} \quad g(Y,Z,x)=\frac{\mu_2 Y Z}{1 + \tau_2(x) Y}, 
\end{equation*}
where $\mu_1$ and $\mu_2$ denote encounter rates between producers and consumers, and between consumers and predators, respectively. $\tau_i(x)$ are allowed to be nonnegative; in particular, when $\tau_i(x) \equiv 0$, then $f$ or $g$ becomes Holling type I functional response. 

It is assumed that the following properties hold whenever we discuss system \eqref{eq: general model}.

\begin{itemize}
    \item[($A_I$)] The initial values $X_0$, $Y_0$, $Z_0$ satisfy 
    \begin{equation*}
         (X_0,Y_0, Z_0) \in [W^{2,\infty}(\Omega)]^3,\quad   X_0, Y_0, Z_0 \gneqq 0.
    \end{equation*}

    \item[($A_C$)] The parameters $r$, $K$, $e_1$, $e_2$, $d_1$, $d_2$, $\mu_1$, $\mu_2$, $\delta_X$, $\delta_{Y}$, $\delta_Z$ are positive constants.

    \item[($A_\chi$)] $\chi_i \in C^2 (\overline{\Omega}; (0, \infty))$, with maximum $\chi_{iM}$ and minimum $\chi_{im}$.
    
    \item[($A_\tau$)] $\tau_i \in C^1 (\overline{\Omega}; [0, \infty))$, with maximum $\tau_{iM}$.

    \item[($A_h$)] $h \in C^1 \big( [0, \infty); [0, \infty) \big)$, whose further properties are discussed below.
\end{itemize}

We slightly abuse the notation $f \sim g$ as $z \to z_0$ to indicate
\begin{equation*}
    0 < \liminf_{z \to z_0} \left| \frac{f(z)}{g(z)} \right| \leq \limsup_{z \to z_0} \left| \frac{f(z)}{g(z)} \right| < \infty.
\end{equation*}
Table \ref{table: h properties} lists the asymptotic behaviors of the function $h$ and some related functions. In addition, we require that the derivative $h'$ satisfies
\begin{equation*}
    h'(z) \geq 0 \ \text{and} \ 
    h'(z) = O(1) \ \text{as} \ z \to 0.
\end{equation*}
Notice that $H$ is increasing and $\mathcal{H}$ is a positive function. 

\renewcommand{\arraystretch}{1.5}
\begin{table}[t]
    \centering
    \begin{tabular}{|c|c|c|}
        \hline
        Definition & $z \to 0$ & $z \to \infty$ \\ \hline
        $h(z)$ & $\sim z^{\beta}$, $\beta \in [1, 2)$ & $\sim z^{\alpha}$, $\alpha \in [0, \frac{4}{n + 2}) \cap [0, 1]$ \\ \hline
        $H(z) = \int_1^z h^{-1} (s) \mathrm d s$ &
        $\begin{cases}
            \sim \ln z,      \ \beta = 1, \\
            \sim z^{1 - \beta}, \ \beta > 1.
        \end{cases}$ & 
        $\begin{cases}
            \sim z^{1 - \alpha}, \ \alpha < 1, \\
            \sim \ln z, \ \alpha = 1
        \end{cases}$ \\ \hline
        $\mathcal{H} (z) = \int_1^z H(s) \mathrm d s + 1$ & $\sim 1$ & 
        $\begin{cases}
            \sim z^{2 - \alpha}, \ \alpha < 1, \\
            \sim z \ln z, \ \alpha = 1
        \end{cases}$ \\ \hline
        $\widetilde{H} (z) = \int_0^z h^{-1/2} (s) \mathrm d s$ & $\sim z^{1 - \beta / 2}$ & $\sim z^{1 - \alpha / 2}$ \\ \hline
    \end{tabular}
    \caption{Properties of the function $h$ and some related functions}
    \label{table: h properties}
\end{table}
\renewcommand{\arraystretch}{1}

\begin{remark}
    For \( n \in \mathbb{N}\), typical examples of the function $h$ include constant functions $h(z) = c$, saturated type \( h(z) = \frac{z}{1 + \epsilon z^m} \) and Ricker type \( h(z) = z e^{-\epsilon z} \), where $c > 0$, $m > \max \left\{ \frac{n - 2}{n + 2}, 0 \right\}$ and $\epsilon >0$.
\end{remark}

The analysis of existence and stability in spatial trophic cascade models presents significant challenges. To establish the $L^\infty$-bound of $Z$, it is necessary to obtain an a priori bound for $||\nabla Y||_{L^\infty}$. However, the estimation of this bound depends on both $Y$ itself and $X$. There are only a few results considering this model with spacially homogeneous predation functions $f$, $g$, and taxis coefficients $\chi_i$. Jin \cite{jin2022global} proved global existence and stability of solutions, in which they took $n = 2$, $h(Y) = Y$ and Holling Type I predation function. These findings were further supported by Jin and Wang \cite{jin2023global}, who extended the analysis to encompass more general predation and taxis strength functions. However, this latter study introduced an additional logistic growth term for predators.

In this article, we explore the global existence and stability of solutions to system \eqref{eq: general model} -- \eqref{eq: bc ic}. Most existing work on prey taxis models is limited to spaces of fewer than three dimensions. This study extends the current literature by considering arbitrary dimensions and incorporating space-dependent predation functions and taxis coefficients, notably, all without introducing a logistic growth term for predators or imposing a smallness condition on taxis coefficients. The proof utilizes coupled energy estimates to provide initial bounds on $X$, $Y$, and $Z$, together with maximal regularity property \cite{matthias1997heat} and the $L^p$-$L^q$ estimates for heat kernel {\cite{winkler2010aggregation}} to improve the integrability and regularity of the solution. In contrast, in \cite{wang2019global,jin2023global}, where all coefficients are homogeneous in space, the quadratic decay terms in the equation for $Y$ provide a bound for $||Y||_{L^2 (\Omega) L^2 (t_1, t_2)}$. However, their method cannot be directly applied to our case. Our situation requires more complex energy estimates to establish an initial bound on $Y$ and necessitates careful handling of the additional terms introduced by the space dependency of $\chi$, $f$, and $g$. Since the space dimension is considered to be arbitrary, we adopt a bootstrap argument to raise the integrability of the solution until $L^\infty$. To achieve this, we extend some classic inequalities under a newly introduced weighted $L^p$-norm, which plays a crucial role in obtaining the global-in-time boundedness of the solution.

Our main result on the global existence and boundedness of classical solutions of system \eqref{eq: general model} -- \eqref{eq: bc ic} is stated as follows:
\begin{theorem}[Global boundedness]
    \label{th: global existence}
    Let $\Omega\subset \mathbb{R}^n$ $(n\geq 1)$ be a bounded domain with smooth boundary. Under Assumptions $(A_I)$ -- $(A_h)$, equation \eqref{eq: general model} -- \eqref{eq: bc ic} has a unique maximal solution $U = (X, Y, Z) \in C \big( \overline{\Omega} \times [0, \infty); \mathbb R_{\geq 0}^3 \big) \cap C^{2, 1} \big( \overline{\Omega} \times (0, \infty); \mathbb R_+^3 \big)$. Furthermore, there exists a constant $C>0$ independent of t such that 
    \begin{equation*}
        \|X(\cdot, t)\|_{W^{1,\infty}} + \|Y(\cdot, t)\|_{W^{1,\infty}} + \|Z(\cdot, t)\|_{L^\infty} \leq C, \qquad \forall t \in \left[0, \infty \right).
    \end{equation*}
\end{theorem}

In this study, we also conduct a stability analysis of the system \eqref{eq: general model}. When $\tau_1$ and $\tau_2$ are constant functions, this system has four types of homogeneous (constant) steady states $U^*$.

\begin{itemize}
    \item Extinction steady state: \((0,0,0)\);
    \item Prey-only steady state: \((K,0,0)\);
    \item Semi-coexistence steady state: \((X^*, Y^*, 0)\), where \(X^* = \frac{d_1}{\mu_1 e_1 - \tau_1 d_1}\) and \(Y^* = \frac{e_1 r}{d_1} X^* \left(1 - \frac{X^*}{K}\right)\);
    \item Coexistence steady states: \((X_1^*, Y_1^*, Z_1^*)\) and \((X_2^*, Y_2^*, Z_2^*)\), where:
        \begin{itemize}
            \item \(Y_1^* = Y_2^* = \frac{d_2}{\mu_2 e_2 - \tau_2 d_2}\),
            \item \(X_1^*\) and \(X_2^*\) are the roots of the equation \(\frac{\tau_1}{K} X^2 + \left(\frac{1}{K} - \tau_1\right) X + \frac{\mu_1}{r} Y^* - 1 = 0\),
            \item \(Z_i^* = \frac{r e_1 e_2}{d_2} X_i^* \left(1 - \frac{X_i^*}{K}\right) - \frac{e_2 d_1}{\mu_2 e_2 - \tau_2 d_2}\), \ for \(i = 1, 2\).
        \end{itemize}
\end{itemize}

Most existing work on prey-taxis systems \cite{jin2017global,jin2022global,jin2023global,ahn2020global} employs Lyapunov functions to find some conditions for global stability of constant equilibria. However, the general system \eqref{eq: general model} may exhibit multi-stability, which represents a common ecosystem structure in population dynamics. Our work extends these analyses by considering all scenarios of instability, local stability, and global stability. For the trivial and semi-trivial equilibria, these three cases of stability are discussed under different conditions. The need for a general treatment of local and global stability motivate us to generalize the LaSalle's invariance principle to the $L^\infty$ setting, given in Proposition \ref{prop: lyapunov to L^infty}. Supported by its applications in Lemmas \ref{lem: stability for prey only} to \ref{lem: stability of coexistence}, this turns out to be a powerful tool to establish the stability of equilibria.

Now we state the main result on stability. For the exact definition of each type of stability, please refer to Lemmas \ref{lem: stability for prey only}, \ref{lem: stability of semicoexistence},  \ref{lem: stability of coexistence} and Remark \ref{rem: LaSalle}.

\begin{theorem}[Stability]
\label{th: stability}
    Let $\Omega\subset \mathbb{R}^n$ $(n\geq 1)$ be a bounded domain with smooth boundary, and $U(t)$ be the solution to system \eqref{eq: general model} -- \eqref{eq: bc ic}. Suppose Assumptions $(A_I)$ -- $(A_h)$ hold and $\tau_1$, $\tau_2$ are constant functions. 
    \begin{itemize}
    \item[(i)] Extinction steady state: 
    $(0,0,0)$ is unstable.
    
    \item[(ii)] Prey-only steady state: 
    \begin{enumerate}
    \item[(a)]  If \(d_1 > K \left(\mu_1 e_1 - d_1 \tau_1\right)\), then $(K,0,0)$ is locally asymptotically stable.
    \item[(b)] If \(d_1 \geq K \mu_1 e_1 \), then $(K,0,0)$ is globally asymptotically stable.
    \item[(c)]  If 
    \(d_1 \leq K \left(\mu_1 e_1 - d_1 \tau_1\right)\),
    then $(K,0,0)$ is unstable.
    \end{enumerate}

    \item[(iii)] Semi-coexistence steady state: Assume \( Y^* < \frac{d_2}{e_2 \mu_2} \) and $ \left\|\frac{\chi_1(x)h(Y) X}{Y}\right\|_{L^\infty}^2 < \frac{4 e_1 \delta_X \delta_Y X^*}{\left(1+\tau_1 X^*\right) Y^*}$.
\begin{enumerate}
    \item[(a)]  If \( X^* > \frac{1}{2}\left(K - \frac{1}{\tau_1}\right) \), then $(X^*,Y^*,0)$ is locally asymptotically stable.
    
    \item[(b)] If \( X^* > K - \frac{1}{\tau_1} \), then $(X^*,Y^*,0)$ is globally asymptotically stable.
    \item[(c)]  If 
    \( X^* \leq \frac{1}{2}\left(K - \frac{1}{\tau_1}\right)\),
    then \((X^*,Y^*,0)\) is unstable.
\end{enumerate}
    \item[(iv)] Coexistence steady state: Let $\Phi$ be the right hand side of system \eqref{eq: general model}, given in equation \eqref{eq: vector Phi}.
    \begin{itemize}
        \item[(a)] If 
        \begin{equation*}
            \frac{1}{\delta_X} \chi_{1M}^2 h^2(Y^*) + \frac{1}{\delta_Z} \chi_{2M}^2 h^2(Z^*) < 4 \delta_Y
        \end{equation*}
        and $D\Phi (U^*)$ is negative definite, then the equilibrium $U^*$ is locally asymptotically stable.
        \item[(b)] If $D\Phi (U^*)$ has an eigenvalue with positive real part, then the equilibrium $U^*$ is unstable.
        \end{itemize}
    \end{itemize}
\end{theorem}

\begin{remark}
    It follows from Table \ref{table: h properties} that \( \left\|\frac{h(Y)}{Y}\right\|_{L^\infty} \) is bounded. We will prove in Lemma \ref{lem: X bounded} that \(\|X\|_{L^\infty} \leq \max\{ \|X_0\|_{L^\infty}, K\}\). Therefore, the condition for the semi-coexistence steady state can be satisfied when $\chi_{1M}$ is sufficiently small.
\end{remark}

The remainder of this article is structured as follows. Section \ref{sec: Preliminaries} covers some preliminaries, including the definition and properties of the weighted $L^p$ norm, and some useful inequalities. Section \ref{sec: Local} derives the local existence and positivity of the solution. In Section \ref{sec: Global existence}, we establish global existence. Finally, in Section \ref{sec: Stability}, we further improve the regularity of solutions and use Lyapunov functions to prove the stability.

\section{Preliminaries}
\label{sec: Preliminaries}
We begin this section by proving various properties related to the function $h$. Throughout this article, we abbreviate the norm $||\cdot||_{L^p(\Omega)}$ and $||\cdot||_{W^{k, p}(\Omega)}$ into $||\cdot||_{L^p}$ and $||\cdot||_{W^{k, p}}$. We also drop $\mathrm{d}x$ when we perform integration in $\Omega$. 

\begin{lemma}
    \label{lemma: norms of h(psi)}
    Let $p \in [1, \infty)$. For any measurable function $\phi$ on $\Omega$ and any $\nu > 0$, there exist constants $C$, $C'$ such that
    \begin{equation}
        \label{eq: norms of h(psi) 1}
        \begin{aligned}
            & \int_\Omega h(\phi)^p 
            \leq  C\int_\Omega |\phi|^{p \alpha}  + C , \\
            & \int_\Omega |h(\phi) H(\phi)|^p 
            \leq C \int_\Omega |\phi|^{p + \nu}  + C, \\
            & \int_\Omega |\phi H(\phi)|^p  \leq C  \int_\Omega |\phi|^{p(2 - \alpha)}  + C\int_\Omega |\phi|^{p + \nu}  + C, \\
            & C' \int_\Omega |\phi|^{p(2-\alpha)}  
            \leq \int_\Omega \mathcal{H} (\phi)^p  
            \leq C  \int_\Omega |\phi|^{p(2 - \alpha)}  + C \int_\Omega |\phi|^{p + \nu}  + C.
        \end{aligned}
    \end{equation}
\end{lemma}

\begin{proof}
    We prove the inequality on $\mathcal{H}(\phi)$ only; other inequalities can be proved by a similar argument. By Table \ref{table: h properties}, $\mathcal{H}(z) \sim 1$ as $ z\to 0$, hence on the set $\Omega_1=\{|\phi| < 1\}$, we have 
    \begin{equation*}
        \int_{\Omega_1} \mathcal{H}(\phi)^p \leq  c_1 |\Omega|,
        \end{equation*}
        and
    \begin{equation*}
        \int_{\Omega_1} \mathcal{H} (\phi)^p  
        \geq c_2 |\Omega_1|
        \geq c_2 \int_{\Omega_1} |\phi|^{p(2-\alpha)} .
    \end{equation*}
    On the other hand, on the set $\Omega_2 = \{ |\phi| \geq 1 \}$ we have
    \begin{equation*}
        \begin{aligned}
            \int_{\Omega_2} \mathcal{H}(\phi)^p  &\leq 
            \begin{cases}
            c_3 \int_{\Omega_2} |\phi|^{p(2-\alpha)} , \ \alpha < 1 \\
            c_3 \int_{\Omega_2} |\phi \ln \phi|^p , \ \alpha = 1
        \end{cases}
        \leq
        \begin{cases}
            c_3 \int_{\Omega_2} |\phi|^{p(2-\alpha)} , \ \alpha < 1, \\
            c_4 \int_{\Omega_2} |\phi| ^{p +\nu} , \ \alpha = 1,
        \end{cases}
        \end{aligned}
    \end{equation*}
    and
    \begin{equation*}
        \begin{aligned}
            \int_{\Omega_2} \mathcal{H} (\phi)^p  
            \geq c_5 \int_{\Omega_2} |\phi|^{p(2-\alpha)} .
        \end{aligned}
    \end{equation*}
    By summing above equations, one can readily verify the last inequality of \eqref{eq: norms of h(psi) 1} holds.
\end{proof}

\begin{lemma}
    \label{lemma: G-N inequality on psi}
    For any measurable function $\phi$ on $\Omega$ and any $\epsilon > 0$, there exists a constant $C$ dependent on $||\phi||_{L^1 (\Omega)}$ such that
    \begin{equation*}
        \int_\Omega \left(\phi |H(\phi)| + |h(\phi) H(\phi)|^{4/3} + \mathcal{H} (\phi) \right) 
        \leq \epsilon \int_\Omega \frac{|\nabla \phi|^2}{h(\phi)}  + C.
    \end{equation*}
\end{lemma}

\begin{proof}
    We prove this lemma under the condition
    \begin{equation*}
        \alpha \in \Big[ 0, \frac{2}{n} + \frac{2}{3} \Big) \cap [0, 1],
    \end{equation*}
    which contains the current range $[0, \frac{4}{n + 2}) \cap [0, 1]$ of $\alpha$. By Lemma \ref{lemma: norms of h(psi)}, for any $\nu > 0$ we have
    \begin{equation*}
        \int_\Omega \left(\phi |H(\phi)| + |h(\phi) H(\phi)|^{4/3} + \mathcal{H} (\phi)\right)  
        \leq ||\phi||_{L^{2 - \alpha}}^{2 - \alpha} + ||\phi||_{L^{4/3 + \nu}}^{4/3 + \nu} + c_1.
    \end{equation*}
    Let $\Omega_1 = \{ x \in \Omega: |\phi(x)| \geq 1 \}$, then
    \begin{equation*}
        \int_\Omega \left(\phi |H(\phi)| + |h(\phi) H(\phi)|^{4/3} + \mathcal{H} (\phi) \right) 
        \leq ||\phi||_{L^{2 - \alpha} (\Omega_1)}^{2 - \alpha} + ||\phi||_{L^{4/3 + \nu} (\Omega_1)}^{4/3 + \nu} + c_1 + 2|\Omega|.
    \end{equation*}
    We claim that for any $\epsilon > 0$ and $q \in [1, 2 - \alpha] \cup [1, \frac{4}{3} + \nu_0)$, where $\nu_0 = \frac{2}{3} + \frac{2}{n} - \alpha > 0$, we have 
    \begin{equation*}
        \label{eq: proof 1 G-N inequality on psi}
        ||\phi||_{L^q (\Omega_1)}^q \leq \epsilon \int_\Omega \frac{|\nabla \phi|^2}{h(\phi)}  + C,
    \end{equation*}
    where $C$ depends on $||\phi||_{L^1 (\Omega)}$. Based on the properties of $\widetilde{H}$ in Table \ref{table: h properties}, we rewrite equation \eqref{eq: proof 1 G-N inequality on psi} into
    \begin{equation*}
        \label{eq: proof 2 G-N inequality on psi}
        ||\widetilde{H} (\phi)||_{L^\frac{2q}{2 - \alpha} (\Omega_1)}^\frac{2q}{2 - \alpha} 
        \leq \epsilon ||\nabla \widetilde{H} (\phi)||_{L^2(\Omega)}^2 + C,
    \end{equation*}
    where $C$ depends on $||\widetilde{H} (\phi)||_{L^\frac{2}{2 - \alpha} (\Omega)}$.
    
    We proceed via Gagliardo-Nirenberg inequality. Take
    \begin{equation*}
        \theta = \frac{n (q - 1) (2 - \alpha)}{2q + nq (1 - \alpha)} \geq 0,
    \end{equation*}
    then $\theta < 1$ is equivalent to 
    \begin{equation*}
        \left( 1 - \frac{2}{n} \right) q < 2 - \alpha.
    \end{equation*}
    The above inequality holds when $n \leq 2$ or when $1 \leq q \leq 2 - \alpha$. When $n \geq 3$, we have
    \begin{equation*}
        q < \frac{2 - \alpha}{1 - \frac{2}{n}}
        = \frac{4}{3} + \frac{\frac{2}{3} + \frac{2}{n} - \alpha}{1 - \frac{2}{n}} + \frac{2}{3n - 6}.
    \end{equation*}
    We observe that the above inequality holds for $1 \leq q < \frac{4}{3} + \nu_0$. Now the Gagliardo-Nirenberg inequality implies
    \begin{equation*}
        ||\widetilde{H} (\phi)||_{L^\frac{2q}{2 - \alpha} (\Omega_1)}^\frac{2q}{2 - \alpha} 
        \leq c_2 ||\nabla \widetilde{H} (\phi)||_{L^2 (\Omega_1)}^{\frac{2q}{2 - \alpha} \theta} ||\widetilde{H} (\phi)||_{L^\frac{2}{2 - \alpha} (\Omega_1)}^{\frac{2q}{2 - \alpha} (1 - \theta)} + c_2 ||\widetilde{H} (\phi)||_{L^{\frac{2}{2 - \alpha}} (\Omega_1)}^{\frac{2q}{2 - \alpha}}.
    \end{equation*}
    By direct calculation, the relation
    \begin{equation*}
        \frac{2q}{2 - \alpha} \theta = \frac{2n (q - 1)}{2 + n (1 - \alpha)} < 2
    \end{equation*}
    is equivalent to
    \begin{equation*}
        q < \frac{2}{n} + (2 - \alpha) = \frac{4}{3} + \left( \frac{2}{3} + \frac{2}{n} - \alpha \right).
    \end{equation*}
    This is true whenever $q \in [1, 2 - \alpha] \cup [1, \frac{4}{3} + \nu_0)$. Therefore, by Young's inequality, there exists $\theta' > 0$ such that
    \begin{equation*}
        ||\widetilde{H} (\phi)||_{L^\frac{2q}{2 - \alpha} (\Omega_1)}^\frac{2q}{2 - \alpha} 
        \leq \epsilon ||\nabla \widetilde{H} (\phi)||_{L^2 (\Omega_1)}^2 + c_3 ||\widetilde{H} (\phi)||_{L^\frac{2}{2 - \alpha} (\Omega_1)}^{\theta'} + c_2 ||\widetilde{H} (\phi)||_{L^{\frac{2}{2 - \alpha}} (\Omega_1)}^{\frac{2q}{2 - \alpha}}.
    \end{equation*}
    This implies equation \eqref{eq: proof 2 G-N inequality on psi}, and we conclude the proof of the lemma.
\end{proof}

Now we introduce a family of weighted $L^p$ norms, which appears naturally in the a priori estimates.

\begin{definition}
    \label{def: L^p_epsilon}
    For $p \in [1, \infty]$, $t \in [0, \infty)$ and $\epsilon > 0$, define $L^p_\epsilon (0, t)$ to be the Banach space consisting of all functions in $L^p (0, t)$ and equipped with the norm
    \begin{equation*}
        ||f||_{L^p_{\epsilon}(0,t)} := e^{-\epsilon t} ||e^{\epsilon \cdot} f||_{L^p (0,t)}.
    \end{equation*}
    For $t = \infty$, we define 
    \begin{equation*}
         ||f||_{L^p_{\epsilon} (0,\infty)}:= \limsup \limits_{t \to \infty} ||f||_{L^p_\epsilon (0,t) }
    \end{equation*}
   and set 
    \begin{equation*}
        L^p_\epsilon(0,\infty) = \left\{ f \in L^1_{\text{loc}} (0, \infty): ||f||_{L^p_\epsilon (0,\infty)} < \infty \right\}.
    \end{equation*}
\end{definition}

\begin{remark}
\label{rem: holder for L_p_epsilon}
It's easy to verify the H\"older's inequality for $L_\epsilon^p$ norm, i.e. 
    \begin{equation*}
        ||f g ||_{L^p_\epsilon(0,t)} \leq ||f||_{L^{p_1}_{\epsilon_1}(0,t)}||g||_{L^{p_2}_{\epsilon_2}(0,t)},
    \end{equation*}
    where $\epsilon=\epsilon_1 + \epsilon_2$, $\frac{1}{p} = \frac{1}{p_1} + \frac{1}{p_2}$, and $t \in [0 ,\infty]$.
\end{remark}
    
We now show the boundedness of a class of convolution operators on the space $L^p_\epsilon (0, t)$ based on Young's convolution inequality.

\begin{proposition}
    \label{prop: Young's inequality on L^p_epsilon}
    Define an operator $T_{\rho, \lambda}$ by
    \begin{equation}
        \label{eq: operator T_rho,lambda}
        T_{\rho, \lambda} \phi (t) = \int_0^t (1 + (t - s)^{-\rho}) e^{-\lambda (t - s)} \phi(s) \mathrm d s,
    \end{equation}
    where $\rho \in (0, 1)$ and $\lambda > 0$. Suppose
    \begin{equation}
        \label{eq: prop condition on p q}
        1 \leq p \leq q \leq \infty
        \quad \text{and } \quad
        \frac{1}{p} + \rho < \frac{1}{q} + 1,
    \end{equation}
    then for any $t \in [0, \infty]$ and $0 < \epsilon < \lambda$, $T_{\rho, \lambda}$ is a bounded operator from $L_\epsilon^p (0, t)$ to $L^q_\epsilon (0,t)$, and the bound is uniform in $t$.
\end{proposition}

\begin{proof}
    We first prove that the operator $T_{\rho, \lambda}$ is uniformly bounded from $L^p (0, t)$ to $L^q (0, t)$. Fix $t \in [0, \infty]$ and $\phi \in L^p (0, t)$. For $s \in \mathbb R$, define $\widetilde{\phi} (s) = \phi (s) \boldone_{(0, t)} (s)$ and $\widetilde{\psi} (s) = (1 + s^{-\rho}) e^{-\lambda s} \boldone_{(0, t)} (s)$. Then we have $T_{\rho, \lambda} \phi (s) = \widetilde{\phi} * \widetilde{\psi} (s)$ for $s \in (0, t)$. Let $r$ be a real number such that
    \begin{equation*}
        \frac{1}{p} + \frac{1}{r} = \frac{1}{q} + 1,
    \end{equation*}
    then from \eqref{eq: prop condition on p q} we know that $1 \leq r < \frac{1}{\rho}$. It follows that
    \begin{equation*}
    \begin{aligned}
        \| \widetilde{\psi} \|_{L^r (\mathbb R)}
        & \leq \| e^{-\lambda s} \|_{L^r (0, t)} + \| s^{-\rho} e^{-\lambda s} \|_{L^r (0, t)}
        \\
        &\leq (r \lambda)^{-\frac{1}{r}} + (r \lambda)^{\rho - \frac{1}{r}} \Gamma(1 - \rho r)^\frac{1}{r}
        =: c < \infty.
    \end{aligned}
    \end{equation*}
    By Young's convolution inequality, we get
    \begin{equation}
        \label{eq: Lp Lq boundedness of T_rho lambda}
        \| T_{\rho, \lambda} \phi \|_{L^q (0, t)}
        \leq \| \widetilde{\phi} * \widetilde{\psi} \|_{L^q (\mathbb R)}
        \leq \| \widetilde{\phi} \|_{L^p (\mathbb R)} \| \widetilde{\psi} \|_{L^r (\mathbb R)}
        \leq c \| \phi \|_{L^p (0, t)}.
    \end{equation}

    Now, suppose $t < \infty$ and $\phi \in L^p_\epsilon (0, t)$. By definitions of the operator $T_{\rho, \lambda}$ and $L^q_\epsilon$ norm, we have
    \begin{equation*}
    \begin{aligned}
        ||T_{\rho,\lambda} \phi||_{L^q_\epsilon (0,t)} 
        & = e^{-\epsilon t} \left\| e^{\epsilon r} \int_0^r (1 + (r-s)^{-\rho}) e^{-\lambda (r-s)} \phi(s) \mathrm{d}s \right\|_{L^q (0, t)} \\
        & = e^{-\epsilon t} \left\| \int_0^r (1 + (r-s)^{-\rho}) e^{-(\lambda - \epsilon) (r-s)} e^{\epsilon s} \phi(s) \mathrm{d}s \right\|_{L^q (0, t)} \\
        & = e^{-\epsilon t} ||T_{\rho, \lambda - \epsilon} (e^{\epsilon \cdot} \phi)||_{L^q (0,t)}.
    \end{aligned}
    \end{equation*}
    By \eqref{eq: Lp Lq boundedness of T_rho lambda}, we obtain that
    \begin{equation}
        \label{eq: t < infty, Young's inequality on L^p_epsilon}
        ||T_{\rho,\lambda} \phi||_{L^q_\epsilon (0,t)} 
        \leq c e^{-\epsilon t} ||e^{\epsilon \cdot} \phi||_{L^p (0,t)} 
        = c ||\phi||_{L^p_\epsilon(0,t)}.
    \end{equation}
    Finally, we take the upper limit of $t \to \infty$ in equation \eqref{eq: t < infty, Young's inequality on L^p_epsilon} and deduce that \eqref{eq: t < infty, Young's inequality on L^p_epsilon} is also true for $t = \infty$.
\end{proof}

Next, we prove some connections between the original $L^p$ norm and the new norm in this article.

\begin{proposition}
\label{prop: L epsilon infinity norm}
    For any $t \in [0, \infty]$ and $\epsilon > 0$, we have 
    \begin{equation}
    \label{eq: L epsilon infinity norm}
        ||\phi||_{L^\infty_\epsilon (0, t)} 
        \leq ||\phi||_{L^\infty (0, t)} 
        \leq \sup_{s \in (0, t)} ||\phi||_{L^\infty_\epsilon (0, s)}.
    \end{equation}
\end{proposition}

\begin{proof}
We first consider $t < \infty$. The first inequality is obvious as
\begin{equation*}
    \begin{aligned}
        ||\phi||_{L^\infty_\epsilon(0,t)} & =  \left\|e^{\epsilon (\cdot - t)} \phi \right\|_{L^\infty(0,t)} \leq \left\| \phi\right\|_{L^\infty(0,t)}.
    \end{aligned}
\end{equation*}
For any $0 < r \leq s < t$, we have
\begin{equation*}
    \begin{aligned}
     \left\| \phi \right\|_{L^\infty_\epsilon (0,s)} & = e^{-\epsilon s} \left\|e^{\epsilon \cdot} \phi \right\|_{L^\infty(0,s)} \\
     & \geq e^{-\epsilon s} \left\|e^{\epsilon \cdot} \phi \right\|_{L^\infty(s-r,s)}
     \\
     & \geq e^{-\epsilon r} \left\| \phi \right\|_{L^\infty(s-r,s)}.
    \end{aligned}
\end{equation*}
Choose $r = t / N$ for $N \in \mathbb N$, then we have
\begin{equation*}
\begin{aligned}
    \sup_{s \in (0, t)} ||\phi||_{L^\infty_\epsilon (0, s)}
    & \geq \sup_{1 \leq k \leq N} ||\phi||_{L^\infty_\epsilon (0, k r)} \\
    & \geq \sup_{1 \leq k \leq N} e^{-\epsilon r} \left\| \phi \right\|_{L^\infty ((k - 1) r, kr)}\\
    & = e^{-\epsilon r} \left\| \phi \right\|_{L^\infty (0, t)}.
    \end{aligned}
\end{equation*}
Now we let $N \to \infty$ to prove the second inequality in \eqref{eq: L epsilon infinity norm}. Finally, we take $t \to \infty$ to deduce that \eqref{eq: L epsilon infinity norm} is also true for $t = \infty$.
\end{proof}

\begin{proposition}
    \label{prop: comparison of different norms}
    Let $1 \leq p \leq q \leq \infty$, $t \in [0, \infty]$ and $\epsilon > 0$. Then there exists a constant $c$ independent of $t$ so that
    \begin{equation}
        \label{eq: comparison of different norms}
        ||\phi||_{L^p_\epsilon (0, t)} \leq c \left\| \phi \right\|_{L^q (0, t)}.
    \end{equation}
\end{proposition}

\begin{proof}
    Assume that $\phi \in L^q (0, t)$, otherwise the inequality is trivial.
    By H\"older's inequality,
    \begin{equation*}
        ||\phi||_{L^p_\epsilon (0, t)} 
        = ||e^{\epsilon (\cdot - t)} \phi||_{L^p (0,t)} 
        \leq ||e^{\epsilon (\cdot - t)}||_{L^\frac{pq}{q-p} (0,t)} ||\phi||_{L^q (0,t)} = c ||\phi||_{L^q (0,t)},
    \end{equation*}
    where
    \begin{equation*}
        c = ||e^{-\epsilon \cdot}||_{L^\frac{pq}{q-p} (0, \infty)} = \left(\frac{q-p}{\epsilon pq}  \right)^{\frac{q-p}{pq}}.
    \end{equation*}
    Finally we take $t \to \infty$ to deduce that \eqref{eq: comparison of different norms} is also true for $t = \infty$.
\end{proof}

\begin{remark}
    Fix $t < \infty$, then for any $1 \leq p \leq \infty$ and $\epsilon > 0$, the norm $\| \cdot \|_{L^p_\epsilon (0, t)}$ is equivalent to $\| \cdot \|_{L^p (0, t)}$. However, it is necessary to introduce the new norm since we are interested in the uniform boundedness of the norm for $t \in [0,\infty]$.
\end{remark}

Finally we gather some useful tools for our subsequent analysis.

\begin{lemma}{\cite{winkler2010aggregation}}
    \label{lem: L^p-L^q estimate}
    Let $\left\{ e^{td\Delta} \right\}_{t\geq 0} $ be the Newmann heat semigroup in $\Omega$, and let $\lambda>0$ denote the first nonzero eigenvalue of $-\Delta$ in $\Omega$ under Neumann boundary conditions, where $d$ is a positive constant. Then for all $t>0$, there exist some constant $C$ depending only on $\Omega$ such that 
    \begin{itemize}
        \item[(i)] If $2\leq p < \infty$, then
        \begin{equation}
        \label{eq: lemma Lp Lq estimate 1}
            ||\nabla e^{td\Delta}\phi||_{L^p (\Omega)} \leq C \left(1+t^{-\frac{n}{2}\left(\frac{1}{q}-\frac{1}{p}\right)} \right) e^{-d\lambda t}||\nabla \phi||_{L^q (\Omega)}
        \end{equation}
        for all $\phi \in W^{1,q}(\Omega)$.
        \item[(ii)] If $1 \leq q \leq p \leq \infty$, then
        \begin{equation}
        \label{eq: lemma Lp Lq estimate 2}
            ||\nabla e^{td\Delta}\phi||_{L^p (\Omega)} \leq C \left(1+t^{ - \frac{1}{2} -\frac{n}{2}\left(\frac{1}{q}-\frac{1}{p}\right)} \right) e^{-d\lambda t}||\phi||_{L^q (\Omega)}
        \end{equation}
        for all $\phi \in L^q(\Omega)$.
        \item[(iii)] If $1\leq q\leq p\leq \infty$, then 
        \begin{equation}
        \label{eq: lemma Lp Lq estimate 3}
            ||e^{td\Delta}\phi||_{L^p (\Omega)} \leq C \left(1+t^{ -\frac{n}{2}\left(\frac{1}{q}-\frac{1}{p}\right)} \right) e^{-d\lambda t}
            ||\phi||_{L^q (\Omega)}
        \end{equation}
        for all $\phi \in L^q(\Omega)$.
        \item[(iv)] If $1\leq q\leq p\leq \infty$, then 
        \begin{equation}
        \label{eq: lemma Lp Lq estimate 4}
            ||e^{td\Delta}\nabla \cdot \phi||_{L^p (\Omega)} \leq C \left(1+t^{- \frac{1}{2} -\frac{n}{2}\left(\frac{1}{q}-\frac{1}{p}\right)} \right) e^{-d\lambda t}||\phi||_{L^q (\Omega)}
        \end{equation}
        for all $\phi \in \left(C_0^\infty (\Omega)\right)^n $.
    \end{itemize}
\end{lemma}

\begin{lemma}{\cite{jin2022global}}
    \label{lem: D2 ln phi}
    Let $\phi \in C^2(\overline{\Omega})$ be a positive function satisfying $\partial_\nu \phi = 0$ on $\partial \Omega$. Then there exists a constant $c > 0$ such that
    \begin{equation*}
        c \left( \int_\Omega \frac{|D^2 \phi|^2}{\phi} + \int_\Omega \frac{|\nabla \phi|^4}{\phi^3} \right)   \leq \int_\Omega \phi |D^2 \ln \phi|^2 .
    \end{equation*}
\end{lemma}

\begin{lemma}{\cite{mizoguchi2014nondegeneracy}}
    \label{lem: boundary term}
    Assume that $\Omega$ is a bounded domain, and let $\phi \in C^2(\overline{\Omega})$ satisfying $\partial_\nu \phi = 0$ on $\partial \Omega$. Then we have 
    \begin{equation*}
        \partial_\nu |\nabla \phi|^2 \leq 2 \kappa |\nabla \phi|^2,
    \end{equation*}
    where $\kappa=\kappa(\Omega)$ is an upper bound of the of the curvatures of $\partial \Omega$.
\end{lemma}

\section{Local existence and positivity}
\label{sec: Local}

In this section we aim to establish local existence of the solution of equation \eqref{eq: general model} with nonnegative initial values. First we prove the positivity of the equation \eqref{eq: general model} -- \eqref{eq: bc ic}. We further generalize the situation into the following partial differential equation on $\Omega \times (0, \infty) \subset \mathbb{R}^n \times \mathbb R$:
    \begin{equation}
        \label{eq: PDE u^k}
        \begin{aligned}
            & \partial_t u^k = \sum_{1 \leq i, j \leq n} a_{ij}^k (x, t) \partial_{x_i x_j} u^k + \sum_{i = 1}^n b_i^k (U, D U, x, t) \partial_{x_i} u^k + f^k (U, DU^{(k)}, D^2 U^{(k)}, x, t) u^k, \\
            &\partial_\nu u^k(x,t) = 0, \ (x, t) \in \partial\Omega \times (0, \infty), \\
            &u^k(x, 0) = u^k_0(x), \ x \in \Omega,
        \end{aligned}
    \end{equation}
    where $k = 1, 2, \ldots, m$, $a_{ij}^k$, $b_i^k$, $\phi^k$ and $u_0^k$ are continuous functions in their arguments,
    \begin{equation*}    
        U = (u^1, u^2, \ldots, u^m) \in C \big( \overline{\Omega} \times [0, T_\text{max}); \mathbb R^m \big) \cap C^{2, 1} \big( \overline{\Omega} \times (0, T_\text{max}); \mathbb R^m \big)
    \end{equation*}
    is a classical solution and $U^{(k)}$ is a vector formed by deleting $u^k$ from $U$.

\begin{lemma}
    \label{lem: general positivity}
    Assume that equation \eqref{eq: PDE u^k} is uniformly positive definite, i.e., there exists a constant $C > 0$ such that
    \begin{equation*}
        \xi^T (a_{ij}^k (x, t))_{1 \leq i, j \leq n} \, \xi \geq C |\xi|^2, \ \forall k = 1, 2, \ldots, m, \ \xi \in \mathbb{R}^n, \ (x, t) \in \overline{\Omega} \times [0, \infty).
    \end{equation*}
    If the initial values satisfy
    \begin{equation*}
        u_0^k(x) \geq 0 \text{ and } u_0^k(x) \not\equiv 0, \ \forall k = 1, 2, \ldots, m,
    \end{equation*}
    then 
    \begin{equation*}
        u^k(x, t) > 0, \ \forall k = 1, 2, \ldots, m \text{ and } (x, t) \in \overline{\Omega} \times (0, \widetilde{T}_\text{max}).
    \end{equation*}
\end{lemma}

\begin{proof}
    For each $A > 0$, define $U_A(x, t) = e^{-At} U(x, t)$, then $U_A$ satisfies
    \begin{equation}
        \label{eq: PDE u_A^k}
        \begin{aligned}
            & \mathcal{L}^k u_A^k = \left( f^k (U, DU^{(k)}, D^2 U^{(k)}, x, t) - A \right) u_A^k, \ (x, t) \in \Omega \times (0, \infty), \\
            &\partial_\nu u_A^k (x,t) = 0, \ (x, t) \in \partial \Omega \times (0, \infty), \\
            &u_A^k(x, 0) = u^k_0(x), \ x \in \Omega,
        \end{aligned}
    \end{equation}
    where
    \begin{equation*}
        \mathcal{L}^k v = \partial_t v - \sum_{1 \leq i, j \leq n} a_{ij}^k (x, t) \partial_{x_i x_j} v - \sum_{i = 1}^n b_i^k (U, DU, x, t) \partial_{x_i} v
    \end{equation*}
    for $U$ satisfying equation \eqref{eq: PDE u^k} and $v \in C \big( \overline{\Omega} \times [0, T_\text{max}) \big) \cap C^{2, 1} \big( \overline{\Omega} \times (0, T_\text{max}) \big)$. For any $T_1 < T_\text{max}$, $U$ is bounded in $\overline{\Omega} \times [0, T_1]$. Hence, we can choose $A > 0$ so that $f^k (U, DU^{(k)}, D^2 U^{(k)}, x, t) - A < 0$ for $k = 1, 2, \ldots, m$ and $(x, t) \in \overline{\Omega} \times [0, T_1]$. Now we apply the strong maximum principle to equation \eqref{eq: PDE u_A^k}. Taking the Neumann boundary condition into account, if $u_A^k$ takes a nonpositive minimum in $\overline{\Omega} \times [0, T_1]$, then $u_A^k$ is a constant. Together with the initial condition, this implies that $u_A^k > 0$ in $\overline{\Omega} \times [0, T_1]$. The same is true for $u^k$ since $u_A^k(x, t) = e^{-At} u^k(x, t)$. Finally, the lemma is proved by letting $T_1 \rightarrow T_\text{max}$.
\end{proof}

It is easy to verify that equation \eqref{eq: general model} -- \eqref{eq: bc ic} can be written into the form of equation \eqref{eq: PDE u^k}, with the coefficients $a_{ij}^k$ given by
\begin{equation*}
    (a_{ij}^k (x, t))_{1 \leq i, j \leq n} = \delta_k I_n
\end{equation*}
and $\delta_k \in \{ \delta_X, \delta_{Y}, \delta_Z \}$. The assumptions in Lemma \ref{lem: general positivity} are also satisfied. Hence, we immediately establish the positivity for equation \eqref{eq: general model} -- \eqref{eq: bc ic}.

\begin{lemma}[Positivity]
    \label{th: positivity}
    The solution to equation \eqref{eq: general model} -- \eqref{eq: bc ic} satisfies 
    \begin{equation*}
        X(x, t), Y(x, t), Z(x, t) > 0, \ \forall (x, t) \in \overline{\Omega} \times (0, T_{\text{max}}).
    \end{equation*}
\end{lemma}

Now we rewrite system \eqref{eq: general model} into following form:
\begin{equation*}
    \partial_t U = \nabla \left (\mathbb A(x, U) \nabla U \right ) + \Phi (x, U), \ (x, t) \in \Omega \times (0, \infty), 
\end{equation*}
where $U=(X,Y,Z)$ is the solution, matrix
\begin{equation}
    \label{eq: matrix mathbb A}
    \mathbb A = \mathbb A (x, U) = 
    \begin{pmatrix}
    \delta_X & 0 & 0 \\
    -\chi_1(x) h(Y) & \delta_{Y} & 0 \\
    0 & -\chi_2(x) h(Z) &\delta_Z
\end{pmatrix},
\end{equation}
and function 
\begin{equation}
    \label{eq: vector Phi}
    \Phi (x, U) = \left (
    \begin{aligned}
        & r(X) - f(X,Y,x) \\
        & e_1 f(X,Y,x) - g(Y,Z,x) - d_1 Y \\
        & e_2 g(Y,Z,x) - d_2 Z 
    \end{aligned}
    \right ). 
\end{equation}

\begin{lemma}[Local existence]
    \label{th: local existence}
    Equation \eqref{eq: general model} -- \eqref{eq: bc ic} has a unique maximal solution $U = (X, Y, Z)$ satisfying
    \begin{equation}
        \label{eq: smoothness}
        U \in C \big( \overline{\Omega} \times [0, T_{\max}); \mathbb R_{\geq 0}^3 \big) \cap C^{2, 1} \big( \overline{\Omega} \times (0, T_{\max}); \mathbb R_+^3 \big),
    \end{equation}
    where $T_{\max} \in (0, \infty]$. If $T_{\max} < \infty$, then
    \begin{equation}
        \label{eq: blow up condition}
        \lim_{t \to T_{\max}} \big( ||X (\cdot, t)||_{L^\infty} + ||Y (\cdot, t)||_{L^\infty} + ||Z (\cdot, t)||_{L^\infty} \big) = \infty.
    \end{equation}
\end{lemma}

\begin{proof}
    Fix $p > n$, $\epsilon > 0$ and define
    \begin{equation*}
        V_\epsilon = \{ v \in W^{1, p} (\Omega; \mathbb{R}^3): v(x) \in G_\epsilon = (-\epsilon, \infty)^3, \ \forall x \in \overline{\Omega} \}.
    \end{equation*}
    It is easy to see that all eigenvalues of the matrix $\mathbb A$ in equation \eqref{eq: matrix mathbb A} have positive real parts. Moreover, assumption ($A_I$) implies that the initial values belong to $V_{\epsilon}$. Therefore, the local existence of the solution $U$ is guaranteed by (\cite{amann1990dynamic}, p. 17), where $U$ satisfy
    \begin{equation*}
        U \in C \big( [0, T_{\max}), V_\epsilon \big) \cap C^{2, 1} \big( \overline{\Omega} \times (0, T_{\max}), \mathbb R^3 \big)
    \end{equation*}
    and $T_{\max} \in (0, \infty]$ is the lifespan defined by
    \begin{equation*}
        T_{\max} := \sup \{ T > 0: U(\cdot, t) \in V_\epsilon, \ \forall t \in [0, T] \}.
    \end{equation*}
    By Sobolev embedding and the fact that $U(\cdot, t) \in C(\overline{\Omega}; \mathbb{R}^3)$ for all $t \geq 0$, we obtain
    \begin{equation*}
        U \in C \big( [0, T_{\max}), C(\overline{\Omega}; \mathbb{R}^3) \big) = C \big( \overline{\Omega} \times [0, T_{\max}); \mathbb R^3 \big).
    \end{equation*}
    Now \eqref{eq: smoothness} follows from Lemma \ref{th: positivity} and assumption ($A_I$). 
    
    By Lemma \ref{th: positivity}, the solution $U$ takes value away from the boundary of $G_\epsilon$. Moreover, the matrix $\mathbb{A}$ is a lower triangular matrix. Hence, we obtain the blow-up criterion \eqref{eq: blow up condition} by (Theorem 6.1, \cite{amann1989dynamic}).
\end{proof}

\section{Global existence}
\label{sec: Global existence}

To study global existence of the solution $U$, it is enough to show that, for some $0< t^* < T_\text{max}$, the solution $\widetilde{U}$ to equation \eqref{eq: general model} -- \eqref{eq: bc ic} with initial data $\widetilde{U}(0) = U(t^*)$ has a global solution. Since $U$ remains positive for all forward time, without loss of generality, throughout Lemma \ref{lem: X bounded} to Corollary \ref{cor: L^inf boundedness of nabla Z} we assume that the initial value $U(0)$ is strictly positive in $\Omega$.

\subsection{$L^\infty$ boundedness of $Y$}

\begin{lemma}
    \label{lem: X bounded}
    The solution to equation \eqref{eq: general model} -- \eqref{eq: bc ic} satisfies 
    \begin{equation}
        \label{eq: X bounded}
        X(x, t) \leq \max \left\{ ||X_0||_{L^\infty}, K \right\} =: \overline{X}, \quad \forall (x, t) \in \overline{\Omega} \times [0, T_{\max}).
    \end{equation}
\end{lemma}

\begin{proof}
    Wherever $X(x, t) \geq K$, the first equation of \eqref{eq: general model} implies
    \begin{equation}
        \label{eq: proof of X bounded}
        \partial_t X - \delta_X \Delta X \leq r(X) \leq 0.
    \end{equation}
    The strong maximum principle can be applied to equation \eqref{eq: proof of X bounded} with boundary condition \eqref{eq: bc ic}. It follows that, if $X$ takes maximum value in $\overline{\Omega} \times (0, T_{\text{max}})$ that is not less than $K$, then $X$ is a constant. Since the initial value is nonnegative by assumption ($A_I$), this proves the bound \eqref{eq: X bounded}.
\end{proof}

\begin{lemma}
    \label{lemma: Y L1 bounded, Z L1 bounded}
    The solution for equation \eqref{eq: general model} -- \eqref{eq: bc ic} satisfies
    \begin{equation}
        \label{eq: Y, Z L1 bounded}
        \left\|Y\right\|_{L^1},\left\|Z\right\|_{L^1} \leq C_1, \quad \forall t \in [0, T_{\max}).
    \end{equation}
\end{lemma}

\begin{proof}
We multiply the first equation in \eqref{eq: general model} by $e_1e_2$, the second equation by $e_2$, and then add the results to the third equation. An integration by parts yields
    \begin{equation}
        \label{eq: proof 1 of Z L1 bounded}
        \frac{\mathrm{d}}{\mathrm{d}t} \int_\Omega \left(e_1 e_2 X + e_2 Y + Z \right)  + e_2 d_1\int_\Omega Y  + d_2\int_\Omega Z  = e_1 e_2 \int_\Omega r(X) .
    \end{equation}
    Adding $e_1 e_2 \min \left\{d_1, d_2\right\} \int_\Omega X $ to both sides of \eqref{eq: proof 1 of Z L1 bounded}, along with Lemma \ref{lem: X bounded}, one obtains the inequality
    \begin{equation}
    \label{eq: proof 2 of Z L1 bounded}
        \begin{aligned}
            & \quad \ \frac{\mathrm{d}}{\mathrm{d}t} \int_\Omega (e_1 e_2 X + e_2 Y + Z )  + \min \left\{ d_1, d_2\right\} \int_\Omega \left(e_1 e_2 X + e_2 Y + Z \right)  \\
            & \leq e_1 e_2 \int_\Omega r(X)  + e_1 e_2 \min \left\{d_1,d_2\right\} \int_\Omega X  \\
            & \leq e_1 e_2 \left( \overline{r(X)} + \min \left\{d_1,d_2\right\} \overline{X}\right) |\Omega|,
        \end{aligned}
    \end{equation}
    where
    \begin{equation}
        \label{eq: r(X) bar}
        \overline{r(X)}=\max\left\{r(X): X \in [0, \overline{X}] \right\}.
    \end{equation}
    Applying Gr\"onwall's inequality to \eqref{eq: proof 2 of Z L1 bounded}, we get
    \begin{equation*}
    \begin{aligned}
        &\quad \ \int_\Omega \left(e_1 e_2 X + e_2 Y + Z \right)\leq \max \left\{e_{1} e_{2} \|X_0 \|_{L^1} + e_{2} \| Y_{0}\|_{L^1} + \| Z_0 \|_{L^1}, \ e_1 e_2 \left( \frac{\overline{r(X)}}{\min \{d_1, d_2\}} + \overline{X}\right) |\Omega| \right\}.
    \end{aligned}
    \end{equation*}
    This implies \eqref{eq: Y, Z L1 bounded}.
\end{proof} 

\begin{lemma}
    \label{lem: energy estimate on X}
    There exists a constant $\nu > 0$ such that the solution for equation \eqref{eq: general model} -- \eqref{eq: bc ic} satisfies
    \begin{equation}
        \label{eq: energy estimate on X}
        \left( \frac{\mathrm{d}}{\mathrm{d}t} + 1 \right) \int_\Omega \left( \frac{|\nabla X|^2}{X} + \tau_1 |\nabla X|^2 \right) + \nu \int_\Omega \left( \frac{|D^2 X|^2}{X} + \frac{|\nabla X|^4}{X^3} \right) 
        \leq - 2 \mu_1 \int_\Omega \nabla X \nabla Y  + C_2, \quad \forall t \in [0, T_{\max}).
    \end{equation}
\end{lemma}

\begin{proof}
    We multiply the first equation of \eqref{eq: general model} by $-\Delta X \left( \frac{1}{X} + \tau_1 \right)$ and integrate by parts, then the left-hand side becomes
    \begin{equation}
        \label{eq: energy estimate on X proof 1}
        \begin{aligned}
            - \int_\Omega X_t \Delta X \left(  \frac{1}{X} + \tau_1 \right)
            &= \int_\Omega \nabla X \nabla(\ln{X})_t  + \int_\Omega \nabla X \nabla (X_t \tau_1)  \\
            &= \int_\Omega \nabla X \left(\frac{\nabla X}{X}\right)_t  + \int_\Omega \nabla X ((\nabla X)_t \tau_1 + X_t \nabla \tau_1)  \\
            &= \frac{\mathrm{d}}{\mathrm{d}t} \int_\Omega \frac{|\nabla X|^2}{X} - \frac{1}{2} \int_\Omega \frac{\left(|\nabla X|^2\right)_t}{X}  + \frac{1}{2} \frac{\mathrm{d}}{\mathrm{d}t} \int_\Omega \tau_1 |\nabla X|^2  + \int_\Omega \nabla \tau_1 \nabla X X_t  \\
            & = \frac{1}{2} \frac{\mathrm{d}}{\mathrm{d}t} \int_\Omega \left( \frac{|\nabla X|^2}{X} + \tau_1 |\nabla X|^2 \right)  - \frac{1}{2}\int_\Omega \frac{|\nabla X|^2}{X^2}X_t  + \int_\Omega \nabla \tau_1 \nabla X X_t .
        \end{aligned}
    \end{equation}
    Using Young's inequality and the first equation of \eqref{eq: general model} again, we get
    \begin{equation}
        \begin{aligned}
            & \quad \ \frac{1}{2} \int_\Omega \frac{|\nabla X|^2}{X^2} X_t  - \int_\Omega \nabla \tau_1 \nabla X X_t  \\
            & = \frac{\delta_X}{2} \int_\Omega \frac{|\nabla X|^2}{X^2}\Delta X  + \frac{r}{2} \int_\Omega \frac{|\nabla X|^2}{X}  - \frac{r}{2K} \int_\Omega |\nabla X|^2  - \frac{1}{2} \int_\Omega \frac{|\nabla X|^2}{X^2} f(X, Y, x)    - \delta_X \int_\Omega \nabla \tau_1 \nabla X \Delta X  \\
            & \quad \ - \int_\Omega \nabla \tau_1 \nabla X r(X)  + \int_\Omega \nabla \tau_1 \nabla X f(X, Y, x)  \\
            & \leq \frac{\delta_X}{2} \int_\Omega \frac{|\nabla X|^2}{X^2}\Delta X  + \frac{r}{2} \int_\Omega \frac{|\nabla X|^2}{X}  - \delta_X \int_\Omega \nabla \tau_1 \nabla X \Delta X  - \int_\Omega \nabla \tau_1 \nabla X r(X)  + \frac{1}{2} \int_\Omega |\nabla \tau_1|^2 X^2 f(X, Y, x) .
        \end{aligned}
    \end{equation}
    In view of $f(X, Y, x) \leq \mu_1 XY$, $X \leq \overline{X}$ \eqref{eq: X bounded}, $||Y||_{L^1} \leq C_1$ \eqref{eq: Y, Z L1 bounded} and assumption ($A_\tau$), we obtain
    \begin{equation*}
        \begin{aligned}
            & \quad \ \frac{1}{2} \int_\Omega \frac{|\nabla X|^2}{X^2} X_t  - \int_\Omega \nabla \tau_1 \nabla X X_t  \\
            & \leq \frac{\delta_X}{2} \int_\Omega \frac{|\nabla X|^2}{X^2}\Delta X  + \frac{r}{2} \int_\Omega \frac{|\nabla X|^2}{X}  + c_1 \int_\Omega \big( |\nabla X| |\Delta X| + |\nabla X| + 1 \big)  \\
            & \leq \frac{\delta_X}{2} \int_\Omega \frac{|\nabla X|^2}{X^2}\Delta X  + c_1 \int_\Omega |\nabla X| |\Delta X|  + c_2 \int_\Omega \frac{|\nabla X|^2}{X}  + c_2.
        \end{aligned}
    \end{equation*}
    In the meantime, when we multiply the first equation of \eqref{eq: general model} by $-\Delta X (\frac{1}{X} + \tau_1)$, the right-hand side becomes
    \begin{equation}
        \label{eq: energy estimate on X proof 10}
        \begin{aligned}
            &\quad \ r \int_\Omega \left( -1 + \left( \frac{1}{K} - \tau_1 \right) X + \frac{\tau_1}{K} X^2 \right) \Delta X + \mu_1 \int_\Omega \Delta X Y  - \delta_X \int_\Omega |\Delta X|^2 \left( \frac{1}{X} + \tau_1 \right)   \\
            &= r \int_\Omega \left( \nabla \tau_1 X \nabla X + \left( \tau_1 - \frac{1}{K} \right) |\nabla X|^2 - \frac{2\tau_1}{K} X |\nabla X|^2 - \frac{1}{K} X^2 \nabla \tau_1 \nabla X \right)  \\
            &\quad \  - \mu_1 \int_\Omega \nabla X \nabla Y  - \delta_X \int_\Omega |\Delta X|^2 \left( \frac{1}{X} + \tau_1 \right)  \\
            &\leq c_3 \int_\Omega \left( |\nabla X| + |\nabla X|^2 \right)  - \mu_1 \int_\Omega \nabla X \nabla Y  - \delta_X \int_\Omega \frac{|\Delta X|^2}{X} \\
            &\leq c_4 \int_\Omega \frac{|\nabla X|^2}{X}  + c_4 - \mu_1 \int_\Omega \nabla X \nabla Y  - \delta_X \int_\Omega \frac{|\Delta X|^2}{X} .
        \end{aligned}
    \end{equation}
    Combining equations \eqref{eq: energy estimate on X proof 1}-\eqref{eq: energy estimate on X proof 10}, we obtain
    \begin{align*}
        &\quad \ \frac{1}{2} \frac{\mathrm{d}}{\mathrm{d}t} \int_\Omega \left( \frac{|\nabla X|^2}{X} + \tau_1 |\nabla X|^2 \right)  + \delta_X \int_\Omega \frac{|\Delta X|^2}{X}  \\&  \leq c_1 \int_\Omega |\nabla X| |\Delta X|  + (c_2 + c_4) \int_\Omega \frac{|\nabla X|^2}{X}  + (c_2 + c_4) - \mu_1 \int_\Omega \nabla X \nabla Y  + \frac{\delta_X}{2} \int_\Omega \frac{|\nabla X|^2}{X^2}\Delta X .
    \end{align*}
    Now we substitute the identity (Equation (3.8), \cite{jin2022global})
    \begin{equation*}
        \int_\Omega \frac{|\nabla X|^2}{X^2} \Delta X  = 2\int_\Omega \frac{|\Delta X|^2}{X}  + \int_{\partial \Omega} \frac{\partial |\nabla X|^2}{\partial \nu} \frac{1}{X} \mathrm{d}S - 2 \int_\Omega X |D^2 \ln X |^2 
    \end{equation*}
    into the previous equation and derive
    \begin{equation*}
        \begin{aligned}
            &\quad \ \frac{1}{2} \frac{\mathrm{d}}{\mathrm{d}t} \int_\Omega \left( \frac{|\nabla X|^2}{X} + \tau_1 |\nabla X|^2 \right)  + \delta_X \int_\Omega X |D^2 \ln X|^2   \\
            & \leq c_1 \int_\Omega |\nabla X| |\Delta X|  + (c_2 + c_4) \int_\Omega \frac{|\nabla X|^2}{X}  + (c_2 + c_4) - \mu_1 \int_\Omega \nabla X \nabla Y  + \frac{\delta_X}{2} \int_{\partial \Omega} \frac{\partial |\nabla X|^2}{\partial \nu} \frac{1}{X} \mathrm{d}S.
        \end{aligned}
    \end{equation*}
    By Lemma \ref{lem: D2 ln phi}, there exists a constant $c > 0$ so that
    \begin{equation}
        \label{eq: energy estimate on X proof 2}
        \begin{aligned}
            &\quad \ \frac{1}{2} \frac{\mathrm{d}}{\mathrm{d}t} \int_\Omega \left( \frac{|\nabla X|^2}{X} + \tau_1 |\nabla X|^2 \right)  + c \int_\Omega \left( \frac{|D^2 X|^2}{X} + \frac{|\nabla X|^4}{X^3} \right) \\ & \leq c_1 \int_\Omega |\nabla X| |\Delta X| + (c_2 + c_4) \int_\Omega \frac{|\nabla X|^2}{X}  + (c_2 + c_4) - \mu_1 \int_\Omega \nabla X \nabla Y  + \frac{\delta_X}{2} \int_{\partial \Omega} \frac{\partial |\nabla X|^2}{\partial \nu} \frac{1}{X} \mathrm{d}S.
        \end{aligned}
    \end{equation}
    Using Lemma \ref{lem: boundary term} and the following trace inequality (\cite{quittner2019superlinear}, Remark 52.9)
    \begin{equation*} 
        ||\psi||_{L^2 (\partial \Omega)} \leq \epsilon ||\nabla \psi||_{L^2 (\Omega)} + C_\epsilon ||\psi||_{L^2 (\Omega)}, \quad \forall \epsilon > 0,
    \end{equation*}
    we derive that
    \begin{equation}
        \label{eq: boundary term}
        \begin{aligned}
            \frac{\delta_X}{2} \int_{\partial \Omega} \frac{\partial |\nabla X|^2}{\partial \nu} \frac{1}{X} \mathrm{d}S
            &\leq \kappa \delta_X \int_{\partial \Omega} \frac{|\nabla X|^2}{X} \mathrm{d}S
            = 4\kappa \delta_X \left\| \nabla (X^{1/2}) \right\|_{L^2 (\partial \Omega)}^2 \\
            &\leq \frac{c}{16} \left\| \frac{D^2 X}{X^{1/2}} - \frac{2|\nabla X|^2}{X^{3/2}} \right\|_{L^2 (\Omega)}^2 + c_5 \int_{\Omega} \frac{|\nabla X|^2}{X}  \\
            & \leq \frac{c}{2} \int_\Omega \left( \frac{|D^2 X|^2}{X} + \frac{|\nabla X|^4}{X^3} \right)  + c_5 \int_{\Omega} \frac{|\nabla X|^2}{X} ,
        \end{aligned}
    \end{equation}
    which, alongside \eqref{eq: energy estimate on X proof 2}, gives
    \begin{equation*}
        \begin{aligned}
            &\quad \ \frac{1}{2} \frac{\mathrm{d}}{\mathrm{d}t} \int_\Omega \left( \frac{|\nabla X|^2}{X} + \tau_1 |\nabla X|^2 \right)  + \frac{c}{2} \int_\Omega \left( \frac{|D^2 X|^2}{X} + \frac{|\nabla X|^4}{X^3} \right)  \\
            &\leq c_1 \int_\Omega |\nabla X| |\Delta X|  + (c_2 + c_4 + c_5) \int_\Omega \frac{|\nabla X|^2}{X}  + (c_2 + c_4) - \mu_1 \int_\Omega \nabla X \nabla Y .
        \end{aligned}
    \end{equation*}
    Using Young’s inequality and the boundedness of $X$, one has
    \begin{align*}
        c_1 \int_\Omega |\nabla X| |\Delta X| 
        &\leq \frac{c}{4} \int_\Omega \frac{|\Delta X|^2}{X}  + \frac{c_1^2}{c} \int_\Omega |\nabla X|^2 X  \\
        &\leq \frac{c}{4} \int_\Omega \frac{|D^2 X|^2}{X}  + \frac{c}{8} \int_\Omega \frac{|\nabla X|^4}{X^3}  + \frac{2 c_1^4}{c^3} \overline{X}^5 |\Omega|
    \end{align*}
    and
    \begin{align*}
        &\quad \ \left( c_2 + c_4 + c_5 \right) \int_\Omega \frac{|\nabla X|^2}{X}  + \frac{1}{2} \int_\Omega \left( \frac{|\nabla X|^2}{X} + \tau_1 |\nabla X|^2 \right)  \leq c_6 \int_\Omega \frac{|\nabla X|^2}{X}  
        \ \leq \frac{c}{8} \int_\Omega \frac{|\nabla X|^4}{X^3}  + \frac{2 c_6^2}{c} \overline{X}|\Omega|.
    \end{align*}
    Therefore,
    \begin{equation*}
        \begin{aligned}
            &\quad \ \frac{1}{2} \left( \frac{\mathrm{d}}{\mathrm{d}t} + 1 \right) \int_\Omega \left( \frac{|\nabla X|^2}{X} + \tau_1 |\nabla X|^2 \right) + \frac{c}{4} \int_\Omega \left( \frac{|D^2 X|^2}{X} + \frac{|\nabla X|^4}{X^3} \right)  \leq - \mu_1 \int_\Omega \nabla X \nabla Y  + c_7,
        \end{aligned}
    \end{equation*}
    which immediately implies \eqref{eq: energy estimate on X}.
\end{proof}

\begin{lemma}
    \label{lem: energy estimate on Y}
    For any $\epsilon > 0$, the solution for equation \eqref{eq: general model} -- \eqref{eq: bc ic} satisfies
    \begin{equation}
        \label{eq: energy estimate on Y}
        \begin{aligned}
            & \quad \left( \frac{\mathrm{d}}{\mathrm{d}t} + 1 \right) \int_\Omega \frac{1}{\chi_1} \mathcal{H}(Y) + \frac{\delta_Y}{2} \int_\Omega \frac{1}{\chi_1} \frac{|\nabla Y|^2}{h(Y)}  \leq \epsilon \int_\Omega \frac{|\nabla X|^4}{X^3}  + \int_\Omega \nabla X \nabla Y  + C_3, \quad \forall t \in [0, T_{\max}).
        \end{aligned}
    \end{equation}
\end{lemma}

\begin{proof}
Multiplying the second equation of \eqref{eq: general model} by $\frac{1}{\chi_1(x)} H(Y)$, we obtain
\begin{equation*}
    \begin{aligned}
        \frac{\mathrm{d}}{\mathrm{d}t} \int_\Omega \frac{1}{\chi_1} \mathcal{H}(Y) 
        &= \int_\Omega Y_t \frac{1}{\chi_1} H(Y)  \\
        &= \int_\Omega (e_1 f(X,Y,x) - g(Y,Z,x) - d_1 Y) \frac{1}{\chi_1} H(Y)  \\
        &\quad \ - \int_\Omega \nabla( \chi_1 h(Y) \nabla X) \frac{1}{\chi_1}H(Y)  + \delta_{Y} \int_\Omega \Delta Y \frac{1}{\chi_1}H(Y) .
    \end{aligned}
\end{equation*} 
Using integration by parts, the second term on the right-hand side becomes
\begin{equation*}
    \begin{aligned}
    -\int_\Omega \nabla( \chi_1 h(Y) \nabla X) \frac{1}{\chi_1}H(Y) 
    &= \int_\Omega \chi_1 h(Y) \nabla X \nabla \left( \frac{1}{\chi_1} H(Y) \right) \\
    & =  \int_\Omega \chi_1 h(Y) \nabla X \left( \nabla\left(\frac{1}{\chi_1}\right) H(Y) + \frac{1}{\chi_1}\frac{\nabla Y}{h(Y)} \right) \\
    &=\int_\Omega \chi_1 \nabla \left(\frac{1}{\chi_1}\right) h(Y)H(Y) \nabla X  + \int_\Omega \nabla X \nabla Y .
    \end{aligned}
\end{equation*}
The third term becomes
\begin{equation*}
    \begin{aligned}
    \delta_{Y} \int_\Omega \Delta Y \frac{1}{\chi_1}H(Y) 
    &= -\delta_Y \int_\Omega \nabla Y \nabla \left(\frac{1}{\chi_1} H(Y) \right)  \\
    &= -\delta_Y \int_\Omega \nabla Y \nabla \left(\frac{1}{\chi_1}\right) H(Y)  - \delta_Y \int_\Omega \frac{|\nabla Y|^2}{h(Y)}\frac{1}{\chi_1} \\
    &= - \delta_Y \int_\Omega \nabla \left(\frac{1}{\chi_1}\right) \nabla \mathcal{H}(Y)  - \delta_Y \int_\Omega \frac{|\nabla Y|^2}{h(Y)} \frac{1}{\chi_1} \\
    &= \delta_Y \int_\Omega \Delta\left(\frac{1}{\chi_1}\right) \mathcal{H}(Y)  - \delta_Y \int_\Omega \frac{|\nabla Y|^2}{h(Y)} \frac{1}{\chi_1} .
    \end{aligned}
\end{equation*}
Therefore,
\begin{equation}
\label{eq: proof: energy estimate on Y}
    \begin{aligned}
        \frac{\mathrm{d}}{\mathrm{d}t} \int_\Omega \frac{1}{\chi_1} \mathcal{H}(Y)  & = \int_\Omega (e_1 f(X,Y,x) - g(Y,Z,x) - d_1 Y) \frac{1}{\chi_1} H(Y)   + \int_\Omega \chi_1 \nabla\left(\frac{1}{\chi_1}\right) h(Y) H(Y) \nabla X  \\
        & \quad \ + \int_\Omega \nabla X \nabla Y   + \delta_Y \int_\Omega \Delta\left(\frac{1}{\chi_1}\right) \mathcal{H}(Y)  - \delta_Y \int_\Omega \frac{|\nabla Y|^2}{h(Y)} \frac{1}{\chi_1} .
    \end{aligned}
\end{equation}
By Table \ref{table: h properties}, $H$ is an increasing function and $z H(z) = 0$ at $z = 0$ and $1$. Therefore, $\min_{z \geq 0} z H(z) := -c_2 > -\infty$. It follows from assumption ($A_\chi$) and the boundedness of $Z$ \eqref{eq: Y, Z L1 bounded} that
\begin{equation*}
    -\int_\Omega g(Y, Z, x) \frac{1}{\chi_1} H(Y)  
    = \int_\Omega \frac{\mu_2 Z}{1 + \tau_2 Y} \frac{1}{\chi_1} (-Y H(Y))  
    \leq \frac{c_2 \mu_2}{\chi_{1m}} C_1.
\end{equation*}
Note that $\chi_1$, $\frac{1}{\chi_1}$, $\nabla \frac{1}{\chi_1}$, $\Delta \frac{1}{\chi_1}$ are bounded by assumption ($A_\chi$). Using Young's inequality and the fact $f(X, Y, x) \leq \mu_1 XY$, for any $\epsilon > 0$ we obtain
\begin{equation*}
    \begin{aligned}
        &\quad \ \left( \frac{\mathrm{d}}{\mathrm{d}t} + 1 \right) \int_\Omega \frac{1}{\chi_1} \mathcal{H}(Y) + \delta_Y \int_\Omega \frac{1}{\chi_1} \frac{|\nabla Y|^2}{h(Y)}  
        \\
        &  \leq -\int_\Omega g(Y,Z,x) \frac{1}{\chi_1}H(Y) + c_1 \int_\Omega (f(X,Y,x) + Y) |H(Y)|   + c_1 \int_\Omega |h(Y) H(Y) \nabla X|  + \int_\Omega \nabla X \nabla Y  + c_1 \int_\Omega \mathcal{H}(Y)  
        \\
        &  \leq \frac{c_2 \mu_2}{\chi_{1m}} C_1 + \epsilon \overline{X}^{-3} \int_\Omega |\nabla X|^4  + c_3 \int_\Omega \left( Y |H(Y)| + |h(Y) H(Y)|^\frac{4}{3} + \mathcal{H} (Y) \right) + \int_\Omega \nabla X \nabla Y  ,
    \end{aligned}
\end{equation*}
where $\overline{X}$ is the maximum of $X$ \eqref{eq: X bounded}. Now Lemma \ref{lemma: G-N inequality on psi} and boundedness of $||Y||_{L^1}$ \eqref{eq: Y, Z L1 bounded} imply
\begin{equation*}
    \begin{aligned}
        & \quad \ \left( \frac{\mathrm{d}}{\mathrm{d}t} + 1 \right) \int_\Omega \frac{1}{\chi_1} \mathcal{H}(Y) + \delta_Y \int_\Omega \frac{1}{\chi_1} \frac{|\nabla Y|^2}{h(Y)}  \\
        & \leq \epsilon \overline{X}^{-3} \int_\Omega |\nabla X|^4  + \frac{\delta_Y}{2 \chi_{1M}} \int_\Omega \frac{|\nabla Y|^2}{h(Y)}  + \int_\Omega \nabla X \nabla Y  + c_4 \\
        & \leq \epsilon \int_\Omega \frac{|\nabla X|^4}{X^3}  + \frac{\delta_Y}{2} \int_\Omega \frac{1}{\chi_1} \frac{|\nabla Y|^2}{h(Y)}  + \int_\Omega \nabla X \nabla Y  + c_4,
    \end{aligned}
\end{equation*}
which implies equation \eqref{eq: energy estimate on Y}.
\end{proof}

\begin{lemma}
    \label{lem: nabla X L2 Y L2-p2}
    The solution for equation \eqref{eq: general model} -- \eqref{eq: bc ic} satisfies
    \begin{equation}
    \label{eq: nabla X L2 Y L2-p2}
        ||\nabla X||_{L^2}, \ ||Y||_{L^{2 - \alpha}} \leq C_4, \quad \forall t \in [0, T_{\max}).
    \end{equation}
\end{lemma}
     
\begin{proof}
By adding up \eqref{eq: energy estimate on X} and \eqref{eq: energy estimate on Y}, we have
\begin{align*}
    & \quad \left( \frac{\mathrm{d}}{\mathrm{d}t} + 1 \right) \int_\Omega \left( \frac{|\nabla X|^2}{X} + \tau_1 |\nabla X|^2 + \frac{2 \mu_1}{\chi_1} \mathcal{H}(Y) \right) + \nu \int_\Omega \left( \frac{|D^2 X|^2}{X} + \frac{|\nabla X|^4}{X^3} \right) \leq 2 \mu_1 \epsilon \int_\Omega \frac{|\nabla X|^4}{X^3}  + C_2 + 2 \mu_1 C_3.
\end{align*}
Note that $\nu > 0$ is a fixed constant and $\epsilon > 0$ is arbitrary. Utilizing the Gr\"onwall inequality and choosing $0 < \epsilon \leq \frac{\nu}{2 \mu_1}$, we derive
\begin{equation*}
\begin{aligned}
    & \quad \ \int_\Omega \left( \frac{|\nabla X|^2}{X} + \tau_1 |\nabla X|^2 + \frac{2 \mu_1}{\chi_1} \mathcal{H}(Y) \right) \\ 
    & \leq c_1 e^{-t} + \int_0^t (C_2 + 2\mu_1 C_3) e^{-(t-s)} \mathrm{d}s\\
    & \leq c_1 + C_2 + 2\mu_1 C_3.
    \end{aligned}
\end{equation*}
By Lemma \ref{lemma: norms of h(psi)} and assumption ($A_\chi$), we have
\begin{equation*}
    \int_\Omega \frac{2\mu_1}{\chi_1} \mathcal{H} (Y)  
    \geq \frac{2\mu_1}{\chi_{1M}} \int_\Omega \mathcal{H} (Y)  
    \geq c_2 ||Y||_{L^{2 - \alpha}}^{2 - \alpha}.
\end{equation*}
It follows from assumption ($A_\tau$) that
\begin{equation*}
    \tau_{1m} ||\nabla X||_{L^2}^2 + c_2 ||Y||_{L^{2 - \alpha}}^{2 - \alpha}  
    \leq c_1 + C_2 + 2\mu_1 C_3,
\end{equation*}
which implies \eqref{eq: nabla X L2 Y L2-p2}.
\end{proof}

Denote $\eta^* = \frac{2}{n} - \frac{\alpha}{2 - \alpha}$, which is positive due to $\alpha < \frac{4}{n + 2}$, and fix a number $0 < \eta < \min \left\{ \frac{1}{n}, \eta^* \right\}$. Define an increasing sequence of numbers $\{ q_k \}$ by $q_0 = 2 - \alpha$ and $\frac{1}{q_k} = \frac{1}{q_{k - 1}} - \eta$ until $q_k > n$ for some $k = k^*$. In addition, we set $q_{k^* + 1} = \infty$. Using induction, we will prove the boundedness of the solution in $L^{q_k} (\Omega)$ for any $0 \leq k \leq k^* + 1$. 

\begin{lemma}
    \label{lemma: L^qk boundedness of nabla X}
    Suppose the solution for equation \eqref{eq: general model} -- \eqref{eq: bc ic} satisfies
    \begin{equation}
        \label{eq: assumption Y Lqk for nabla X}
        ||Y||_{L^{q_k}} \leq C_{5, k}, \quad \forall t \in [0, T_{\max})
    \end{equation}
    for some $0 \leq k \leq k^*$. Then for any $\overline{q}_{k + 1} \in [1, \infty]$ satisfying
    \begin{equation}
        \label{eq: bar q_k+1}
        \frac{1}{\overline{q}_{k + 1}} > \frac{1}{q_k} - \frac{1}{n},
    \end{equation}
    we have
    \begin{equation}
        \label{eq: L^qk boundedness of nabla X}
        ||\nabla X||_{L^{\overline{q}_{k + 1}}} \leq C_{5, k + 1}, \quad \forall t \in [0, T_{\max}).
    \end{equation}
    In particular, we can set $\overline{q}_{k + 1} = q_{k + 1}$.
\end{lemma}
 
\begin{proof}
    Apply the variation of constants formula to the first equation in \eqref{eq: general model} and obtain
    \begin{equation}
        \label{eq: variation of constants X}
        \begin{aligned}
            X(t) & = e^{\delta_X t \Delta} X_0 + \int_0^t e^{\delta_X (t-s)\Delta}\left(r(X) - f(X,Y,x) \right) \mathrm{d}s.
        \end{aligned}
    \end{equation}
    Taking the gradient at each side, we get
    \begin{equation*}
        \begin{aligned}
            \nabla X(t) & = \nabla e^{\delta_X t \Delta} X_0 + \int_0^t \nabla e^{\delta_X (t-s)\Delta} \left(r(X) - f(X,Y,x) \right) \mathrm{d}s.
        \end{aligned}
    \end{equation*}
    Applying \eqref{eq: lemma Lp Lq estimate 1} and \eqref{eq: lemma Lp Lq estimate 2} and setting
    \begin{equation}
        \label{eq: rho}
        \rho = \frac{n}{2}\left(\frac{1}{q_k}-\frac{1}{\overline{q}_{k+1}}\right) \in \Big[ 0, \frac{1}{2} \Big),
    \end{equation}
    we derive
    \begin{equation*}
    \begin{aligned}
        ||\nabla X||_{L^{{\overline{q}}_{k + 1}}} 
        & \leq c_1 e^{-\delta_X \lambda t}||\nabla X_0||_{L^{{\overline{q}}_{k + 1}}} + c_1 \int_0^t \left( 1 + (t-s)^{-\frac{1}{2} - \rho} \right) e^{- \delta_X \lambda (t-s)} \Big( || r(X)||_{L^{q_k}} + || f(X,Y,x) ||_{L^{q_k}} \Big) \mathrm{d}s
        \end{aligned}
    \end{equation*}
    for some $c_1 > 0$. By the boundedness of $\| X \|_{L^{\infty}}$ \eqref{eq: X bounded} and $||Y||_{L^{q_k}}$ \eqref{eq: assumption Y Lqk for nabla X} we obtain
    \begin{equation*}
            || r(X) ||_{L^{q_k}} \leq \overline{r(X)} |\Omega|^{1/q_k}, 
    \end{equation*} 
            and
    \begin{equation*}
            || f(X,Y,x) ||_{L^{q_k}} \leq \mu_1 ||X||_{L^\infty} ||Y||_{L^{q_k}} \leq \mu_1 \overline{X} C_{5, k}.
    \end{equation*}
    Therefore,
    \begin{equation*}
        \begin{aligned}
            ||\nabla X||_{L^{\overline{q}_{k+1}}} 
            & \leq c_1 e^{-\delta_X \lambda t} ||\nabla X_0||_{L^{q_{k+1}}} + c_1 \int_0^t \left( 1 + (t-s)^{-\frac{1}{2} - \rho} \right) e^{- \delta_X \lambda (t-s)} \left( \overline{r(X)} |\Omega|^{1/q_k} + \mu_1 \overline{X} C_{5, k} \right) \mathrm{d}s \\
            & \leq c_1 ||\nabla X_0||_{L^{q_{k+1}}} + c_1 c_2 \left( \overline{r(X)} |\Omega|^{1/q_k} + \mu_1 \overline{X} C_{5, k} \right)  =: C_{5, k+1}.
        \end{aligned}
    \end{equation*}
\end{proof}

\begin{lemma}
    \label{lem: L^qk boundedness of Y}
    Suppose the solution for equation \eqref{eq: general model} -- \eqref{eq: bc ic} satisfies
    \begin{equation}
        \label{eq: assumption L^qk boundedness of Y}
        ||Y||_{L^{q_k}} \leq C_{6, k}, \quad \forall t \in [0, T_{\max})
    \end{equation}
    for some $0 \leq k \leq k^*$, then
    \begin{equation}
        \label{eq: L^qk boundedness of Y}
        ||Y||_{L^{q_{k + 1}}} \leq C_{6, k + 1}, \quad \forall t \in [0, T_{\max}).
    \end{equation}
\end{lemma}

\begin{proof}
    We consider $k < k^*$ first. Applying the variation of constants formula to the second equation in \eqref{eq: general model}, we obtain
    \begin{equation}
    \label{eq: variation of constants Y}
    \begin{aligned}
        Y(t) &= e^{(\delta_{Y}\Delta -d_1)t}Y_{0} + \int_0^t e^{(\delta_{Y}\Delta - d_1) (t-s)} \left(e_1f(X,Y,x) - g(Y,Z,x) -\nabla \left(\chi_1(x) h(Y)\nabla X\right)\right) \mathrm{d}s \\
        & \leq e^{(\delta_{Y}\Delta -d_1)t}Y_{0} + \int_0^t e^{(\delta_{Y}\Delta - d_1) (t-s)} \left(e_1f(X,Y,x) -\nabla (\chi_1(x) h(Y)\nabla X\right)) \, \mathrm{d}s.
    \end{aligned}
    \end{equation}
    It follows that
    \begin{equation}
        \label{eq: proof 0 L^qk boundedness of Y}
        \begin{aligned}
            ||Y||_{L^{q_{k+1}}}
            &\leq \left\| e^{(\delta_{Y}\Delta -d_1)t} Y_0 \right\|_{L^{q_{k + 1}}} + \int_0^t \Big( e_1 \left\| e^{(\delta_{Y}\Delta - d_1) (t-s)} f(X,Y,x) \right\|_{L^{q_{k + 1}}} \\
            &\quad \ + \left\| e^{(\delta_{Y} \Delta - d_1) (t-s)} \nabla (\chi_1(x) h(Y)\nabla X) \right\|_{L^{q_{k+1}}} \Big) \mathrm{d}s.
        \end{aligned}
    \end{equation}
    Equation \eqref{eq: lemma Lp Lq estimate 3} implies
    \begin{equation}
        \label{eq: proof 1 L^qk boundedness of Y}
        \left\| e^{(\delta_{Y}\Delta -d_1)t} Y_0 \right\|_{L^{q_{k + 1}}}
        \leq c_1 e^{-(\delta_{Y} \lambda + d_1) t} ||Y_0||_{L^{q_{k+1}}},
    \end{equation}
    and
    \begin{equation*}
        \left\| e^{(\delta_{Y}\Delta - d_1) (t-s)} f(X,Y,x) \right\|_{L^{q_{k + 1}}}
        \leq c_2 \left( 1 + (t-s)^{-\rho_1} \right)e^{-(\delta_{Y} \lambda + d_1) (t-s)} ||f(X,Y,x)||_{L^{q_k}},
    \end{equation*}
    where
    \begin{equation}
        \label{eq: rho_1}
        0 < \rho_1 = \frac{n}{2} \left( \frac{1}{q_k} - \frac{1}{q_{k+1}} \right) = \frac{n}{2} \eta < \frac{1}{2}.
    \end{equation}
    By assumption \eqref{eq: assumption L^qk boundedness of Y}, we get
    \begin{equation*}
        ||f(X,Y,x)||_{L^{q_k}} \leq \mu_1 \overline{X} C_{6, k},
    \end{equation*}
    and
    \begin{equation}
        \label{eq: proof 2 L^qk boundedness of Y}
        \left\| e^{(\delta_{Y}\Delta - d_1) (t-s)} f(X,Y,x) \right\|_{L^{q_{k + 1}}}
        \leq \mu_1 \overline{X} C_{6, k} \left( 1 + (t-s)^{-\rho_1} \right)e^{-(\delta_{Y} \lambda + d_1) (t-s)}.
    \end{equation}
    Choose a number $q \leq q_{k+1}$ such that
    \begin{equation}
        \label{eq: proof 1 q_k+1}
        \frac{1}{q_{k+1}} + \frac{1}{n} + \eta - \eta^* < \frac{1}{q} < \frac{1}{q_{k+1}} + \frac{1}{n}.
    \end{equation}
    By \eqref{eq: lemma Lp Lq estimate 4},
    \begin{equation}
        \label{eq: proof 3 L^qk boundedness of Y}
        \begin{aligned}
            &\quad \ \left\| e^{(\delta_{Y} \Delta - d_1) (t-s)} \nabla (\chi_1(x) h(Y)\nabla X) \right\|_{L^{q_{k+1}}}
           \leq c_3 \left(1+ (t-s)^{- \frac{1}{2} - \rho_2} \right)e^{-(\delta_{Y} \lambda + d_1) (t-s)}||\chi_1(x) h(Y)\nabla X||_{L^q},
        \end{aligned}
    \end{equation}
    where
    \begin{equation}
        \label{eq: rho_2}
        0 \leq \rho_2 = \frac{n}{2} \left( \frac{1}{q} - \frac{1}{q_{k+1}} \right) < \frac{1}{2}.
    \end{equation}
    Let $\overline{q}_{k+1}$ satisfy
    \begin{equation}
    \label{eq: proof 2 bar_q_k+1}
        \frac{1}{q} = \frac{1}{\overline{q}_{k+1}} + \frac{\alpha}{q_k}.
    \end{equation}
    Then, on the one hand, we obtain by assumption ($A_\chi$), H\"older's inequality and Lemma \ref{lemma: norms of h(psi)} that
    \begin{equation}
        \label{eq: proof 4 L^qk boundedness of Y}
        \begin{aligned}
            ||\chi_1(x) h(Y) \nabla X||_{L^q} 
            & \leq \chi_{1M} \| \nabla X \|_{L^{\overline{q}_{k+1}}} \left\| h(Y) \right\|_{L^{q_k / \alpha}} \leq c_4 \chi_{1M} \| \nabla X \|_{L^{\overline{q}_{k+1}}} \left( \| Y \|_{L^{q_k}}^{q_k} + 1 \right)^{\alpha / q_k}.
        \end{aligned}
    \end{equation}
    On the other hand, together with \eqref{eq: proof 2 bar_q_k+1}, definition of $\eta^*$ and the fact $q_k \geq q_0 = 2 - \alpha$, we get
    \begin{equation}
        \label{eq: proof q_k+1 bar satisfies condition}
        \frac{1}{\overline{q}_{k+1}}
        > \frac{1}{q_{k+1}} + \frac{1}{n} - \eta - \eta^* - \frac{\alpha}{q_k}
        = \frac{1}{q_k} - \frac{1}{n} + \frac{\alpha}{2 - \alpha} - \frac{\alpha}{q_k}
        \geq \frac{1}{q_k} - \frac{1}{n}.
    \end{equation}
    That is, $\overline{q}_{k+1}$ satisfies condition \eqref{eq: bar q_k+1}. Now, estimates of $||\nabla X||_{L^{\overline{q}_{k+1}}}$ ( Lemma \ref{lemma: L^qk boundedness of nabla X}) and $||Y||_{L^{q_k}}$ \eqref{eq: assumption L^qk boundedness of Y} imply
    \begin{equation}
        \label{eq: proof 6 L^qk boundedness of Y}
        ||\chi_1(x) h(Y) \nabla X||_{L^q} 
        \leq c_4 \chi_{1M} C_{5} \left( C_{6, k}^{q_k} + 1 \right)^{\alpha / q_k} = c_5.
    \end{equation}
    By equations \eqref{eq: proof 3 L^qk boundedness of Y} and \eqref{eq: proof 6 L^qk boundedness of Y} we get
    \begin{equation}
        \label{eq: proof 5 L^qk boundedness of Y}
        \begin{aligned}
            \left\| e^{(\delta_{Y} \Delta - d_1) (t-s)} \nabla (\chi_1(x) h(Y)\nabla X) \right\|_{L^{q_{k+1}}}
            \leq c_3 c_5 \left(1+ (t-s)^{- \frac{1}{2} - \rho_2} \right)e^{-(\delta_{Y} \lambda + d_1) (t-s)}.
        \end{aligned}
    \end{equation}
Combining equations \eqref{eq: proof 0 L^qk boundedness of Y}, \eqref{eq: proof 1 L^qk boundedness of Y}, \eqref{eq: proof 2 L^qk boundedness of Y}, \eqref{eq: proof 5 L^qk boundedness of Y} and noting that $\rho_1, \rho_2 < \frac{1}{2}$, we obtain
\begin{equation*}
    \begin{aligned}
        ||Y(t)||_{L^{q_{k+1}}}
        & \leq c_1 e^{-(\delta_{Y} \lambda + d_1) t} ||Y_{0}||_{L^{q_{k+1}}} \\
        & \quad + e_1 \mu_1 \overline{X} C_{6, k} \int_0^t \left( 1 + (t-s)^{-\rho_1} \right)e^{-(\delta_{Y} \lambda + d_1) (t-s)} \mathrm{d}s \\
        & \quad + c_3 c_5 \int_0^t \left(1+ (t-s)^{- \frac{1}{2} - \rho_2} \right)e^{-(\delta_{Y} \lambda + d_1) (t-s)} \mathrm{d}s =: C_{6, k+1}.
    \end{aligned}
\end{equation*}
When $k = k^*$, we set $q = q_{k^*} > n$, then the corresponding $\rho_1$ and $\rho_2$ are still less than $\frac{1}{2}$. Besides, we set $\overline{q}_{k^*+1} = \infty$, which satisfies \eqref{eq: bar q_k+1}. Now we can repeat the rest of the proof, replacing only the estimate \eqref{eq: proof 4 L^qk boundedness of Y} by
\begin{align*}
    ||\chi_1(x) h(Y) \nabla X||_{L^q} 
    & \leq \chi_{1M} \| \nabla X \|_{L^{\overline{q}_{k^* + 1}}} \left\| h(Y) \right\|_{L^{q_{k^*}}} \\
    & \leq \chi_{1M} |\Omega|^{(1 - \alpha) / q_{k^*}} \| \nabla X \|_{L^{\overline{q}_{k^* + 1}}} \left\| h(Y) \right\|_{L^{q_{k^*} / \alpha}} \\
    & \leq c_4 \chi_{1M} |\Omega|^{(1 - \alpha) / q_{k^*}} \| \nabla X \|_{L^{\overline{q}_{k^* + 1}}} \left( \| Y \|_{L^{q_{k^*}}}^{q_{k^*}} + 1 \right)^{\alpha / q_{k^*}}.
\end{align*}
\end{proof}
Taking $k = k^*$ in previous two estimates, we immediately obtain the following corollary.

\begin{corollary}
    \label{cor: L^inf boundedness of Y nabla X}
    The solution for equation \eqref{eq: general model} -- \eqref{eq: bc ic} satisfies
    \begin{equation}
        \label{eq: L^inf boundedness of Y nabla X}
        ||Y||_{L^\infty}, \ ||\nabla X||_{L^\infty} \leq C_7, \quad \forall t \in [0, T_{\max}).
    \end{equation}
\end{corollary}

\subsection{$L^\infty$ boundedness of $Z$}

\begin{lemma}
    \label{lem: Delta X Lp Lq_eps bounded}
    For all $1 < p, q < \infty$ and $\epsilon > 0$, the solution for equation \eqref{eq: general model} -- \eqref{eq: bc ic} satisfies
    \begin{equation}
        \label{eq: Delta X Lp Lq_eps bounded}
        ||\Delta X||_{L^p (\Omega) L^q_\epsilon (0, t)} \leq C_8, \quad \forall t \in [0, T_{\max}).
    \end{equation}
\end{lemma}

\begin{proof}
    Let $\widetilde{X}(t) = e^{\epsilon t} (X(t) - X_0)$, then we obtain by the first equation of \eqref{eq: general model} that
    \begin{equation*}
        \widetilde{X}_t = \delta_X \Delta \widetilde{X} - \epsilon \widetilde{X} + F(x,t),
    \end{equation*}
    where 
    \begin{equation}
        F(x,t)= e^{\epsilon t} \big( \delta_X \Delta X_0 + 2\epsilon (X - X_0) + r(X) - f(X,Y,x) \big).
    \end{equation}
    By maximal regularity property (Theorem 3.1, \cite{matthias1997heat}), for any $t \in [0, T_\text{max})$ we have
    \begin{equation*}
    \begin{aligned}
        \int_0^t \left\| (\delta_X \Delta - \epsilon) \widetilde{X} \right\|_{L^p}^q \mathrm{d}s 
        & \leq c_1 \int_0^t \left\| F(\cdot, s) \right\|_{L^p}^q \mathrm{d}s.
        \end{aligned}
    \end{equation*}
    The boundedness of $X$ \eqref{eq: X bounded} and $Y$ \eqref{eq: L^inf boundedness of Y nabla X} imply
    \begin{align*}
        \int_0^t \| \Delta \widetilde{X} \|_{L^p}^q \mathrm{d}s 
        & \leq c_2 \int_0^t \| \widetilde{X} \|_{L^p}^q \mathrm{d}s + c_2 \int_0^t \left\| F(\cdot ,s) \right\|_{L^p}^q \mathrm{d}s \\
        &\leq c_3 \int_0^t e^{\epsilon q s} \left( ||\Delta X_0||_{L^p}^q + ||X||_{L^p}^q + ||X_0||_{L^p}^q  + \||\overline{r(X)}||_{L^p}^q + ||X||_{L^\infty}^q||Y||_{L^p}^q \right) \mathrm{d}s \\
        & \leq \frac{c_3}{\epsilon q} e^{\epsilon q t} \Big( ||\Delta X_0||_{L^p}^q + ||X_0||_{L^p}^q + c_4 \Big).
    \end{align*}
    Therefore,
    \begin{align*}
        ||\Delta X||_{L^p (\Omega) L^q_\epsilon (0, t)}^q 
        & = e^{-\epsilon q t} \int_0^t \| e^{\epsilon s} \Delta X \|_{L^p}^q \mathrm{d}s \\
        & = 2^{q-1}e^{-\epsilon q t} \int_0^t \left( \| \Delta \widetilde{X} \|_{L^p}^q + e^{\epsilon q s} \| \Delta X_0 \|_{L^p}^q \right) \mathrm{d}s\\
        & \leq 2^{q-1} \frac{c_3 +1}{\epsilon q} \Big( ||\Delta X_0||_{L^p}^q + ||X_0||_{L^p}^q + c_4 \Big) =: C_8^q.
    \end{align*}
\end{proof}

\begin{lemma}
    \label{lem: new energy estimate on Y}
    There exists a constant $\nu > 0$ such that the solution for equation \eqref{eq: general model} -- \eqref{eq: bc ic} satisfies
    \begin{equation}
        \label{eq: new energy estimate on Y}
        \begin{aligned}
            &\quad \left( \frac{\mathrm{d}}{\mathrm{d}t} + d_1 \right) \int_\Omega \left( \frac{|\nabla Y|^2}{Y} + \tau_2 |\nabla Y|^2 \right) + \nu \int_\Omega \left( \frac{|D^2 Y|^2}{Y} + \frac{|\nabla Y|^4}{Y^3} \right) \\
            &\leq C_{9} \left( \int_\Omega |\Delta X|^2  + 1 \right) - 2\mu_2 \int_\Omega \nabla Y \nabla Z , \quad \forall t \in [0, T_{\max}).
        \end{aligned}
    \end{equation}
\end{lemma}

\begin{proof}
    We multiply the second equation of \eqref{eq: general model} by $-\Delta Y \left( \frac{1}{Y} + \tau_2 \right)$ and integrate by parts. Using the same calculation as in \eqref{eq: energy estimate on X proof 1}, the left-hand side becomes
    \begin{equation}
        \label{eq: new energy estimate on Y proof 1}
        \begin{aligned}
            & \quad \ - \int_\Omega Y_t \Delta Y \left(  \frac{1}{Y} + \tau_2 \right) = \frac{1}{2} \frac{\mathrm{d}}{\mathrm{d}t} \int_\Omega \left( \frac{|\nabla Y|^2}{Y} + \tau_2 |\nabla Y|^2 \right) - \frac{1}{2} \int_\Omega \frac{|\nabla Y|^2}{Y^2} Y_t  + \int_\Omega \nabla \tau_2 \nabla Y Y_t .
        \end{aligned}
    \end{equation}
    Using Young's inequality and the second equation of \eqref{eq: general model} again, we get
    \begin{align*}
        & \quad \ \frac{1}{2} \int_\Omega \frac{|\nabla Y|^2}{Y^2} Y_t  - \int_\Omega \nabla \tau_2 \nabla Y Y_t  \\
        & = \int_\Omega \left( \frac{|\nabla Y|^2}{2 Y^2} - \nabla \tau_2 \nabla Y \right) \left( \delta_Y \Delta Y - d_1 Y + e_1 f(X, Y, x) + \nabla (\chi_1(x)h(Y)\nabla X) \right)  \\
        & \quad \ - \int_\Omega \left( \frac{|\nabla Y|^2}{2 Y^2} - \nabla \tau_2 \nabla Y \right) g(Y, Z, x)  \\
        & \leq \int_\Omega \left( \frac{|\nabla Y|^2}{2 Y^2} - \nabla \tau_2 \nabla Y \right) \left( \delta_Y \Delta Y - d_1 Y + e_1 f(X, Y, x) + \nabla (\chi_1(x)h(Y)\nabla X) \right) + \frac{1}{2} \int_\Omega |\nabla \tau_2|^2 Y^2 g(Y, Z, x) .
    \end{align*}
    The facts $h(z) = O(z)$, $h'(z) = O(1)$ as $z \to 0$, $\nabla X$, $Y$ are bounded \eqref{eq: L^inf boundedness of Y nabla X} and assumption ($A_\chi$) imply that 
    \begin{equation}
        |\nabla (\chi_1 (x) h(Y) \nabla X)| 
        \leq c_1 \left( Y^{1/2} + |\nabla Y| + Y^{1/2} |\Delta X| \right).
    \end{equation}
    Since $f(X, Y, x) \leq \mu_1 XY$, $g(Y, Z, x) \leq \mu_2 YZ$, $X \leq \overline{X}$ \eqref{eq: X bounded}, $||Y||_{L^\infty} \leq C_7$ \eqref{eq: L^inf boundedness of Y nabla X} and $||Z||_{L^1} \leq C_1$ \eqref{eq: Y, Z L1 bounded}, we obtain
    \begin{equation}
        \begin{aligned}
            & \quad \ \frac{1}{2} \int_\Omega \frac{|\nabla Y|^2}{Y^2} Y_t  - \int_\Omega \nabla \tau_2 \nabla Y Y_t  \\ & \leq \frac{\delta_Y}{2} \int_\Omega \frac{|\nabla Y|^2}{Y^2} \Delta Y  + c_2 \int_\Omega \left( |\nabla Y| |\Delta Y| + \frac{|\nabla Y|^2}{Y} + |\nabla Y| + 1 \right)  \\
            & \quad \ + c_2 \int_\Omega \left( \frac{|\nabla Y|^2}{Y^2} + |\nabla Y| \right) \left( Y^{1/2} + |\nabla Y| + Y^{1/2} |\Delta X| \right) \\
            & \leq \frac{\delta_Y}{2} \int_\Omega \frac{|\nabla Y|^2}{Y^2} \Delta Y  + c_3 \int_\Omega \left( |\nabla Y| |\Delta Y| + \frac{|\nabla Y|^2}{Y^{3/2}} + \frac{|\nabla Y|^3}{Y^2} \right)  + c_3 \int_\Omega \left( \frac{|\nabla Y|^2}{Y^{3/2}} + 1 \right) |\Delta X|  + c_3.
        \end{aligned}
    \end{equation}
    In the meantime, when we multiply the second equation of \eqref{eq: general model} by $-\Delta Y (\frac{1}{Y} + \tau_2)$, the right-hand side becomes
    \begin{equation}
        \label{eq: new energy estimate on Y proof 2}
        \begin{aligned}
            & \quad \ -e_1 \int_\Omega \frac{\mu_1 X}{1 + \tau_1 X} \Delta Y (1 + \tau_2 Y)  - \mu_2 \int_\Omega \nabla Y \nabla Z  - d_1 \int_\Omega \nabla Y \nabla (\tau_2 Y)  \\ 
            & \quad \ + \int_\Omega \nabla (\chi_1 (x) h(Y) \nabla X) \Delta Y \left( \frac{1}{Y} + \tau_2 \right)  - \delta_Y \int_\Omega |\Delta Y|^2 \left( \frac{1}{Y} + \tau_2 \right)  \\
            &\leq c_4 \int_\Omega \big( |\Delta Y| + |\nabla Y| \big)  - \mu_2 \int_\Omega \nabla Y \nabla Z  - d_1 \int_\Omega \tau_2 |\nabla Y|^2  \\
            &\quad \ + c_4 \int_\Omega \left( Y^{1/2} + |\nabla Y| + Y^{1/2} |\Delta X| \right) \frac{|\Delta Y|}{Y}  - \delta_Y \int_\Omega \frac{|\Delta Y|^2}{Y}  \\
            &\leq c_5 \int_\Omega \left( \frac{|\Delta Y|}{Y^{1/2}} + \frac{|\nabla Y| |\Delta Y|}{Y} \right)  + c_5 \int_\Omega \frac{|\Delta Y|}{Y^{1/2}} |\Delta X|  + c_5 \\
            &\quad \ - \mu_2 \int_\Omega \nabla Y \nabla Z  - \frac{d_1}{2} \int_\Omega \tau_2 |\nabla Y|^2  - \delta_Y \int_\Omega \frac{|\Delta Y|^2}{Y} .
        \end{aligned}
    \end{equation}
    Combining equations \eqref{eq: new energy estimate on Y proof 1}-\eqref{eq: new energy estimate on Y proof 2}, we obtain
    \begin{equation}
        \begin{aligned}
            &\quad \ \frac{1}{2} \left( \frac{\mathrm{d}}{\mathrm{d}t} + d_1 \right) \int_\Omega \left( \frac{|\nabla Y|^2}{Y} + \tau_2 |\nabla Y|^2 \right) + \delta_Y \int_\Omega \frac{|\Delta Y|^2}{Y}   \\
            &\leq \frac{\delta_Y}{2} \int_\Omega \frac{|\nabla Y|^2}{Y^2} \Delta Y  + (c_3 + c_5) \int_\Omega \left( \frac{|\nabla Y| |\Delta Y|}{Y} + \frac{|\nabla Y|^2}{Y^{3/2}} + \frac{|\nabla Y|^3}{Y^2} + \frac{|\Delta Y|}{Y^{1/2}} \right)  \\
            & \quad \ + (c_3 + c_5) \int_\Omega \left( \frac{|\nabla Y|^2}{Y^{3/2}} + \frac{|\Delta Y|}{Y^{1/2}} + 1 \right) |\Delta X|  + (c_3 + c_5) - \mu_2 \int_\Omega \nabla Y \nabla Z .
        \end{aligned}
    \end{equation}
    Now we substitute the identity (Equation (3.8), \cite{jin2022global})
    \begin{equation*}
        \int_\Omega \frac{|\nabla Y|^2}{Y^2} \Delta Y  = 2\int_\Omega \frac{|\Delta Y|^2}{Y}  + \int_{\partial \Omega} \frac{\partial |\nabla Y|^2}{\partial \nu} \frac{1}{Y} \mathrm{d}S - 2 \int_\Omega Y |D^2 \ln Y|^2 
    \end{equation*}
    into the previous equation and derive
    \begin{align*}
        &\quad \ \frac{1}{2} \left( \frac{\mathrm{d}}{\mathrm{d}t} + d_1 \right) \int_\Omega \left( \frac{|\nabla Y|^2}{Y} + \tau_2 |\nabla Y|^2 \right) + \delta_Y \int_\Omega Y |D^2 \ln Y|^2  \\
        &\leq \frac{\delta_Y}{2} \int_{\partial \Omega} \frac{\partial |\nabla Y|^2}{\partial \nu} \frac{1}{Y} \mathrm{d}S + (c_3 + c_5) \int_\Omega \left( \frac{|\nabla Y| |\Delta Y|}{Y} + \frac{|\nabla Y|^2}{Y^{3/2}} + \frac{|\nabla Y|^3}{Y^2} + \frac{|\Delta Y|}{Y^{1/2}} \right)  \\
        & \quad \ + (c_3 + c_5) \int_\Omega \left( \frac{|\nabla Y|^2}{Y^{3/2}} + \frac{|\Delta Y|}{Y^{1/2}} + 1 \right) |\Delta X|  + (c_3 + c_5) - \mu_2 \int_\Omega \nabla Y \nabla Z .
    \end{align*}
    By Lemma \ref{lem: D2 ln phi}, there exists a constant $c > 0$ so that
    \begin{align*}
        &\quad \ \frac{1}{2} \left( \frac{\mathrm{d}}{\mathrm{d}t} + d_1 \right) \int_\Omega \left( \frac{|\nabla Y|^2}{Y} + \tau_2 |\nabla Y|^2 \right) + c \int_\Omega \left( \frac{|D^2 Y|^2}{Y} + \frac{|\nabla Y|^4}{Y^3} \right)  \\
        &\leq \frac{\delta_Y}{2} \int_{\partial \Omega} \frac{\partial |\nabla Y|^2}{\partial \nu} \frac{1}{Y} \mathrm{d}S + (c_3 + c_5) \int_\Omega \left( \frac{|\nabla Y| |\Delta Y|}{Y} + \frac{|\nabla Y|^2}{Y^{3/2}} + \frac{|\nabla Y|^3}{Y^2} + \frac{|\Delta Y|}{Y^{1/2}} \right)  \\
        & \quad \ + (c_3 + c_5) \int_\Omega \left( \frac{|\nabla Y|^2}{Y^{3/2}} + \frac{|\Delta Y|}{Y^{1/2}} + 1 \right) |\Delta X|  + (c_3 + c_5) - \mu_2 \int_\Omega \nabla Y \nabla Z .
    \end{align*}
    Now the boundary estimate \eqref{eq: boundary term} implies
    \begin{equation}
        \label{eq: new energy estimate on Y proof 3}
        \begin{aligned}
            &\quad \ \frac{1}{2} \left( \frac{\mathrm{d}}{\mathrm{d}t} + d_1 \right) \int_\Omega \left( \frac{|\nabla Y|^2}{Y} + \tau_2 |\nabla Y|^2 \right) + \frac{c}{2} \int_\Omega \left( \frac{|D^2 Y|^2}{Y} + \frac{|\nabla Y|^4}{Y^3} \right)  \\
            &\leq (c_3 + c_5 + c_6) \int_\Omega \left( \frac{|\nabla Y| |\Delta Y|}{Y} + \frac{|\nabla Y|^2}{Y^{3/2}} + \frac{|\nabla Y|^3}{Y^2} + \frac{|\Delta Y|}{Y^{1/2}} \right)  \\
            & \quad \ + (c_3 + c_5) \int_\Omega \left( \frac{|\nabla Y|^2}{Y^{3/2}} + \frac{|\Delta Y|}{Y^{1/2}} + 1 \right) |\Delta X|  + (c_3 + c_5) - \mu_2 \int_\Omega \nabla Y \nabla Z .
        \end{aligned}
    \end{equation}
    Using Young’s inequality and the boundedness of $||Y||_{L^\infty}$ \eqref{eq: L^inf boundedness of Y nabla X}, one has
    \begin{align*}
        &\quad \ (c_3 + c_5 + c_6) \int_\Omega \left( \frac{|\nabla Y| |\Delta Y|}{Y} + \frac{|\nabla Y|^2}{Y^{3/2}} + \frac{|\nabla Y|^3}{Y^2} + \frac{|\Delta Y|}{Y^{1/2}} \right)  \\
        &\leq \frac{c}{8} \int_\Omega \frac{|\Delta Y|^2}{Y}  + \frac{2 (c_3 + c_5 + c_6)^2}{c} \int_\Omega \frac{|\nabla Y|^2}{Y} + (c_3 + c_5 + c_6) \int_\Omega \left( \frac{|\nabla Y|^2}{Y^{3/2}} + \frac{|\nabla Y|^3}{Y^2} + \frac{|\Delta Y|}{Y^{1/2}} \right)  \\
        &\leq \frac{c}{8} \int_\Omega \frac{|D^2 Y|^2}{Y}  + \frac{c}{8} \int_\Omega \frac{|\nabla Y|^4}{Y^3}  + c_7
    \end{align*}
    and
    \begin{align*}
        &\quad \ (c_3 + c_5) \int_\Omega \left( \frac{|\nabla Y|^2}{Y^{3/2}} + \frac{|\Delta Y|}{Y^{1/2}} + 1 \right) |\Delta X|  \leq \frac{c}{8} \int_\Omega \left( \frac{|\nabla Y|^4}{Y^3} + \frac{|\Delta Y|^2}{Y} + 1 \right)  + \frac{2 (c_3 + c_5)^2}{c} \int_\Omega |\Delta X|^2 .
    \end{align*}
    Therefore,
    \begin{equation}
        \begin{aligned}
            &\quad \ \frac{1}{2} \left( \frac{\mathrm{d}}{\mathrm{d}t} + d_1 \right) \int_\Omega \left( \frac{|\nabla Y|^2}{Y} + \tau_2 |\nabla Y|^2 \right) + \frac{c}{4} \int_\Omega \left( \frac{|D^2 Y|^2}{Y} + \frac{|\nabla Y|^4}{Y^3} \right)  \\
            &\leq \frac{2 (c_3 + c_5)^2}{c} \int_\Omega |\Delta X|^2  - \mu_2 \int_\Omega \nabla Y \nabla Z  + c_7 + \frac{c}{8} |\Omega|,
        \end{aligned}
    \end{equation}
    which immediately implies \eqref{eq: new energy estimate on Y}.
\end{proof}

\begin{lemma}
    \label{lem: energy estimate on Z}
    For any $\epsilon > 0$, the solution for equation \eqref{eq: general model} -- \eqref{eq: bc ic} satisfies
    \begin{equation}
        \label{eq: energy estimate on Z}
        \begin{aligned}
            &\quad \ \left( \frac{\mathrm{d}}{\mathrm{d}t} + d_1 \right) \int_\Omega \frac{1}{\chi_2} \mathcal{H}(Z) + \frac{\delta_Z}{2} \int_\Omega \frac{1}{\chi_2} \frac{|\nabla Z|^2}{h(Z)} \leq \epsilon \int_\Omega \frac{|\nabla Y|^4}{Y^3}  + \int_\Omega \nabla Y \nabla Z  + C_{10}, \quad \forall t \in [0, T_{\max}).
        \end{aligned}
    \end{equation}
\end{lemma}

\begin{proof}
Multiplying the third equation of \eqref{eq: general model} by $\frac{1}{\chi_2(x)} H(Z)$, and repeating the process as in Lemma \ref{lem: energy estimate on Y} until \eqref{eq: proof: energy estimate on Y}, we obtain
\begin{equation*}
    \begin{aligned}
        \frac{\mathrm{d}}{\mathrm{d}t} \int_\Omega \frac{1}{\chi_2} \mathcal{H}(Z)
        & = \int_\Omega (e_2 g(Y,Z,x) - d_2 Z) \frac{1}{\chi_2} H(Z)  + \int_\Omega \chi_2 \nabla\left(\frac{1}{\chi_2}\right) h(Z) H(Z) \nabla Y  \\
        & \quad \ + \int_\Omega \nabla Y \nabla Z  + \delta_Z \int_\Omega \Delta\left(\frac{1}{\chi_2}\right) \mathcal{H}(Z)  - \delta_Z \int_\Omega \frac{|\nabla Z|^2}{h(Z)} \frac{1}{\chi_2} .
    \end{aligned}
\end{equation*}
Using Young's inequality, assumption ($A_\chi$), and boundedness of $Y$ \eqref{eq: L^inf boundedness of Y nabla X}, for any $\epsilon > 0$ we obtain
\begin{equation*}
    \begin{aligned}
        &\quad \left( \frac{\mathrm{d}}{\mathrm{d}t} + d_1 \right)\int_\Omega \frac{1}{\chi_2} \mathcal{H}(Z) + \delta_Z \int_\Omega \frac{1}{\chi_2} \frac{|\nabla Z|^2}{h(Z)}  \\
        & \leq c_1 \int_\Omega (g(Y,Z,x) + Z) |H(Z)|  + c_1 \int_\Omega |h(Z) H(Z) \nabla Y|  + \int_\Omega \nabla Y \nabla Z  + c_1 \int_\Omega \mathcal{H}(Z)  \\
        & \leq \epsilon C_{7}^{-3} \int_\Omega |\nabla Y|^4  + c_2 \int_\Omega \left( Z |H(Z)| + |h(Z) H(Z)|^{4/3} + \mathcal{H} (Z) \right) + \int_\Omega \nabla Y \nabla Z  .
    \end{aligned}
\end{equation*}
Now Lemma \ref{lemma: G-N inequality on psi} and boundedness of $||Z||_{L^1}$ \eqref{eq: Y, Z L1 bounded} imply
\begin{equation*}
    \begin{aligned}
        & \quad \left( \frac{\mathrm{d}}{\mathrm{d}t} + d_1 \right) \int_\Omega \frac{1}{\chi_2} \mathcal{H}(Z) + \delta_Z \int_\Omega \frac{1}{\chi_2} \frac{|\nabla Z|^2}{h(Z)}  \\
        & \leq \epsilon C_7^{-3} \int_\Omega |\nabla Y|^4  + \frac{\delta_Z}{2 \chi_{2M}} \int_\Omega \frac{|\nabla Z|^2}{h(Z)}  + \int_\Omega \nabla Y \nabla Z  + c_3 \\
        & \leq \epsilon \int_\Omega \frac{|\nabla Y|^4}{Y^3}  + \frac{\delta_Z}{2} \int_\Omega \frac{1}{\chi_2} \frac{|\nabla Z|^2}{h(Z)}  + \int_\Omega \nabla Y \nabla Z  + c_3,
    \end{aligned}
\end{equation*}
which implies equation \eqref{eq: energy estimate on Z}.
\end{proof}

\begin{lemma}
    \label{lem: nabla Y L2 Z L2-p2}
    For any $0 < \epsilon \leq \frac{d_1}{2}$, the solution for equation \eqref{eq: general model} -- \eqref{eq: bc ic} satisfies
    \begin{equation}
    \label{eq: nabla Y L2 Z L2-p2}
        ||\nabla Y||_{L^2}, \ ||Z||_{L^{2 - \alpha}}, \ \left\| \int_\Omega \frac{1}{\chi_2} \frac{|\nabla Z|^2}{h(Z)}  \right\|_{L^1_{2\epsilon} (0, t)} \leq C_{11}, \quad \forall t \in [0, T_{\max}).
    \end{equation}
\end{lemma}
    
\begin{proof}
By adding up \eqref{eq: new energy estimate on Y} and \eqref{eq: energy estimate on Z}, we have
    \begin{equation*}
        \begin{aligned}
            & \left( \frac{\mathrm{d}}{\mathrm{d}t} + d_1 \right) \int_\Omega \left( \frac{|\nabla Y|^2}{Y} + \tau_2 |\nabla Y|^2 + \frac{2 \mu_2}{\chi_2} \mathcal{H}(Z) \right) + \nu \int_\Omega \left( \frac{|D^2 Y|^2}{Y} + \frac{|\nabla Y|^4}{Y^3} \right) \\
            & + 2\mu_2 \epsilon \int_\Omega \frac{|\nabla Y|^4}{Y^3} + \delta_Z \mu_2 \int_\Omega \frac{1}{\chi_2} \frac{|\nabla Z|^2}{h(Z)}  \leq C_{9}\int_\Omega |\Delta X|^2  + C_{9}+2\mu_2C_{10}.
        \end{aligned}
    \end{equation*}
    Note that $\nu > 0$ is a fixed constant and $\epsilon > 0$ is arbitrary. Choosing $0 < \epsilon \leq \frac{\nu}{2 \mu_2}$ and $0 < \epsilon' \leq \frac{d_1}{2}$ yields
    \begin{equation*}
        \begin{aligned}
            & \quad \left( \frac{\mathrm{d}}{\mathrm{d}t} + 2\epsilon' \right) \int_\Omega \left( \frac{|\nabla Y|^2}{Y} + \tau_2 |\nabla Y|^2 + \frac{2 \mu_2}{\chi_2} \mathcal{H}(Z) \right) + \delta_Z \mu_2 \int_\Omega \frac{1}{\chi_2} \frac{|\nabla Z|^2}{h(Z)} \leq C_{9}\int_\Omega |\Delta X|^2  + C_{9}+2\mu_2C_{10}.
        \end{aligned}
    \end{equation*}
    We use variance of constants and the estimate of $||\Delta X||_{L^2(\Omega)L_{\epsilon'}^2(0,t)}$ \eqref{eq: Delta X Lp Lq_eps bounded} to derive that
    \begin{equation*}
        \begin{aligned}
             &\quad \ \int_\Omega \left( \frac{|\nabla Y|^2}{Y} + \tau_2 |\nabla Y|^2 + \frac{2 \mu_2}{\chi_2} \mathcal{H}(Z) \right)  + \delta_Z \mu_2 \int_0^t e^{-2 \epsilon' (t - s)} \left( \int_\Omega \frac{1}{\chi_2} \frac{|\nabla Z|^2}{h(Z)} \right)  \mathrm{d}s \\
             &\leq c_1 e^{-2\epsilon' t} + \int_0^t (C_{9} + 2 \mu_2 C_{10}) e^{-2\epsilon' (t-s)} \mathrm{d}s + C_{9} ||\Delta X||_{L^2(\Omega)L_{\epsilon'}^2(0,t)}^2 \\
             &\leq c_1 + \frac{C_{9} + 2\mu_2 C_{10}}{2\epsilon'} + C_8^2.
        \end{aligned}
    \end{equation*}
    By Lemma \ref{lemma: norms of h(psi)}, we have
    \begin{equation*}
        \int_\Omega \frac{2\mu_2}{\chi_2} \mathcal{H} (Z)  
        \geq \frac{2\mu_2}{\chi_{2M}} \int_\Omega \mathcal{H} (Z)  
        \geq c_2 ||Z||_{L^{2 - \alpha}}^{2 - \alpha}.
    \end{equation*}
    It follows that
    \begin{equation*}
        \tau_{2m} ||\nabla Y||_{L^2}^2 + c_2 ||Z||_{L^{2 - \alpha}}^{2 - \alpha} + \delta_Z \mu_2 \left\| \int_\Omega \frac{1}{\chi_2} \frac{|\nabla Z|^2}{h(Z)}  \right\|_{L^1_{2\epsilon'} (0, t)}
        \leq c_1 + \frac{C_{9} + 2\mu_2 C_{10}}{2\epsilon'} + C_8^2,
    \end{equation*}
    which implies \eqref{eq: nabla Y L2 Z L2-p2}.
\end{proof}

\begin{lemma}
    \label{lemma: L^qk boundedness of nabla Y}
    Let $0 \leq k \leq k^*$ and $p_k \in (1, \infty)$. Suppose for any $p \in (p_k, \infty)$ and $\epsilon \in (0, \delta_Y \lambda + d_1)$, the solution for equation \eqref{eq: general model} -- \eqref{eq: bc ic} satisfies 
    \begin{equation}
    \label{eq: assumption L^qk boundedness of Z for nabla Y} 
        ||\nabla Y||_{L^{q_k}(\Omega) L^p_\epsilon(0,t)}, \ ||Z||_{L^{q_k}(\Omega) L_\epsilon^p(0,t)} \leq C_{12, k}, \quad \forall t \in [0, T_{\max}).
    \end{equation}
    Then for any $\overline{q}_{k + 1} \in [1, \infty]$ satisfying \eqref{eq: bar q_k+1}, any $p \in (p_k, \infty)$ and $\epsilon \in (0, \delta_Y \lambda + d_1)$, we have
    \begin{equation}
        \label{eq: L^qk boundedness of nabla Y}
        ||\nabla Y||_{L^{\overline{q}_{k + 1}}(\Omega) L^p_\epsilon(0,t)} \leq C_{12, k + 1}, \quad \forall t \in [0, T_{\max}).
    \end{equation}
    In particular, we can set $\overline{q}_{k + 1} = q_{k + 1}$.
\end{lemma}

\begin{proof}
    Taking gradient to each side of the first line of \eqref{eq: variation of constants Y}, we get
    \begin{equation*}
        \begin{aligned}
            \nabla Y(t) & = \nabla e^{(\delta_Y  \Delta - d_1) t} Y_0 + \int_0^t \nabla e^{(\delta_Y \Delta -d_1)(t-s)} \left(e_1 f(X,Y,x) - g(Y,Z,x) - \nabla ( \chi_1(x) h(Y) \nabla X) \right) \mathrm{d}s.
        \end{aligned}
    \end{equation*}
    Applying \eqref{eq: lemma Lp Lq estimate 1} and \eqref{eq: lemma Lp Lq estimate 2}, we derive
    \begin{equation*}
    \begin{aligned}
        ||\nabla Y||_{L^{{\overline{q}}_{k + 1}}} 
        & \leq c_1 e^{-(\delta_Y \lambda + d_1) t}||\nabla Y_0||_{L^{{\overline{q}}_{k + 1}}} + c_1 \int_0^t \left( 1 + (t-s)^{-\frac{1}{2} - \rho} \right) e^{- (\delta_Y \lambda + d_1)(t-s)} \\
        & \quad \Big( e_1||f(X,Y,x) ||_{L^{q_k}} + || g(Y,Z,x) ||_{L^{q_k}} + ||\nabla ( \chi_1(x) h(Y) \nabla X)||_{L^{q_k}}\Big) \mathrm{d}s
    \end{aligned}
    \end{equation*}
    for some $c_1 > 0$, and $\rho \in [0, \frac{1}{2})$ is given in \eqref{eq: rho}. Since $\epsilon \in (0, \delta_Y \lambda + d_1)$, it follows from Propositions \ref{prop: Young's inequality on L^p_epsilon} and \ref{prop: comparison of different norms} that
    \begin{equation}
    \label{eq: proof 1 nabla Y}
    \begin{aligned}
        & \quad \ ||\nabla Y||_{L^{{\overline{q}}_{k + 1}}(\Omega) L^p_\epsilon(0,t)}\\ 
        & \leq c_1 ||e^{-(\delta_Y \lambda + d_1) s}||_{L^p_\epsilon (0, t)} ||\nabla Y_0||_{L^{{\overline{q}}_{k + 1}}(\Omega)}  + c_1 \Big\| T_{\frac{1}{2}+ \rho, \delta_Y \lambda + d_1} \big( e_1||f(X,Y,x) ||_{L^{q_k}} \\
        & \quad + || g(Y,Z,x) ||_{L^{q_k}} + ||\nabla ( \chi_1(x) h(Y) \nabla X)||_{L^{q_k}} \big) \Big\|_{L^p_\epsilon(0,t)} \\
        & \leq c_1 ||e^{-(\delta_Y \lambda + d_1) s}||_{L^p_\epsilon (0, t)} ||\nabla Y_0||_{L^{{\overline{q}}_{k + 1}}(\Omega)} + c_2 \Big\| e_1||f(X,Y,x) ||_{L^{q_k}} \\
        & \quad + || g(Y,Z,x) ||_{L^{q_k}} + ||\nabla ( \chi_1(x) h(Y) \nabla X)||_{L^{q_k}} \Big\|_{L^p_\epsilon(0,t)} \\
        & \leq c_1 c_3 ||e^{-(\delta_Y \lambda + d_1) s}||_{L^\infty (0, t)} ||\nabla Y_0||_{L^{{\overline{q}}_{k + 1}}(\Omega)} + c_2 c_3 e_1 ||f(X,Y,x)||_{L^{q_k} (\Omega) L^\infty (0, t)} \\
        & \quad + c_2 || g(Y,Z,x) ||_{L^{q_k} (\Omega) L^p_\epsilon (0,t)} + c_2 || \nabla (\chi_1(x) h(Y) \nabla X) ||_{L^{q_k} (\Omega) L^p_\epsilon (0,t)}.
        \end{aligned}
    \end{equation}
    By boundedness of $\| X \|_{L^{\infty}}$ \eqref{eq: X bounded}, $\| Y \|_{L^{\infty}}$ \eqref{eq: L^inf boundedness of Y nabla X} and $||Z||_{L^{q_k} (\Omega) L^p_\epsilon (0,t)}$ \eqref{eq: assumption L^qk boundedness of Z for nabla Y} we obtain
    \begin{equation*}
    \begin{aligned}
        || f(X,Y,x) ||_{L^{q_k} (\Omega) L^\infty (0, t)} 
        &\leq \mu_1 ||X||_{L^{\infty} (\Omega) L^\infty (0, t)} ||Y||_{L^{\infty} (\Omega) L^\infty (0, t)} |\Omega|^{\frac{1}{q_k}}
      \leq \mu_1 \overline{X} C_{7} |\Omega|^{\frac{1}{q_k}}
       \end{aligned}
    \end{equation*}
    and
    \begin{equation*}
        || g(Y,Z,x) ||_{L^{q_k} (\Omega) L^p_\epsilon (0,t)} 
        \leq \mu_2 ||Y||_{L^{\infty} (\Omega) L^\infty (0,t)} ||Z||_{L^{q_k} (\Omega) L^p_\epsilon (0,t)}
        \leq \mu_2 C_7 C_{12, k}.
    \end{equation*}
    Moreover, since $Y$ is bounded and $h(z) = O(z)$, $h'(z) = O(1)$ as $z \to 0$, we know that $h(Y)$ and $h'(Y)$ are also bounded. Together with Proposition \ref{prop: comparison of different norms} and the boundedness of $||\nabla X||_{L^\infty}$ \eqref{eq: L^inf boundedness of Y nabla X},  $||\Delta X||_{L^{q_k} (\Omega) L^p_\epsilon (0,t)}$ \eqref{eq: Delta X Lp Lq_eps bounded} and $||\nabla Y||_{L^{q_k} (\Omega) L^p_\epsilon (0,t)}$ \eqref{eq: assumption L^qk boundedness of Z for nabla Y}, we have
    \begin{equation}
    \label{eq: proof 2 nabla Y}
        \begin{aligned}
            & \quad \ ||\nabla (\chi_1(x) h(Y) \nabla X) ||_{L^{q_k} (\Omega) L^p_\epsilon (0,t)} \\
            & \leq c_4 ||\nabla X ||_{L^{q_k} (\Omega) L^p_\epsilon (0,t)} + c_4 ||\nabla X \nabla Y||_{L^{q_k} (\Omega) L^p_\epsilon (0,t)} + c_4 ||\Delta X||_{L^{q_k} (\Omega) L^p_\epsilon (0,t)}\\
            & \leq c_3 c_4 ||\nabla X ||_{L^\infty (\Omega) L^\infty (0, t)} |\Omega|^{1/{q_k}} + c_4 ||\nabla X||_{L^\infty (\Omega) L^\infty (0, t)} ||\nabla Y||_{L^{q_k} (\Omega) L^p_\epsilon (0,t)} + c_4 ||\Delta X ||_{L^{q_k} (\Omega) L^p_\epsilon (0,t)} \\
            & \leq c_3 c_4 C_7 |\Omega|^{1/{q_k}} + c_3 c_4 C_7 C_{12, k} + c_4 C_8.
        \end{aligned}
    \end{equation}
    Combining equations \eqref{eq: proof 1 nabla Y}-\eqref{eq: proof 2 nabla Y}, we obtain
    \begin{equation*}
        ||\nabla Y||_{L^{\overline{q}_{k+1}}(\Omega)L^p_\epsilon(0,t)} 
        \leq c_1 c_3 ||\nabla Y_0||_{L^{q_{k+1}}} + c_5 =: C_{12, k+1}.
    \end{equation*}
\end{proof}

\begin{lemma}
    \label{lem: L^qk boundedness of Z}
    Let $0 \leq k \leq k^*$ and $p_k \in (1, \infty)$. Suppose for any $p \in (p_k, \infty)$ and $0 < \epsilon < \min\{ \delta_Y \lambda + d_1, \delta_Z \lambda + d_2 \}$, the solution for equation \eqref{eq: general model} -- \eqref{eq: bc ic} satisfies 
    \begin{equation}
        \label{eq: assumption L^qk boundedness of Z}
        ||\nabla Y||_{L^{q_k}(\Omega) L^p_\epsilon(0,t)}, \ ||Z||_{L^{q_k}(\Omega) L^p_\epsilon(0,t)} \leq C_{13, k}, \quad \forall t \in [0, T_{\max}).
    \end{equation}
    Then there exists $p_{k+1} \in [p_k, \infty)$, so that for any $p \in (p_{k+1}, \infty)$ and $0 < \epsilon < \min\{ \delta_Y \lambda + d_1, \delta_Z \lambda + d_2 \}$, we have
    \begin{equation}
        \label{eq: L^qk boundedness of Z}
        ||Z||_{L^{q_{k + 1}}(\Omega) L^p_\epsilon(0,t)} \leq C_{13, k + 1}, \quad \forall t \in [0, T_{\max}).
    \end{equation}
\end{lemma}

\begin{proof}
     We consider $k < k^*$ first. Applying the variation of constants formula to the third equation in \eqref{eq: general model}, we obtain
     \begin{equation}
     \begin{aligned}
     \label{eq: variation of formula Z}
         Z(t) & = e^{(\delta_{Z}\Delta -d_2)t}Z_{0} + \int_0^t e^{(\delta_{Z}\Delta - d_2) (t-s)} \big(e_2 g(Y,Z,x)-\nabla \left(\chi_2(x) h(Z)\nabla Y\right)\big) \mathrm{d}s.
     \end{aligned}
     \end{equation}
    Choose a number $q$ satisfying \eqref{eq: proof 1 q_k+1}. By applying inequalities \eqref{eq: lemma Lp Lq estimate 3} and \eqref{eq: lemma Lp Lq estimate 4}, we obtain
    \begin{equation*}
    \begin{aligned}
         ||Z(t)||_{L^{q_{k+1}}} 
         & \leq c_1 e^{-(\delta_{Y} \lambda + d_2) t} ||Z_0||_{L^{q_{k+1}}} \\
         & \quad + c_1 \int_0^t \left( 1 + (t-s)^{-\rho_1} \right)e^{-(\delta_{Z} \lambda + d_2) (t-s)}  ||g(Y,Z,x)||_{L^{q_k}} \mathrm{d}s\\
         & \quad + c_1 \int_0^t \left(1+ (t-s)^{- \frac{1}{2} - \rho_2} \right)e^{-(\delta_{Z} \lambda + d_2) (t-s)}||\chi_2(x) h(Z)\nabla Y||_{L^q} \mathrm{d}s,
    \end{aligned}
    \end{equation*}
    where $\rho_1 \in (0, \frac{1}{2})$ is given in \eqref{eq: rho_1}, and $\rho_2 \in [0, \frac{1}{2})$ is given in \eqref{eq: rho_2}. It follows from Propositions \ref{prop: Young's inequality on L^p_epsilon} and \ref{prop: comparison of different norms} that, if $\epsilon < \delta_Z \lambda + d_2$ and $p > \frac{\alpha}{1/2 - \rho_2}$, then
\begin{equation}
\label{eq: Z q_k+1 L epsilon p}
    \begin{aligned}
         ||Z(t)||_{L^{q_{k+1}}(\Omega)L^p_\epsilon(0,t)}
        & \leq c_1 \left\|e^{-(\delta_{Z} \lambda + d_2) t} \right\|_{L_\epsilon^p(0,t)}||Z_0||_{L^{q_{k+1}}} + c_1 \Big\| T_{\rho_1,\delta_{Z} \lambda + d_2} ||g(Y,Z,x)||_{L^{q_k}} \Big\|_{L^p_\epsilon(0,t)} \\
        & \quad + c_1 \left\|T_{\frac{1}{2}+\rho_2,\delta_Z\lambda+d_2}||\chi_2(x) h(Z)\nabla Y||_{L^{q}} \right\|_{L^p_\epsilon (0, t)}\\
        & \leq c_1 c_3 \left\|e^{-(\delta_{Z} \lambda + d_2) t} \right\|_{L_\epsilon^\infty(0,t)}||Z_0||_{L^{q_{k+1}}} + c_2 ||g(Y,Z,x)||_{L^{q_{k}}(\Omega)L^p_\epsilon(0,t)} \\
        & \quad + c_2 \left\|\chi_2(x) h(Z)\nabla Y \right\|_{L^q(\Omega) L^{\frac{p}{1 + \alpha}}_\epsilon(0,t)}.
    \end{aligned}
\end{equation}
    We choose $p_{k+1} = \max \left\{ p_k, \frac{\alpha}{1/2 - \rho_2} \right\}$ and let $p > p_{k+1}$ afterwards. By the boundedness of $||Y||_{L^\infty}$ \eqref{eq: L^inf boundedness of Y nabla X} and $||Z||_{L^{q_k}(\Omega)L^p_\epsilon(0,t)}$ \eqref{eq: assumption L^qk boundedness of Z}, we have
    \begin{equation*}
        ||g(Y,Z,x)||_{L^{q_k}(\Omega)L^p_\epsilon(0,t)} \leq \mu_2 ||Y||_{L^\infty}||Z||_{L^{q_k}(\Omega)L^p_\epsilon(0,t)} \leq \mu_2 C_7 C_{13,k}.
    \end{equation*}
Let $\overline{q}_{k+1}$ satisfy \eqref{eq: proof 2 bar_q_k+1}, then by the same argument as \eqref{eq: proof q_k+1 bar satisfies condition}, $\overline{q}_{k+1}$ satisfies condition \eqref{eq: bar q_k+1}. Therefore, we obtain by H\"older's inequality and Lemma \ref{lemma: norms of h(psi)} that
    \begin{equation}
    \label{eq: proof 1 L^qk boundedness of Z}
        \begin{aligned}
            ||\chi_2(x) h(Z) \nabla Y||_{L^q} 
            & \leq \chi_{2M} \| \nabla Y \|_{L^{\overline{q}_{k+1}}} \left\| h(Z) \right\|_{L^{q_k / \alpha}} \\
            & \leq c_4 \chi_{2M} \| \nabla Y \|_{L^{\overline{q}_{k+1}}} \left( \| Z \|_{L^{q_k}}^{q_k} + 1 \right)^{\alpha / q_k} \\
            & \leq c_5 \| \nabla Y \|_{L^{\overline{q}_{k+1}}} \left( \| Z \|_{L^{q_k}}^{\alpha} + 1 \right).
        \end{aligned}
    \end{equation}  
    It follows from Remark \ref{rem: holder for L_p_epsilon}, boundedness of $||\nabla Y||_{L^{\overline{q}_{k + 1}}(\Omega)L^p_\epsilon(0,t)}$ (Lemma \ref{lemma: L^qk boundedness of nabla Y}), and assumption \eqref{eq: assumption L^qk boundedness of Z} that
    \begin{equation}
    \label{eq: proof 2 L^qk boundedness of Z}
        \begin{aligned}
            & \quad \ \|\chi_2(x) h(Z) \nabla Y\|_{L^q(\Omega) L^{\frac{p}{1 + \alpha}}_\epsilon (0,t)} \\
            & \leq c_5 \| \nabla Y \|_{L^{\overline{q}_{k+1}} (\Omega) L^p_{(1 - \alpha) \epsilon} (0,t)} \big\| \| Z \|_{L^{q_k}}^{\alpha} + 1 \big\|_{L^{p / \alpha}_{\alpha \epsilon} (0,t)} \\
            & \leq c_5 \| \nabla Y \|_{L^{\overline{q}_{k+1}} (\Omega) L^p_{(1 - \alpha) \epsilon} (0,t)} \left( \| Z \|_{L^{q_k} (\Omega) L^p_{\epsilon} (0,t)}^{\alpha} + \| 1 \|_{L^{p / \alpha}_{\alpha \epsilon} (0, t)} \right) \\
            & \leq c_5 C_{12} \left( C_{13, k}^{\alpha} + c_6 \right)=c_7.
        \end{aligned}
    \end{equation} 
Combining equations \eqref{eq: Z q_k+1 L epsilon p}-\eqref{eq: proof 1 L^qk boundedness of Z}, we obtain
 \begin{equation*}
    ||Z(t)||_{L^{q_{k+1}}(\Omega)L^p_\epsilon(0,t)}
    \leq c_1 c_3 ||Z_0||_{L^{q_{k+1}}} + c_2 \mu_2 C_7 C_{13,k} + c_2 c_7 
    =: C_{13, k+1}.
\end{equation*}
When $k = k^*$, we set $q = q_{k^*} > n$, then the corresponding $\rho_1$ and $\rho_2$ are still less than $\frac{1}{2}$. Besides, we set $\overline{q}_{k^*+1} = \infty$, which satisfies \eqref{eq: bar q_k+1}. Now we can repeat the rest of the proof, replacing only the estimate \eqref{eq: proof 1 L^qk boundedness of Z} by
\begin{align*}
    ||\chi_2(x) h(Z) \nabla Y||_{L^q} 
    & \leq \chi_{2M} \| \nabla Y \|_{L^{\overline{q}_{k^* + 1}}} \left\| h(Z) \right\|_{L^{q_{k^*}}} \\
    & \leq \chi_{2M} |\Omega|^{(1 - \alpha) / q_{k^*}} \| \nabla Y \|_{L^{\overline{q}_{k^* + 1}}} \left\| h(Z) \right\|_{L^{q_{k^*} / \alpha}} \\
    & \leq c_4 \chi_{2M} |\Omega|^{(1 - \alpha) / q_{k^*}} \| \nabla Y \|_{L^{\overline{q}_{k^* + 1}}} \left( \| Z \|_{L^{q_{k^*}}}^{q_{k^*}} + 1 \right)^{\alpha / q_{k^*}} \\
    & \leq c_5' \| \nabla Y \|_{L^{\overline{q}_{k^* + 1}}} \left( \| Z \|_{L^{q_{k^*}}}^{\alpha} + 1 \right).
\end{align*}
\end{proof}

\begin{lemma}
    \label{lem: L^inf boundedness of Z nabla Y}    
    The solution for equation \eqref{eq: general model} -- \eqref{eq: bc ic} satisfies
    \begin{equation}
        \label{eq: L^inf boundedness of Z nabla Y}
        ||Z||_{L^\infty}, \ ||\nabla Y||_{L^\infty} \leq C_{14}, \quad \forall t \in [0, T_{\max}).
    \end{equation}
\end{lemma}

\begin{proof}    
    Lemma \ref{lem: nabla Y L2 Z L2-p2} and Proposition \ref{prop: comparison of different norms} imply that $||Z||_{L^{q_0} (\Omega) L_\epsilon^p (0,t)}$ and \\$||\nabla Y||_{L^{q_0} (\Omega) L_\epsilon^p (0,t)}$ are bounded for every $p \in (1, \infty)$ and $\epsilon > 0$. By iterating Lemmas \ref{lemma: L^qk boundedness of nabla Y} and \ref{lem: L^qk boundedness of Z}, eventually we obtain the boundedness of $||Z||_{L^\infty (\Omega) L_\epsilon^p (0,t)}$ and $||\nabla Y||_{L^\infty (\Omega) L_\epsilon^p (0,t)}$ when $p$ is sufficiently large and $\epsilon$ is sufficiently small. Now we can repeat these lemmas again, in which all norms on $\Omega$ are replaced by $L^\infty (\Omega)$, and the norms on the left-hand sides of equations \eqref{eq: proof 1 nabla Y} and \eqref{eq: Z q_k+1 L epsilon p} are replaced by $L^\infty (\Omega) L^\infty_\epsilon (0, t)$. This yields the uniform boundedness of $||\nabla Y||_{L^\infty (\Omega) L^\infty_\epsilon (0, t)}$ and $||Z||_{L^\infty (\Omega) L^\infty_\epsilon (0, t)}$ with respect to $t$. Therefore, Lemma \ref{lem: L^inf boundedness of Z nabla Y} holds as a consequence of Proposition \ref{prop: L epsilon infinity norm}.
\end{proof}

\begin{proof}[Proof of Theorem \ref{th: global existence}]
Lemmas \ref{lem: X bounded}, \ref{lem: L^inf boundedness of Z nabla Y} and Corollary \ref{cor: L^inf boundedness of Y nabla X} guarantee that the solution to equation \eqref{eq: general model} -- \eqref{eq: bc ic} does not blow up under $L^\infty$ norm in finite time. Thanks to Lemma \ref{th: local existence}, we now obtain the global existence and boundedness of the solution.
\end{proof}

\section{Stability}
\label{sec: Stability}

\subsection{$W^{1,\infty}$ boundedness of U}

In this section, we provide some stability results to the model \eqref{eq: general model}. The proof of stability requires further a priori estimates, which are established in the following lemmas.

\begin{lemma}
    \label{lem: Delta Y Lp Lq_eps bounded}
    For all $1 < p, q < \infty$ and $\epsilon > 0$, the solution for equation \eqref{eq: general model} -- \eqref{eq: bc ic} satisfies
    \begin{equation}
        \label{eq: Delta Y Lp Lq_eps bounded}
        ||\Delta Y||_{L^p (\Omega) L^q_\epsilon (0, t)} \leq C_{15}, \quad \forall t \geq 0.
    \end{equation}
\end{lemma}

\begin{proof}
    Let $\widetilde{Y}(t) = e^{\epsilon t} (Y(t) - Y_0)$, then we have
    \begin{equation*}
        \widetilde{Y}_t = \delta_Y \Delta \widetilde{Y} - d_1 \widetilde{Y} + F(x, t),
    \end{equation*}
    where
    \begin{equation*}
        F(x, t) = e^{\epsilon t} \big( \delta_Y \Delta Y_0 + \epsilon (Y - Y_0) + e_1 f(X,Y,x) - g(Y,Z,x) - d_1 Y_0 - \nabla(\chi_1(x) h(Y) \nabla X) \big).
    \end{equation*}
    By maximal regularity property (Theorem 3.1, \cite{matthias1997heat}), for any $t \geq 0$ we have
    \begin{equation*}
    \begin{aligned}
        \int_0^t \left\| (\delta_Y \Delta - d_1) \widetilde{Y} \right\|_{L^p}^q \mathrm{d}s 
        \leq c_1 \int_0^t \left\| F(\cdot, s) \right\|_{L^p}^q \mathrm{d}s.
    \end{aligned}
    \end{equation*}
    By Young's inequality and boundedness of $Y$, $\nabla X$ \eqref{eq: L^inf boundedness of Y nabla X}, $\nabla Y$ \eqref{eq: L^inf boundedness of Z nabla Y} and $h(Y)$, $h'(Y)$, together with equations \eqref{lemma: G-N inequality on psi}, we have
    \begin{align*}
        & \quad \ ||\nabla (\chi_1(x) h(Y) \nabla X) ||_{L^p} \\
        & \leq ||\nabla \chi_1(x) h(Y) \nabla X ||_{L^p} + ||\chi_1(x) h'(Y)\nabla Y \nabla X ||_{L^p} + ||\chi_1(x) h(Y) \Delta X ||_{L^p}\\
        & \leq c_2 \left(||h(Y)||_{L^\infty}||\nabla X ||_{L^\infty}|\Omega|^{1/p} + \left\|h'(Y)\right\|_{L^\infty}||\nabla X ||_{L^\infty} ||\nabla Y ||_{L^\infty}|\Omega|^{1/p}  + || h(Y)||_{L^\infty} ||\Delta X ||_{L^p}\right)\\
        & \leq c_3 C_7 |\Omega|^{1/p} + c_3 C_7 C_{14} |\Omega|^{1/p} + c_3||\Delta X ||_{L^p}
    \end{align*}
    for some constants $c_2$ and $c_3$.
    Furthermore, the boundedness of $Y$ \eqref{eq: L^inf boundedness of Y nabla X} and $Z$ \eqref{eq: L^inf boundedness of Z nabla Y} implies
    \begin{align*}
        \int_0^t \| \Delta \widetilde{Y} \|_{L^p}^q \mathrm{d}s 
        & \leq c_4 \int_0^t \| \widetilde{Y} \|_{L^p}^q \mathrm{d}s + c_4 \int_0^t \left\| F(\cdot, s) \right\|_{L^p}^q \mathrm{d}s \\
        & \leq c_5 \int_0^t e^{\epsilon q s} \left( ||\Delta Y_0||_{L^p}^q + ||Y||_{L^p}^q + ||Y_0||_{L^p}^q + e_1^q \mu_1^q \overline{X}^q ||Y||_{L^p}^q + \mu_2^q ||Y||_{L^\infty}^q ||Z||_{L^p}^q \right) \mathrm{d}s\\
        & \quad + c_5 \int_0^t e^{\epsilon q s} ||\nabla (\chi_1(x) h(Y) \nabla X) ||_{L^p}^q \mathrm{d}s \\
        & \leq \frac{c_5}{\epsilon q} e^{\epsilon q t} \Big( ||\Delta Y_0||_{L^p}^q + ||Y_0||_{L^p}^q + c_6 \Big) + c_3 c_5 e^{\epsilon q t}||\Delta X||_{L^p(\Omega)L^q_\epsilon(0,t)}^q.
    \end{align*}
    Finally, by the boundedness of $||\Delta X||_{L^p(\Omega)L^q_\epsilon(0,t)}$ \eqref{eq: Delta X Lp Lq_eps bounded}, we have 
    \begin{align*}
        ||\Delta Y||_{L^p (\Omega) L^q_\epsilon (0, t)}^q 
        & = e^{-\epsilon q t} \int_0^t \| e^{\epsilon s} \Delta Y \|_{L^p}^q \mathrm{d}s \\
        & = e^{-\epsilon q t} \int_0^t \left(\| \Delta \widetilde{Y} \|_{L^p}^q + e^{\epsilon q s} \| \Delta Y_0 \|_{L^p}^q \right) \mathrm{d}s\\
        & \leq 2^{q-1} \frac{c_5 + 1}{\epsilon q} \Big( ||\Delta X_0||_{L^p}^q + ||X_0||_{L^p}^q + c_6 \Big) + c_3 c_5 C_8^q =: C_{15}^q.
    \end{align*}
\end{proof}

\begin{lemma}
    \label{lem: nabla Z L2 bounded}
    The solution for equation \eqref{eq: general model} -- \eqref{eq: bc ic} satisfies
    \begin{equation}
        \label{eq: nabla Z L2 bounded}
        ||\nabla Z||_{L^2} \leq C_{16}, \quad \forall t \geq 0.
    \end{equation}
\end{lemma}

\begin{proof}
    By the properties of $h$ and the boundedness of $Z$ \eqref{eq: L^inf boundedness of Z nabla Y}, we know that $h(Z)$ is bounded. Therefore, 
    \begin{equation*}
        \int_\Omega \frac{1}{\chi_2} \frac{|\nabla Z|^2}{h(Z)} 
        \geq \frac{1}{\chi_{2M} ||h(Z)||_{L^\infty}} \int_\Omega |\nabla Z|^2 .
    \end{equation*}
    By Lemma \ref{lem: nabla Y L2 Z L2-p2}, for all $0 < \epsilon \leq \frac{d_1}{2}$, we get
    \begin{equation*}
        C_{11} \geq \left\| \int_\Omega \frac{1}{\chi_2} \frac{|\nabla Z|^2}{h(Z)}  \right\|_{L^1_{2\epsilon} (0, t)}
        \geq \frac{ \left\| \| \nabla Z \|^2_{L^2} \right\|_{L^1_{2\epsilon} (0, t)} }{\chi_{2M} ||h(Z)||_{L^\infty}}
        = \frac{ \left\| \nabla Z \right\|_{L^2(\Omega) L^2_{\epsilon} (0, t)}^2 }{\chi_{2M} ||h(Z)||_{L^\infty}}.
    \end{equation*}
Therefore,
\begin{equation}
    \label{eq: nabla Z L2 L2 epsilon bounded}
    ||\nabla Z||_{L^2(\Omega) L^2_\epsilon (0, t)} \leq \left( \chi_{2M} ||h(Z)||_{L^\infty} C_{11} \right)^{1/2} =: c_1.
\end{equation}

Now we prove that $\| \nabla Z \|_{L^2 (\Omega) L^p_\epsilon (0, t)}$ is bounded for all $p \in [2, \infty)$ and sufficiently small $\epsilon$. Taking gradient on equation \eqref{eq: variation of formula Z}, we obtain
\begin{equation}
\label{eq: proof0: nabla Z L2 L epsilon p}
     \begin{aligned}
        \nabla Z(t) &= \nabla e^{(\delta_{Z}\Delta -d_2)t} Z_0   + \int_0^t \nabla e^{(\delta_{Z}\Delta - d_2) (t-s)} \left(e_2g(Y,Z,x) - \nabla (\chi_2(x)h(Z)\nabla Y) \right) \mathrm{d}s.
    \end{aligned}
\end{equation}
By \eqref{eq: lemma Lp Lq estimate 1} and \eqref{eq: lemma Lp Lq estimate 2}, there exists $c_2 > 0$ such that
\begin{equation*}
     \begin{aligned}
        ||\nabla Z||_{L^2} 
        & \leq c_2 e^{-(\delta_{Z}\lambda + d_2) t } || \nabla  Z_0 ||_{L^{2}} \\
        &\quad \ + c_2 e_2 \int_0^t \left( 1+ (t-s)^{-\frac{1}{2}}\right)  e^{-(\delta_{Z}\lambda + d_2) (t-s)}  \left\|g(Y,Z,x)\right\|_{L^2}  \mathrm{d}s \\
        &\quad \ + c_2 \int_0^t \left( 1+ (t-s)^{-\frac{1}{2}}\right)  e^{-(\delta_{Z}\lambda + d_2) (t-s)}\left\| \nabla (\chi_2(x)h(Z)\nabla Y) \right\|_{L^2} \mathrm{d}s.
    \end{aligned}
\end{equation*}
It follows that 
\begin{equation}
\label{eq: proof1: nabla Z L2 L epsilon p}
    \begin{aligned}
        ||\nabla Z||_{L^2(\Omega)L_\epsilon^p(0,t)} 
        \leq & \ c_2 \|e^{-(\delta_{Z}\lambda + d_2) t } \|_{L_\epsilon^p(0,t)}|| \nabla  Z_0 ||_{L^{2}} \\
        & + c_2 e_2 \left\| T_{\frac{1}{2}, \delta_Z \lambda + d_2} \left\| g(Y,Z,x) \right\|_{L^2} \right\|_{L^p_\epsilon(0,t)} \\
        & + c_2 \left\| T_{\frac{1}{2}, \delta_Z \lambda + d_2} \left\| \nabla (\chi_2(x) h(Z) \nabla Y) \right\|_{L^2} \right\|_{L^p_\epsilon(0,t)}.
    \end{aligned}
\end{equation}
By the boundedness of $Y$ \eqref{eq: L^inf boundedness of Y nabla X}, $Z$ \eqref{eq: L^inf boundedness of Z nabla Y} and Proposition \ref{prop: comparison of different norms}, the second term can be evaluated by
\begin{equation}
    \label{eq: proof2: nabla Z L2 L epsilon p}
    \begin{aligned}
        &\quad \ \big\| T_{\frac{1}{2}, \delta_Z \lambda + d_2} \left\| g(Y,Z,x) \right\|_{L^2} \big\|_{L^p_\epsilon(0,t)} \\
        & \leq c_3 \left\| T_{\frac{1}{2}, \delta_Z \lambda + d_2} \left\| g(Y,Z,x) \right\|_{L^2} \right\|_{L^\infty(0,t)}\\
        & \leq c_3 \mu_2 |\Omega|^{1/2} \left\| T_{\frac{1}{2}, \delta_Z \lambda + d_2} \big( \| Y \|_{L^\infty} \| Z \|_{L^\infty} \big) \right\|_{L^\infty(0,t)} \\
        & \leq c_3 c_4 \mu_2 C_7 C_{14} |\Omega|^{1/2} 
    \end{aligned}
\end{equation}
for some $c_3, c_4 > 0$. By Proposition \ref{prop: Young's inequality on L^p_epsilon}, for any $p \in [2, \infty)$ and $0 < \epsilon < \delta_Z \lambda + d_2$, there exists a constant $c_5 > 0$ so that
\begin{equation}
\label{eq: proof3: nabla Z L2 L epsilon p}
    \begin{aligned}
        \left\| T_{\frac{1}{2}, \delta_Z \lambda + d_2} \left\| \nabla (\chi_2(x) h(Z) \nabla Y) \right\|_{L^2} \right\|_{L^p_\epsilon(0,t)} 
        \leq c_5 \left\| \nabla (\chi_2(x)h(Z)\nabla Y) \right \|_{L^2(\Omega) L^2_\epsilon (0,t)}.
    \end{aligned}
\end{equation}
By the estimates of $||\nabla Y||_{L^\infty}$ \eqref{eq: L^inf boundedness of Z nabla Y}, $||\Delta Y||_{L^2(\Omega) L^{2}_\epsilon (0,t)}$ \eqref{eq: Delta Y Lp Lq_eps bounded}, $||\nabla Z||_{L^2(\Omega) L^{2}_\epsilon (0,t)}$ \eqref{eq: nabla Z L2 L2 epsilon bounded}, boundedness of $\chi_2$, $\nabla \chi_2$, $h(Z)$, $h'(Z)$ and Proposition \ref{prop: comparison of different norms}, we obtain that if $0 < \epsilon < \min \left\{ \frac{d_1}{2}, \delta_Z \lambda + d_2 \right\}$, then
\begin{equation}
    \label{eq: proof4: nabla Z L2 L epsilon p}
    \begin{aligned}
        &\quad \ \left\|\nabla (\chi_2(x)h(Z)\nabla Y) \right\|_{L^2(\Omega) L^{2}_\epsilon(0,t)} \\
        & \leq c_6 \left\| \nabla \chi_2(x) h(Z)\nabla Y \right\|_{L^2(\Omega) L^{2}_\epsilon(0,t)} 
        + \left\| \chi_2(x) h'(Z) \nabla Z \nabla Y \right\|_{L^2(\Omega) L^{2}_\epsilon(0,t)}
         + \left\| \chi_2(x) h(Z) \Delta Y \right\|_{L^2(\Omega) L^{2}_\epsilon(0,t)} \\
        & \leq c_7 \left( ||\nabla Y||_{L^2(\Omega) L^\infty (0,t)} + ||\nabla Y||_{L^\infty (\Omega) L^\infty (0, t)} ||\nabla Z||_{L^2(\Omega) L^2_\epsilon (0,t)} + ||\Delta Y||_{L^2(\Omega) L^2_\epsilon(0,t)} \right) \\
        & \leq c_7 \left( |\Omega|^{1/2} C_7 + c_1 C_{14} + C_{15} \right).
    \end{aligned}
\end{equation}
Finally, combing \eqref{eq: proof1: nabla Z L2 L epsilon p}-\eqref{eq: proof4: nabla Z L2 L epsilon p}, we obtain
\begin{equation*}
    \begin{aligned}
        \left\| \nabla Z \right\|_{L^{2}(\Omega) L^{p}_\epsilon (0,t)}
        & \leq c_8 || \nabla  Z_0 ||_{L^{2}} + c_{2}c_{3}c_{4} e_2\mu_2 C_7 C_{14} |\Omega|^{1/2} + c_{2} c_{5} c_7 \left( |\Omega|^{1/2} C_7 + c_1 C_{14} + C_{15} \right).
        \end{aligned}
    \end{equation*}

    If replace the pair of norms $(L^p_\epsilon (0, t), L^2_\epsilon (0, t))$ in \eqref{eq: proof3: nabla Z L2 L epsilon p} by $(L^\infty_\epsilon (0, t), L^p_\epsilon (0, t))$ and repeat the argument, then for $0 < \epsilon < \min \left\{ \frac{d_1}{2}, \delta_Z \lambda + d_2 \right\}$ we can prove that $||\nabla Z||_{L^2(\Omega) L^\infty_\epsilon (0, t)}$ is uniformly bounded for $t \geq 0$. Now Lemma \ref{lem: nabla Z L2 bounded} holds as a consequence of Proposition \ref{prop: L epsilon infinity norm}.
\end{proof}

\begin{lemma}
    \label{lem: nabla Z Lqk bounded}
    Suppose the solution for equation \eqref{eq: general model} -- \eqref{eq: bc ic} satisfies
    \begin{equation}
        \label{eq: assumption L^qk boundedness of nabla Z}
        ||\nabla Z||_{L^{q_k}} \leq C_{17, k}, \quad \forall t \geq 0
    \end{equation}
    for some $0 \leq k \leq k^*$, then
    \begin{equation}
        \label{eq: nabla Z Lqk+1 bounded}
        ||\nabla Z||_{L^{q_{k+1}}} \leq C_{17, k+1}, \quad \forall t \geq 0.
    \end{equation}
\end{lemma}

\begin{proof}
    Applying \eqref{eq: lemma Lp Lq estimate 1}, \eqref{eq: lemma Lp Lq estimate 2} on \eqref{eq: proof0: nabla Z L2 L epsilon p}, we obtain that
\begin{equation}
     \begin{aligned}
        \label{eq: proof1: nabla Z Lqk+1 evaluation}
        ||\nabla Z||_{L^{q_{k+1}}} 
        \leq & \ c_1 e^{-(\delta_{Z}\lambda + d_2) t } || \nabla  Z_0 ||_{L^{q_{k+1}}} \\
        & + c_1 e_2 \int_0^t \left( 1+ (t-s)^{-\frac{1}{2} - \rho_1 } \right)  e^{-(\delta_{Z}\lambda + d_2) (t-s)} \left\|g(Y,Z,x)\right\|_{L^{q_k}}  \mathrm{d}s \\
        & + c_1 \int_0^t \left( 1+ (t-s)^{-\frac{1}{2} - \rho_1 }\right)  e^{-(\delta_{Z}\lambda + d_2) (t-s)}\left\| \nabla (\chi_2(x)h(Z)\nabla Y) \right\|_{L^{q_k}} \mathrm{d}s,
    \end{aligned}
\end{equation}
where $\rho_1$ is defined in equation \eqref{eq: rho_1}. By the boundedness of $Y$ \eqref{eq: L^inf boundedness of Y nabla X}, $Z$ \eqref{eq: L^inf boundedness of Z nabla Y} and Proposition \ref{prop: comparison of different norms}, the second term can be evaluated by
\begin{equation}
\label{eq: proof2: nabla Z Lqk+1 evaluation}
    \begin{aligned}
        & \quad \ \int_0^t \left( 1+ (t-s)^{-\frac{1}{2} - \rho_1 } \right)  e^{-(\delta_{Z}\lambda + d_2) (t-s)} \left\| g(Y,Z,x) \right\|_{L^{q_k}}  \mathrm{d}s \\
        & \leq \int_0^t \left( 1+ (t-s)^{-\frac{1}{2} - \rho_1 } \right)  e^{-(\delta_{Z}\lambda + d_2) (t-s)} \mu_2 \|Y\|_{L^\infty} \|Z\|_{L^\infty} |\Omega|^\frac{1}{q_k}\mathrm{d}s\\
        & \leq c_2 \mu_2 C_7 C_{14} |\Omega|^\frac{1}{q_k}.
    \end{aligned}
\end{equation}
By Propositions \ref{prop: Young's inequality on L^p_epsilon} and \ref{prop: L epsilon infinity norm}, for $0<\epsilon < \delta_Z\lambda + d_2$, there exist a sufficiently large $p$ and a constant $c_3$, such that
\begin{equation}
\label{eq: proof4: nabla Z Lqk+1 evaluation}
    \begin{aligned}
        & \quad \  \int_0^t \left( 1 + (t-s)^{-\frac{1}{2} - \rho_1}\right) e^{-\left(\delta_{Z} \lambda + d_2\right)(t-s)} \left\| \nabla (\chi_2(x)h(Z)\nabla Y) \right\|_{L^{q_k}}  \mathrm{d}s  \\
        & \leq \sup \limits_{s\in (0,t)} \left\| T_{\frac{1}{2} + \rho_1, \delta_Z \lambda + d_2} \left\| \nabla (\chi_2(x) h(Z)\nabla Y) \right\|_{L^{q_k}} \right\|_{L^\infty_\epsilon (0,s)} \\
        & \leq c_3 \sup\limits_{s\in (0,t)} \left\| \nabla (\chi_2(x)h(Z)\nabla Y) \right \|_{L^{q_k}(\Omega) L^{p}_\epsilon (0,s)}.
    \end{aligned}
\end{equation}
By the estimates of $\left\| \nabla Y \right\|_{L^\infty}$ \eqref{eq: L^inf boundedness of Z nabla Y}, $||\nabla Z||_{L^{q_k}}$ \eqref{eq: assumption L^qk boundedness of nabla Z}, and $||\Delta Y||_{L^{q_k}(\Omega) L^p_\epsilon(0,t)}$ \eqref{eq: Delta Y Lp Lq_eps bounded}, for $0<\epsilon< \min \left\{ \frac{d_1}{2}, \delta_Z\lambda+d_2 \right\}$, there exist constants $c_4$ and $c_{6}$ such that
\begin{equation}
\label{eq: proof5: nabla Z Lqk+1 evaluation}
    \begin{aligned}
        &\quad \ \left\| \nabla (\chi_2(x)h(Z)\nabla Y) \right\|_{L^{q_k}(\Omega) L^p_\epsilon(0,t)} \\
        & \leq c_4 \left\| \nabla \chi_2(x) h(Z)\nabla Y \right\|_{L^{q_k}(\Omega) L^\infty(0,t)} 
        + c_4 \left\| \chi_2(x) h'(Z) \nabla Z \nabla Y \right\|_{L^{q_k}(\Omega) L^\infty(0,t)}
        + \left\| \chi_2(x) h(Z) \Delta Y \right\|_{L^{q_k}(\Omega) L^p_\epsilon(0,t)} \\
        & \leq c_5 \left( ||\nabla Y||_{L^{q_k}(\Omega) L^\infty(0,t)} + ||\nabla Y||_{L^\infty(\Omega) L^\infty(0,t)} ||\nabla Z||_{L^{q_k}(\Omega) L^\infty(0,t)} + ||\Delta Y||_{L^{q_k}(\Omega) L^p_\epsilon(0,t)} \right) \\
        & \leq c_5 \left( |\Omega|^{1/q_k} C_7 + C_{14} C_{17,k} + C_{15} \right).
    \end{aligned}
\end{equation}
It follows from above equations \eqref{eq: proof1: nabla Z Lqk+1 evaluation}-\eqref{eq: proof5: nabla Z Lqk+1 evaluation} that
\begin{equation*}
    \begin{aligned}
        \left\| \nabla Z \right\|_{L^{q_{k+1}}} 
        & \leq c_6 || \nabla Z_{0} ||_{L^{q_{k+1}}} 
        + c_1 c_2 e_2 \mu_2 C_7 C_{14} |\Omega|^{1/q_k}
         + c_1 c_3 c_5 \left( |\Omega|^{1/q_k} C_7 + C_{14} C_{17,k} + C_{15} \right) 
        =: C_{17, k+1}.
    \end{aligned}
\end{equation*}
\end{proof}

Taking $k = k^*$ in the previous lemma, we immediately obtain the following corollary.

\begin{corollary}
    \label{cor: L^inf boundedness of nabla Z}
    The solution for equation \eqref{eq: general model} -- \eqref{eq: bc ic} satisfies
    \begin{equation}
        \label{eq: L^inf boundedness of nabla Z}
        ||\nabla Z||_{L^\infty} \leq C_{18}, \quad \forall t \geq 0.
    \end{equation}
\end{corollary}

\begin{remark}
    \label{rem: W^1, infty}
    By Lemmas \ref{lem: X bounded}, \ref{lem: L^inf boundedness of Z nabla Y} and Corollaries \ref{cor: L^inf boundedness of Y nabla X}, \ref{cor: L^inf boundedness of nabla Z}, $||U(t)||_{W^{1, \infty}}$ is bounded by a constant for all $t$. However, this constant depends on the initial value. If we want to obtain uniform boundedness, we need to guarantee that the initial values are also bounded. By carefully going through the proofs above, the boundedness of the following quantities will guarantee the uniform boundedness of the solution:
    \begin{equation}
        \label{eq: initial value dependence}
        \int_\Omega \frac{|\nabla X_0|^2}{X_0} ,
        \quad \int_\Omega \frac{|\nabla Y_0|^2}{Y_0} ,
        \quad \displaystyle \int_\Omega \frac{|\nabla Z_0|^2}{h(Z_0)} ,
        \quad ||U_0||_{W^{2, \infty}}.
    \end{equation}
\end{remark}

\subsection{Proof of stability}

Now we discuss the stability of the model \eqref{eq: general model} when $\tau_1$ and $\tau_2$ are constant functions. In this case, the function $\Phi$ defined in \eqref{eq: vector Phi} only depends on $U$, and we can find constant equilibrium $U^*$ of \eqref{eq: general model} from $\Phi(U^*) = 0$. 

One can easily check that steady state (0,0,0) is linearly unstable. For other steady states, we propose a generalization of LaSalle's invariance principle.

\begin{proposition}[Generalized LaSalle's invariance principle]
\label{prop: lyapunov to L^infty}
Let $U: [0, \infty] \to W^{1, \infty}(\Omega; \mathbb{R}^3)$ satisfy $\sup_{t\geq 0}$ $\|U(t)\|_{W^{1,\infty}}<\infty$. Suppose there is a function \(\mathcal{L} \in C^1([0, \infty); [0, \infty))\) with \(\mathcal{L}(t) = \mathcal{L}^{(0)}(U(t))\) and \(\frac{\mathrm{d}}{\mathrm{d}t} \mathcal{L}(t) = \mathcal{L}^{(1)}(U(t))\). 

\begin{enumerate}
    \item[(a)] Suppose there is a closed subset $S \subset L^\infty(\Omega; \mathbb{R}^3)$ so that $U(t) \in S$ and $\frac{\mathrm{d}}{\mathrm{d}t} \mathcal{L}(t) \leq 0$ for all $t \geq 0$, then we have $\lim\limits_{t \to \infty} d_\infty (U(t), M) = 0$, where $M$ is the largest invariant set in $E = \left\{ V \in S: \mathcal{L}^{(1)} (V) = 0 \right\}$, and $d_\infty$ is the distance under $L^\infty (\Omega; \mathbb{R}^3)$.
    
    \item[(b)] Let $U^* = (u_1^*, u_2^*, u_3^*) \in L^\infty (\Omega; \mathbb{R}^3)$, and let $S$ be a closed neighbourhood of $U^*$ in $L^\infty (\Omega; \mathbb{R}^3)$. Suppose the set $S$ satisfies the following conditions:
    \begin{itemize}
        \item[(i)] If $U(t) \in S$, then $\frac{\mathrm{d}}{\mathrm{d}t} \mathcal{L}(t)  \leq 0$;
        
        \item[(ii)] There exist $c_1$, $c_2 > 0$ and $p_1$, $p_2$, $p_3 \in [1, \infty)$ so that
        \begin{equation*}
            c_1 d_p (V, U^*) \leq \mathcal{L}^{(0)} (V) \leq c_2 d_p (V, U^*), \ \forall V \in S,
        \end{equation*}
        where 
        \begin{equation*}
            d_p (V, U^*) := \sum_{i = 1}^3 \|v_i - u_i^*\|_{L^{p_i}}^{p_i}, \ V = (v_1, v_2, v_3).
        \end{equation*}
    \end{itemize}
    
    Then there exists $\epsilon > 0$, so that if $d_p (U_0, U^*) < \epsilon$, then $\lim_{t \to \infty} d_\infty (U(t), M) = 0$, where $M$ is the same as in part (a). Here $\epsilon$ depends on $S$, $c_1$, $c_2$ and $K := \sup_{t\geq 0} \|U(t)\|_{W^{1,\infty}}$. 
\end{enumerate}
\end{proposition}

\begin{proof}
    \textit{(a)} Since $\|U(t)\|_{W^{1, \infty}}$ is uniformly bounded for all $t \geq 0$, the Sobolev Embedding Theorem implies that $\{ U(t): t \geq 0 \}$ is a (pre)compact subset of $L^\infty(\Omega; \mathbb{R}^3)$. We can verify that the proof of LaSalle's invariance principle  (Theorem 1, \cite{lasalle1960some}) is valid under the metric of $L^\infty(\Omega; \mathbb{R}^3)$. Now part (a) follows immediately.

    \textit{(b)} Denote $U(t) = (u_1(t), u_2(t), u_3(t))$. By Gagliardo-Nirenberg inequality, there exists $c_3 > 0$ so that
    \begin{equation*}
        ||u_i(t) - u_i^*||_{L^\infty} 
        \leq c_3 ||u_i(t) - u_i^*||_{W^{1, \infty}}^{\frac{n}{n + p_i}} ||u_i(t) - u_i^*||_{L^{p_i}}^{\frac{p_i}{n + p_i}}
    \end{equation*}
    for all $t \geq 0$. Denote $c_4 = \max\limits_{1 \leq i \leq 3} c_3 (K + ||u_i^*||_{W^{1, \infty}})^{\frac{n}{n + p_i}}$, then we have
    \begin{equation}
        \label{eq: stability GN inequality}
        ||u_i(t) - u_i^*||_{L^\infty} 
        \leq c_4 ||u_i(t) - u_i^*||_{L^{p_i}}^{\frac{p_i}{n + p_i}}
    \end{equation}
    for $1 \leq i \leq 3$ and $t \geq 0$.
    
    Choose $\epsilon_0 > 0$ so that $B(U^*, \epsilon_0) \subset S$. Let $U_0$ be any initial value so that 
    \begin{align*}
        d_p (U_0, U^*) < \min \left\{ 1, \left( \frac{\epsilon_0}{c_4} \right)^{p'}, \, \frac{c_1}{c_2 c_4} \epsilon_0^{p'} \right\} 
        =: \epsilon,
    \end{align*}
    where $p' = \max\limits_{1 \leq i \leq 3} (n + p_i)$. Then by \eqref{eq: stability GN inequality} and $\epsilon \leq \min \left\{ 1, \left( \frac{\epsilon_0}{c_4} \right)^{p'} \right\}$, we have
    \begin{equation*}
        \| U_0 - U^* \|_{L^\infty} < c_4 \max\limits_{1 \leq i \leq 3} \epsilon^{1 / (n + p_i)}
        \leq c_4 \epsilon^{1 / p'} \leq \epsilon_0.
    \end{equation*}
    It follows that $U_0 \in S$. By condition (iii) and $\epsilon \leq \frac{c_1}{c_2 c_4} \epsilon_0^{p'}$, we further get
    \begin{equation*}
        \mathcal{L}(0) \leq c_2 d_p (U_0, S_0)
        < \frac{c_1}{c_4} \epsilon_0^{p'}.
    \end{equation*}
    
    Define
    \begin{equation}
        \label{eq: stability t^*}
        t^* = \sup \big\{ T \in [0, \infty): \| U(t) - U^* \|_{L^\infty} \leq \epsilon_0, \ \forall t \in [0, T] \big\} \in [0, \infty],
    \end{equation}
    then $\frac{\mathrm{d}}{\mathrm{d}t} \mathcal{L}(t) \leq 0$ for all $t \in [0, t^*]$ according to condition (i). Suppose $t^* < \infty$, then we deduce from condition (iii) and inequality \eqref{eq: stability GN inequality} that
    \begin{align}
        \sum_{i = 1}^3 \|u_i(t) - u_i^*\|_{L^\infty}^{n + p_i} 
        \leq c_4 d_p (U(t), U^*) 
        \leq \frac{c_4}{c_1} \mathcal{L}(t)
        \leq \frac{c_4}{c_1} \mathcal{L}(0)
        < \epsilon_0^{p'}
        = \min_{1 \leq i \leq 3} \epsilon_0^{n + p_i}.
        \label{eq: stability comparison}
    \end{align}
    It follows that $\| U(t) - U^* \|_{L^\infty} < \epsilon_0$, which contradicts with the definition \eqref{eq: stability t^*} of $t^*$. Therefore, $t^* = \infty$, that is, $\frac{\mathrm{d}}{\mathrm{d}t} \mathcal{L}(t) \leq 0$ holds for any $t \geq 0$. Finally, we apply part (a) to deduce $\lim\limits_{t \to \infty} d_\infty (U(t), M) = 0$.
\end{proof}

\begin{remark}
    \label{rem: LaSalle}
    Proposition \ref{prop: lyapunov to L^infty} provides a useful tool for addressing the stability of solutions for a fixed initial value. For equations \eqref{eq: general model} -- \eqref{eq: bc ic}, the boundedness of \(\|U(t)\|_{W^{1,\infty}}\) is given in Remark \ref{rem: W^1, infty}. Global stability follows if \(S = L^\infty(\Omega; \mathbb{R}^3)\) in part (a). In the case of local stability, the proximity of the initial value to the equilibrium depends on the bound of \(\|U(t)\|_{W^{1,\infty}}\), which, in turn, is influenced by the terms in \eqref{eq: initial value dependence}. These terms are again related to the initial value, making the boundedness of \eqref{eq: initial value dependence} essential when discussing local stability.
\end{remark}

\begin{lemma}[Stability of prey-only steady state] 
    \label{lem: stability for prey only}
    Let \( U(t) = (X(t), Y(t), Z(t))\) be the solution of \eqref{eq: general model} -- \eqref{eq: bc ic}, and let $U^* = (K, 0, 0)$ be the equilibrium.
    \begin{enumerate}
    \item[(a)]  If \(d_1 > K \left(\mu_1 e_1 - d_1 \tau_1\right)\), then for any $R > 0$, there exists $\epsilon = \epsilon(R) > 0$, if the terms in \eqref{eq: initial value dependence} are bounded by $R$ and \(\|X_0 - K\|_{L^2} + \|Y_0\|_{L^1} + \|Z_0 \|_{L^1} < \epsilon\), then we have
    \begin{equation}
        \label{eq: local stability for prey-only} 
        \lim\limits_{t \to \infty} \left(\|X(t) - K\|_{L^\infty} + \|Y(t)\|_{L^\infty} + \|Z(t)\|_{L^\infty} \right) = 0.
    \end{equation} 
    
    \item[(b)] If \(d_1 \geq K \mu_1 e_1 \), then \eqref{eq: local stability for prey-only} holds for all $(X_0, Y_0, Z_0)$.
    \item[(c)]  If 
    \(d_1 \leq K \left(\mu_1 e_1 - d_1 \tau_1\right)\),
    then for any $\epsilon > 0$, there exists a solution \(U(t)\) with initial value $\| U_0 - U^* \|_{L^\infty} < \epsilon$, such that
    \begin{equation}
        \label{eq: unstable stability for prey-only} 
        \limsup\limits_{t \to \infty} \left(\|X(t) - K\|_{L^\infty} + \|Y(t)\|_{L^\infty} + \|Z(t)\|_{L^\infty} \right) > 0.
    \end{equation}
\end{enumerate}
    
\end{lemma}

\begin{proof}
Define
\[
\mathcal{L}_1(t) := \mathcal{L}_1 (X, Y, Z) = \int_\Omega \left(X - K - K \ln{\frac{X}{K}}\right) + \frac{1}{e_1} Y + \frac{1}{e_1e_2} Z,
\]
and set $\mathcal{L}_1^{(0)}$ and $\mathcal{L}_1^{(1)}$ the same way as in Proposition \ref{prop: lyapunov to L^infty}. Using Taylor expansion, there exists \(\eta\) between \(X\) and \(K\) such that
\[
X - K - K \ln{\frac{X}{K}} = \frac{K}{2 \eta^2}(X-K)^2 \geq 0,
\]
with equality if and only if \(X = K\). Therefore, \(\mathcal{L}_1(K,0,0) = 0\) and \(\mathcal{L}_1(X,Y,Z) > 0\) for any \((X,Y,Z) \neq (K,0,0)\) since \(Y \geq 0\) and \(Z \geq 0\). 

From system \eqref{eq: general model}, we have:
\begin{equation}
    \label{eq: proof of prey-only}
    \begin{aligned}
        \frac{\mathrm{d}}{\mathrm{d} t}\mathcal{L}_1(t) &=  \int_\Omega \frac{X-K}{X} X_t + \frac{1}{e_1} \int_\Omega Y_t + \frac{1}{e_1 e_2} \int_\Omega Z_t\\
        & = - \delta_X \int_\Omega \frac{|\nabla X|^2}{X^2} - \frac{r}{K} \int_\Omega (X-K)^2 + \int_\Omega \left( \frac{K \mu_1}{1 + \tau_1 X} - \frac{d_1}{e_1} \right) Y - \frac{d_2}{e_1 e_2} \int_\Omega Z,
     \end{aligned}   
    \end{equation}

\textit{(a)} Suppose \(d_1 > K \left(\mu_1 e_1 - d_1 \tau_1\right)\), then $\frac{K \mu_1}{1 + \tau_1 X} - \frac{d_1}{e_1} < 0$ for $X$ in a neighbourhood of $K$. It follows that \( \mathcal{L}_1^{(1)} \leq 0 \) holds for all \( U\) in a closed neighbourhood $S$ of \((K,0,0)\) (in the norm of $L^\infty (\Omega;\mathbb{R}^3)$). Moreover, there exist $c_1$, $c_2>0$, such that
\begin{equation*}
     c_1 d_p (V, U^*) \leq \mathcal{L}_1^{(0)} (V) \leq c_2 d_p (V, U^*), \ \forall V \in S,
\end{equation*}
where
\begin{equation*}
    d_p (V, U^*) = ||X - K||_{L^2}^2 + ||Y||_{L^1} + ||Z||_{L^1}, \ V = (X, Y, Z).
\end{equation*}
By equation \eqref{eq: proof of prey-only}, we have
\begin{equation*}
    E := \{V \in L^\infty(\Omega;\mathbb{R}^3): \mathcal{L}_1^{(1)} (V) = 0\} = \{ (K,0,0)\}.
\end{equation*} 
Then part (a) follows from Proposition \ref{prop: lyapunov to L^infty}(b) and Remark \ref{rem: LaSalle}. 

\textit{(b)} If $d_1 \geq K e_1 \mu_1$, then $\frac{\mathrm{d}}{\mathrm{d}t} \mathcal{L}_1(t) \leq 0$, with equality holds only if $X(t) = K$ and $Z(t) = 0$. The largest invariant set of system \eqref{eq: general model} in $\{ (K, Y, 0): Y \in L^\infty(\Omega) \}$ is $(K, 0, 0)$. Therefore, by Proposition \ref{prop: lyapunov to L^infty}(a), equation \eqref{eq: local stability for prey-only} holds.

\textit{(c)} \(d_1 \leq K \left(\mu_1 e_1 - d_1 \tau_1\right)\) implies $\frac{K \mu_1}{1 + \tau_1 X} - \frac{d_1}{e_1} > 0$ for all $0 < X < K$. For any $\epsilon$, consider a spatially homogeneous solution \(U = (X, 0, 0)\) with $X_0 < K$ and $||X_0 - K||_{L^\infty} < \epsilon$, then \(\frac{\mathrm{d}}{\mathrm{d}t} \mathcal{L}_2(t) > 0\) for all \(t\). Combined with the fact that \(\mathcal{L}_1(t) > 0\) for all \( U \neq U^*\), then $\limsup\limits_{t \to \infty} \mathcal{L}_1(t) \geq 0$, which implies \eqref{eq: unstable stability for prey-only}.
\end{proof}

Now we consider the stability for the semi-coexistence state $U^* = (X^*,Y^*,0)$, where $X^* = \frac{d_1}{ \mu_1e_1 - \tau_1d_1} $, $Y^* = \frac{e_1 r}{d_1} X^*(1-\frac{X^*}{K})$.

\begin{lemma}[Stability of semi-coexistence steady state]
\label{lem: stability of semicoexistence}
Let \( U(t) = (X(t), Y(t), Z(t))\) be the solution of \eqref{eq: general model} -- \eqref{eq: bc ic}, and $\tau_1$, $\tau_2$ are constant functions. Assume \( Y^* < \frac{d_2}{e_2 \mu_2} \) and 
\begin{equation}
\label{eq: stability for semitrivial}
    \left\|\frac{\chi_1(x)h(Y) X}{Y}\right\|_{L^\infty}^2 < \frac{4 e_1 \delta_X \delta_Y X^*}{\left(1+\tau_1 X^*\right) Y^*}.
\end{equation}
\begin{enumerate}
    \item[(a)]  If \( X^* > \frac{1}{2}\left(K - \frac{1}{\tau_1}\right) \), then for any $R > 0$, there exists $\epsilon = \epsilon(R) > 0$, if the terms in \eqref{eq: initial value dependence} are bounded by $R$ and \(\|X_0 - X^*\|_{L^2} + \|Y_0 - Y^*\|_{L^2} + \|Z_0 \|_{L^1} < \epsilon\), then we have
    \begin{equation}
        \label{eq: local stability for semi-coexistence} 
        \lim\limits_{t \to \infty} \left(\|X(t) - X^*\|_{L^\infty} + \|Y(t) - Y^*\|_{L^\infty} + \|Z(t)\|_{L^\infty} \right) = 0.
    \end{equation} 
    
    \item[(b)] If \( X^* > K - \frac{1}{\tau_1} \), then \eqref{eq: local stability for semi-coexistence} holds for all $(X_0, Y_0, Z_0)$.
    \item[(c)]  If 
    \( X^* \leq \frac{1}{2}\left(K - \frac{1}{\tau_1}\right)\),
    then for any $\epsilon > 0$, there exists a solution \(U(t)\) with initial value $\| U_0 - U^* \|_{L^\infty} < \epsilon$, such that
    \begin{equation}
        \label{eq: unstable stability for semi-coexistence} 
        \limsup\limits_{t \to \infty} \left(\|X(t) - X^*\|_{L^\infty} + \|Y(t) - Y^*\|_{L^\infty} + \|Z(t)\|_{L^\infty} \right) > 0.
    \end{equation}
\end{enumerate}
\end{lemma}

\begin{proof}
Let
\begin{equation*}
\mathcal{L}_2(t) = \int_\Omega \frac{X^*}{1+\tau_1 X^*}\left(\frac{X}{X^*} - 1 - \ln{\frac{X}{X^*}}\right) + \frac{1}{e_1} Y^*\left(\frac{Y}{Y^*} - 1 - \ln{\frac{Y}{Y^*}}\right) + \frac{1}{e_1 e_2} Z = \int_\Omega \mathcal{L}_{2}^{(0)}(U).
\end{equation*}
It's easy to check that \(\mathcal{L}_2^{(0)} (X^*, Y^*,0) = 0\) and \(\mathcal{L}_2^{(0)} (X,Y,Z) > 0\) for all \((X,Y,Z) \neq (X^*,Y^*,0)\). 
Next, we compute the derivative of \(\mathcal{L}_2(t)\):
\begin{equation*}
\frac{\mathrm{d}\mathcal{L}_2}{\mathrm{d}t} = \int_\Omega \frac{X^*}{1 + \tau_1 X^*}\left(\frac{1}{X^*} - \frac{1}{X}\right) X_t + \frac{1}{e_1} Y^*\left(\frac{1}{Y^*} - \frac{1}{Y}\right) Y_t + \frac{1}{e_1 e_2} Z_t .
\end{equation*}
Using the identities
\begin{equation*}
\frac{X^*}{1 + \tau_1 X^*}\left(\frac{1}{X^*} - \frac{1}{X}\right) = \left(\frac{X}{1 + \tau_1 X} - \frac{X^*}{1 + \tau_1 X^*}\right)\left(\frac{1}{X} + \tau_1\right),
\end{equation*}
\begin{equation*}
Y^*\left(\frac{1}{Y^*} - \frac{1}{Y}\right) = \frac{Y - Y^*}{Y},
\end{equation*}
and $(X^*)_t = (Y^*)_t = 0$, we rewrite the derivative as 
\begin{equation*}
\begin{aligned}
\frac{\mathrm{d}\mathcal{L}_2}{\mathrm{d}t} &= \int_\Omega \left(\frac{X}{1 + \tau_1 X} - \frac{X^*}{1 + \tau_1 X^*}\right)\left[\left(\frac{1}{X} + \tau_1\right) X_t - \left(\frac{1}{X^*} + \tau_1\right) \left({X}^*\right)_t\right] \\
&\quad + \frac{1}{e_1} \int_\Omega \left(\frac{Y - Y^*}{Y} Y_t - \frac{Y - Y^*}{Y^*} \left({Y}^*\right)_t\right) + \frac{1}{e_1 e_2} \int_\Omega Z_t.
\end{aligned}
\end{equation*}
From the system \eqref{eq: general model}, we obtain
\begin{equation*}
\begin{aligned}
\frac{\mathrm{d}\mathcal{L}_2}{\mathrm{d}t} & = r \int_\Omega \left(\frac{X}{1 + \tau_1 X} - \frac{X^*}{1 + \tau_1 X^*}\right)\left[\left(1 + \tau_1 X\right)\left(1 - \frac{X}{K}\right) - \left(1 + \tau_1 X^*\right)\left(1 - \frac{X^*}{K}\right)\right] \\
&\quad + \frac{1}{e_1} \left(\frac{\mu_2 Y^*}{1 + \tau_2 Y} - \frac{d_2}{e_2}\right) \int_\Omega Z - \frac{\delta_Y Y^*}{e_1} \int_\Omega \frac{|\nabla Y|^2}{Y^2} - \frac{\delta_X X^*}{1 + \tau_1 X^*} \int_\Omega \frac{|\nabla X|^2}{X^2} + \frac{Y^*}{e_1} \int_\Omega \frac{\chi_1(x) h(Y) \nabla X \nabla Y}{Y^2}\\
& =: \int_\Omega \mathcal{L}_{2}^{(1)}(U).
\end{aligned}
\end{equation*}
Since \( Y^* < \frac{d_2}{e_2 \mu_2} \), we have $\frac{1}{e_1} \left(\frac{\mu_2 Y^*}{1 + \tau_2 Y} - \frac{d_2}{e_2} \right) \int_\Omega Z < 0$.
Let $\widetilde{U}_1 = \left(\frac{\nabla X}{X}, \frac{\nabla Y}{Y}\right)$ and denote the diffusion terms by
\begin{equation*}
    I = - \frac{\delta_Y Y^*}{e_1} \int_\Omega \frac{|\nabla Y|^2}{Y^2} - \frac{\delta_X X^*}{1 + \tau_1 X^*} \int_\Omega \frac{|\nabla X|^2}{X^2} + \frac{Y^*}{e_1} \int_\Omega \frac{\chi_1(x) h(Y) \nabla X \nabla Y}{Y^2}.
\end{equation*}
This can be rewritten as $I = - \int_\Omega \widetilde{U}_1 A \widetilde{U}_1^T$, where
\begin{equation*}
    A = \begin{pmatrix}
        \dfrac{\delta_X X^*}{1 + \tau_1 X^*} &  -\dfrac{Y^*}{2e_1} \dfrac{\chi_1(x) h(Y)X}{Y} \\
        -\dfrac{Y^*}{2e_1} \dfrac{\chi_1(x) h(Y)X}{Y}    & \dfrac{\delta_Y Y^*}{e_1}
    \end{pmatrix}.
\end{equation*}
One can verify that if $\eqref{eq: stability for semitrivial}$ holds, the matrix $A$ is positive definite.  Hence there exists a positive constant $\alpha_1>0$ such that
\begin{equation*}
    I \leq - \alpha_1 \int_\Omega \left(\frac{|\nabla X|^2}{X^2} + \frac{|\nabla Y|^2}{Y^2}\right).
\end{equation*}

On the other hand, define
\begin{equation*}
J(X,X^*) = \left(f_1(X) - f_1(X^*)\right)\cdot \left(f_2(X) - f_2(X^*)\right),
\end{equation*}
where \( f_1(z) = \frac{z}{1 + \tau_1 z} \) and \( f_2(z) = \left(1 + \tau_1 z\right)\left(1 - \frac{z}{K}\right) \). Given that \( Y^* \leq \frac{d_2}{e_2 \mu_2} \), we have \( \frac{Y^*}{1 + \tau_2 Y^*} \mu_2 - \frac{d_2}{e_2} <0 \). It follows that \( \mathcal{L}_2^{(1)} \leq 0 \) if \( J(X,X^*) \leq 0 \). To ensure \( J(X,X^*) \leq 0 \), it is necessary that \( f_1(X) - f_1(X^*) \) and \( f_2(X) - f_2(X^*) \) have opposite signs. Since \( f_1(X) \) is monotonically increasing, \( f_1(X) - f_1(X^*) \) and \( X - X^* \) share the same sign.

\textit{(a)} Suppose \( X^* > \frac{1}{2} \left(K - \frac{1}{\tau_1}\right) \), then \( \mathcal{L}_2^{(1)} \leq 0 \) holds for all \( U\) in a closed neighbourhood $S$ of \(U^* = (X^*,Y^*,0)\) (in the norm of $L^\infty (\Omega;\mathbb{R}^3)$). By a similar argument as in Lemma \ref{lem: stability for prey only}, we can prove from Taylor expansion that there exist $c_1$, $c_2>0$, such that
\begin{equation*}
     c_1 d_p (V, U^*) \leq \mathcal{L}_2^{(0)} (V) \leq c_2 d_p (V, U^*), \ \forall V \in S,
\end{equation*}
where
\begin{equation*}
    d_p (V, U^*) = ||X - X^*||_{L^2}^2 + ||Y - Y^*||_{L^2}^2 + ||Z||_{L^1}, \ V = (X, Y, Z).
\end{equation*}
By equations \eqref{eq: general model}, we have
\begin{equation*}
    E := \{V \in L^\infty(\Omega;\mathbb{R}^3): \mathcal{L}_2^{(1)} (V) = 0\} = \{ (X^*,Y,0) \in L^\infty(\Omega;\mathbb{R}^3): Y \text{ is a nonnegative constant}\}.
\end{equation*} 
The maximal invariant set in $E$ is $M = \{(X^*,Y^*,0)\}$.
Then part (a) follows from Proposition \ref{prop: lyapunov to L^infty}(b) and Remark \ref{rem: LaSalle}. 

\textit{(b)} If \( X^* > K - \frac{1}{\tau_1} \), then \( \frac{\mathrm{d}}{\mathrm{d}t} \mathcal{L}_2(t) \leq 0 \) for all \( t \geq 0 \). Part (b) then follows from Proposition \ref{prop: lyapunov to L^infty}(a).

\textit{(c)} When \( X^* \leq \frac{1}{2}\left(K - \frac{1}{\tau_1}\right) \), we assert that \((c)\) holds. Suppose, for contradiction, that there exists \(\epsilon > 0\) such that for all solutions \(U(t)\) with initial values satisfying \(\|U_0 - U^*\|_{L^\infty} < \epsilon\), equation \eqref{eq: local stability for semi-coexistence} holds. This implies \(\lim_{t \to \infty} \mathcal{L}_2(t) = 0\).

Now, consider a spatially homogeneous solution \((X, Y, 0) \neq (X^*, Y^*, 0)\) with an initial condition in this neighborhood. The Lyapunov function then satisfies
\[
\frac{\mathrm{d}\mathcal{L}_2}{\mathrm{d}t}(t) = r |\Omega| J(X, X^*).
\]
When \( X^* < \frac{1}{2}\left(K - \frac{1}{\tau_1}\right) \), we have \(J(X, X^*) > 0\) in a punctured neighborhood of \(X^*\), implying \(\frac{\mathrm{d}}{\mathrm{d}t} \mathcal{L}_2(t) > 0\) for all \(t\). When \( X^* = \frac{1}{2}\left(K - \frac{1}{\tau_1}\right) \), \(J(X, X^*) > 0\) for all \(X < X^*\). Thus, if \(X_0 < X^*\), then \(\frac{\mathrm{d}}{\mathrm{d}t} \mathcal{L}_2(t) > 0\) for any \(t \geq 0\). Combined with the fact that \(\mathcal{L}_2(t) > 0\) for all \( U \neq U^*\), this contradicts the assumption that \(\lim_{t \to \infty} \mathcal{L}_2(t) = 0\). Therefore, \((c)\) holds.
\end{proof}

\begin{lemma}[Stability of coexistence steady state]
    \label{lem: stability of coexistence}
    Let \( U(t) = (X(t), Y(t), Z(t))\) be the solution of \eqref{eq: general model} -- \eqref{eq: bc ic}, and $\tau_1$ and $\tau_2$ are constant functions.
    \begin{itemize}
        \item[(a)] if 
        \begin{equation}
            \label{eq: A positive definite}
            \frac{1}{\delta_X} \chi_{1M}^2 h^2(Y^*) + \frac{1}{\delta_Z} \chi_{2M}^2 h^2(Z^*) < 4 \delta_Y
        \end{equation}
        and $D\Phi (U^*)$ is negative definite, then for any $R > 0$, there exists $\epsilon = \epsilon(R) > 0$, so that if the terms in \eqref{eq: initial value dependence} are bounded by $R$ and $||U_0 - U^*||_{L^2} < \epsilon$, then $\lim_{t \to \infty} ||U(t) - U^*||_{L^\infty} = 0$.
        \item[(b)] if $D\Phi (U^*)$ has an eigenvalue with positive real part, then there exists $R > 0$, so that for any $\epsilon > 0$, we can choose an initial value $U_0$ with $||U_0 - U^*||_{L^\infty} < \epsilon$, but $||U(\tau) - U^*||_{L^\infty} > R$ for some $\tau > 0$.
    \end{itemize}
\end{lemma}
\begin{proof}
    \textit{(a)} Define a Lyapunov function by
    \begin{equation}
        \label{eq: stability proof 1}
        \mathcal{L}_3(t) = \frac{1}{2} ||U(t) - U^*||_{L^2}^2.
    \end{equation}
    Since $U^*$ is a constant function and $\Phi(U^*) = 0$, we have
    \begin{equation}
    \label{eq: stability proof 2}
    \begin{aligned}
        \frac{\mathrm d}{\mathrm d t} \mathcal{L}_3(t)
        &= \int_\Omega (U - U^*)^T \big( \nabla (\mathbb A(x, U) \nabla U) + \Phi(U) \big) \, \mathrm d x \\
        &= \int_\Omega (U - U^*)^T \nabla (\mathbb A(x, U) \nabla U) \, \mathrm d x + \int_\Omega (U - U^*)^T (\Phi(U) - \Phi(U^*)) \, \mathrm d x \\
        &= -\int_\Omega \nabla U^T \mathbb A(x, U) \nabla U \, \mathrm d x + \int_\Omega (U - U^*)^T (\Phi(U) - \Phi(U^*)) \, \mathrm d x.
    \end{aligned}
    \end{equation}
    Denote
    \begin{equation*}
        \mathbb{B}(x) := \mathbb{A} (x, U^*) + \mathbb{A} (x, U^*)^T = 
        \begin{pmatrix}
            2\delta_X & -\chi_1(x) h(Y^*) & 0 \\
            -\chi_1(x) h(Y^*) & 2\delta_Y & -\chi_2(x) h(Z^*) \\
            0 & -\chi_2(x) h(Z^*) & 2\delta_Z
        \end{pmatrix}.
    \end{equation*}
    We can check that $\mathbb{B}(x)$ has positive eigenvalues if and only if
    \begin{equation}
        \frac{1}{\delta_X} \chi_1^2 (x) h^2(Y^*) + \frac{1}{\delta_Z} \chi_2^2 (x) h^2(Z^*) < 4 \delta_Y.
    \end{equation}
    Therefore, \eqref{eq: A positive definite} implies the matrix $\mathbb{A} (x, U^*)$ is positive definite for all $x \in \overline{\Omega}$. It follows that, there exists $\epsilon_1 > 0$, such that for all $z \in \mathbb{R}^3$ with $|z - U^*|_\infty \leq \epsilon_1$, the matrix $\mathbb{A} (x, z)$ is positive definite. Here $|\cdot|_p$ is the $l^p$ norm of vectors. Besides, we denote the largest eigenvalue of $\big( D\Phi (U^*) + D\Phi (U^*)^T \big) / 2$ to be $-\lambda_0 < 0$. We can also choose $\epsilon_2 > 0$ such that for all $z \in \mathbb{R}^3$ with $|z - U^*|_\infty \leq \epsilon_2$, we have
    \begin{equation*}
        \big| \Phi(z) - \Phi(U^*) - D\Phi (U^*) (z - U^*) \big| \leq \frac{1}{2} \lambda_0 |z - U^*|_2.
    \end{equation*}
    It follows that
    \begin{equation}
    \label{eq: stability proof 4}
    \begin{aligned}
        &\quad \ (z - U^*)^T (\Phi(z) - \Phi(U^*))\\
        &= (z - U^*)^T \big( \Phi(z) - \Phi(U^*) - D\Phi (U^*) (z - U^*) \big) + (z - U^*)^T D\Phi (U^*) (z - U^*) \\
        & \leq -\frac{1}{2} \lambda_0 |z - U^*|_2^2.
    \end{aligned}
    \end{equation}
    Combining \eqref{eq: stability proof 1}-\eqref{eq: stability proof 4}, we obtain
    \begin{equation}
        \label{eq: stability proof 5}
        \frac{\mathrm d}{\mathrm d t} \mathcal{L}_3(t)
        \leq -\lambda_0 \mathcal{L}_3(t)
    \end{equation}
    as long as $||U(t) - U^*||_{L^\infty} \leq \min \{ \epsilon_1, \epsilon_2 \}$. Now we can prove part (a) by applying Proposition \ref{prop: lyapunov to L^infty} and Remark \ref{rem: LaSalle} with $S = \{V \in L^\infty(\Omega; \mathbb{R}^3): ||V - U^*||_{L^\infty} \leq \min \{ \epsilon_1, \epsilon_2 \} \}$ and $E = M = \{U^*\}$.

    \textit{(b)} We consider any spatially homogeneous solution \( U \) with an initial value \( U_0 \). In this case, the equation \eqref{eq: general model} reduces to the ODE \( \frac{\mathrm{d} U}{\mathrm{d} t} = \Phi(U) \). 
    If $D\Phi (U^*)$ has an eigenvalue with positive real part, then $U^*$ is an unstable equilibrium regarding the ODE, that is, there exists $R > 0$ so that for any $\epsilon > 0$, we can choose an initial value $U_0$ with $|U_0 - U^*|_\infty < \epsilon$, but $|U(\tau) - U^*|_\infty > R$ for some $\tau > 0$. Now we see that the definition of instability as described in item (b) is satisfied.
\end{proof}

\begin{proof}[Proof of Theorem \ref{th: stability}]
Theorem \ref{th: stability} follows directly from Lemmas \ref{lem: stability for prey only}, \ref{lem: stability of semicoexistence}, and \ref{lem: stability of coexistence}.
\end{proof}
\section*{Acknowledgments}
    The research of Hao Wang was partially supported by the Natural Sciences and Engineering Research Council of Canada (Individual Discovery Grant RGPIN-2020-03911 and Discovery Accelerator Supplement Award RGPAS-2020-00090) and the CRC program (Tier 1 Canada Research Chair Award).

\bibliographystyle{plainnat}
\bibliography{main}

\end{document}